%% file: main_arxiv.tex
\documentclass[11pt]{article}

\usepackage[margin=1in]{geometry}

\input{preamble_shared.tex}

\usepackage{amsthm}

\newtheorem{proposition}{Proposition}

\usepackage[
  backend=biber,
  style=numeric-comp,
  sorting=none,
  defernumbers=true
]{biblatex}

\DeclareFieldFormat{urldate}{}
\renewbibmacro*{urldate}{}
\title{Exact Hull Reformulation for Quadratically Constrained
Generalized Disjunctive Programs}

\author{
Sergey Gusev \\
Davidson School of Chemical Engineering \\
Purdue University
\and
David E. Bernal Neira \\
Davidson School of Chemical Engineering \\
Purdue University
}

\date{}

\begin{document}

\maketitle

\begin{abstract}
\input{abstract_text.tex}
\end{abstract}

\noindent\textbf{Keywords:} \input{keywords.tex}

\input{paper_content.tex}

\printbibliography

\end{document}

%% file: preamble_shared.tex
\usepackage{graphicx}
\usepackage{amsmath, amssymb}
\usepackage{booktabs}
\usepackage{hyperref}
\usepackage{subcaption}
\usepackage[title]{appendix}
\usepackage{url}

\DeclareMathOperator*{\bigveebar}{\underline{\bigvee}}

%% file: abstract_text.tex
Generalized Disjunctive Programming (GDP) provides a natural framework for optimization models that combine logical decisions with nonlinear constraints.
The Hull Reformulation (HR) is attractive because it yields tight continuous relaxations, but for nonlinear disjunctive constraints, it is commonly implemented using an $\varepsilon$-approximation of the closure of the perspective function.
This approximation introduces fractional expressions, enlarges the relaxation for any $\varepsilon>0$, can cause numerical instability, and may hinder solver convexity recognition.

This paper develops a framework for constructing exact hull reformulations for GDPs with quadratic disjunctive constraints that avoids $\varepsilon$-approximation while preserving solver-friendly structure.
For general (possibly non-convex) quadratic constraints, we derive a General Exact Hull Reformulation (GEHR) that eliminates perspective division and preserves quadratic degree, and we prove that it is equivalent to the standard closed-perspective hull in lifted space.
For convex quadratic constraints ($Q\succeq 0$), we propose using a Conic Exact Hull Reformulation (CEHR), re-derived in this work, that represents the perspective term with rotated second-order cone constraints, enabling conic-capable solvers to certify and exploit convexity directly.

We implement both reformulations in \texttt{Pyomo.GDP} and evaluate them on random convex and non-convex instances, a CSTR network benchmark, $k$-means clustering, and constrained layout problems using \texttt{Gurobi}, \texttt{SCIP}, and \texttt{BARON}.
Across these benchmarks, the proposed exact formulations reduce numerical failures and often improve runtime relative to the $\varepsilon$-approximation.
In particular, CEHR is consistently the most reliable and typically the fastest hull reformulation on convex benchmarks, while GEHR performs best on non-convex benchmarks, improving both robustness and overall performance compared to  $\varepsilon$-approximation.

%% file: keywords.tex
Generalized Disjunctive Programming,
Exact Hull Reformulation,
Perspective Function,
Quadratic Constraints,
Mixed-Integer Nonlinear Programming,
Non-convex Generalized Disjunctive Programming

%% file: paper_content.tex
\section{Introduction}

Problems involving continuous and discrete variables, including those that require making discrete decisions, can generally be represented in the form of Mixed-Integer Nonlinear Programs (MINLPs).
However, formulating problems involving a large number of discrete decisions using traditional MINLP frameworks can become overly complex and counterintuitive.
For instance, common tasks in Process Systems Engineering (PSE), such as selecting process units or determining process configurations, often result in cumbersome algebraic constraints and numerous binary variables.
When modeled in the classical form of MINLP, such problems can be challenging for practitioners to implement and interpret \cite{mencarelliReviewSuperstructureOptimization2020, grossmannSystematicModelingDiscretecontinuous2013, kronqvist50YearsMixedinteger2025}.

To address these limitations, Generalized Disjunctive Programming (GDP) has emerged as a higher-level modeling paradigm that more naturally integrates logical decision-making with algebraic constraints.
Introduced by Raman and Grossmann~\cite{ramanModellingComputationalTechniques1994}, GDP extends Balas's disjunctive programming framework \cite{balasDisjunctiveProgramming2018} to include nonlinear constraints and logical relations \cite{chenModernModelingParadigms2019, mencarelliReviewSuperstructureOptimization2020, grossmannSystematicModelingDiscretecontinuous2013}.
In GDP, discrete alternatives, such as choosing among different equipment configurations or operational modes, are explicitly represented through Boolean variables and logical disjunctions (OR conditions) rather than through algebraically encoded binary variables.
This approach enables clearer and more intuitive formulations compared to conventional MINLP models, making it much easier to formulate the model and apply the framework to practical challenges in real-world scenarios \cite{grossmannGeneralizedDisjunctiveProgramming2012}.

A GDP problem can be formally expressed as follows \cite{ramanModellingComputationalTechniques1994,chenModernModelingParadigms2019}:

\begingroup
\begin{equation}
\begin{aligned}
    \min_{\mathbf{x}, \mathbf{Y}} \quad 
    & f(\mathbf{x}) \\
    \text{s.t.} \quad 
    & g(\mathbf{x}) \leq \mathbf{0}, \\
    & \bigvee_{i \in D_k} 
        \begin{bmatrix} 
            Y_{ik} \\
            h_{ik}(\mathbf{x}) \leq \mathbf{0} 
        \end{bmatrix}, \quad \forall\, k \in K, \\
    & \underset{i \in D_k}{\underline{\bigvee}} Y_{ik}, \quad \forall\, k \in K, \\
    & \Omega(\mathbf{Y}) = \text{True}, \\
    & \mathbf x^{\ell} \leq \mathbf{x} \leq \mathbf{x}^u, \\
    & \mathbf{x} \in \mathbb{R}^n, \\
    & Y_{ik} \in \{\text{False}, \text{True}\}, \quad \forall\, k \in K,\; i \in D_k.
\end{aligned}
\qquad \text{(GDP)}
\label{eq:GDP}
\end{equation}
\endgroup

In this GDP formulation, $\mathbf{x} \in \mathbb{R}^n$ denotes the continuous decision variables and 
$\mathbf{Y}$ the collection of Boolean variables. 
The objective function is given by $f: \mathbb{R}^n \to \mathbb{R}$. 
Global constraints are represented by $g: \mathbb{R}^n \to \mathbb{R}^p$, 
where $p$ is the number of global constraints. 
The mapping $g:\mathbb{R}^n \to \mathbb{R}^p$ is vector-valued. Hence, the constraint
$g(\mathbf{x}) \le \mathbf{0}$ is interpreted elementwise, i.e., $g_j(\mathbf{x}) \le 0$
for all $j=1,\dots,p$.
Each disjunct $i$ within disjunction $k$ is associated with 
$h_{ik} : \mathbb{R}^n \to \mathbb{R}^{q_{ik}}$, 
where $q_{ik}$ is the number of constraints in that disjunct. 
The condition $h_{ik}(\mathbf{x}) \leq \mathbf{0}$ is likewise understood 
componentwise, i.e., $h_{ik,\ell}(\mathbf{x}) \leq 0$ for all $\ell = 1,\dots,q_{ik}$. 
The set $K$ indexes all disjunctions, and $D_k$ denotes the set of disjuncts within disjunction $k$. 
The exclusive OR (XOR) condition,
\(
\underset{i \in D_k}{\underline{\bigvee}} Y_{ik},
\)
ensures exactly one Boolean variable $Y_{ik}$ per disjunction $k$ is \text{True}. 
Possible logical relationships among Boolean variables are represented by $\Omega(\mathbf{Y}) = \text{True}$. 
These logical conditions can be systematically translated into algebraic constraints, enabling their integration 
into conventional optimization frameworks \cite{grossmannSystematicModelingDiscretecontinuous2013}.

Although GDPs can be solved directly using specialized logic-based solvers, the more common practice is to reformulate GDPs into MINLP problems and leverage the maturity and broad availability of MINLP solvers~\cite{grossmannReviewNonlinearMixedInteger2002, ovalleLogicBasedDiscreteSteepestDescent2025, ruizGeneralizedDisjunctiveProgramming2012, kronqvist50YearsMixedinteger2025}.
Among the various reformulation strategies,  the most common approach is to use Big-M reformulation. Additionally, the Hull Reformulation (HR) is particularly notable due to its tight continuous relaxation properties, which could enhance solver performance by providing strong lower bounds~\cite{grossmannGeneralizedConvexDisjunctive2003}.

One of the major challenges in applying HR to nonlinear constraints in GDP problems arises from numerical instabilities associated with existing methods to reformulate constraints, such as the widely used \(\varepsilon\)-approximation of the closure of the perspective function~\cite{furmanComputationallyUsefulAlgebraic2020}.
\footnote{The term ``\(\varepsilon\)-approximation of the closure of the perspective function'' is often referred to simply as ``\(\varepsilon\)-approximation'' in this paper.
When we discuss properties or drawbacks of "\(\varepsilon\)-approximation", we are referring to those of the full method (i.e., the application of the $\varepsilon$-approximation of the closure of the perspective function), not only to the fact that it is an approximation rather than an exact function.}
Although computationally practical, the \(\varepsilon\)-approximation complicates the function structure relative to the original constraints in GDP, leading to numerical issues, making convexity detection more difficult, and, since it is an approximation, enlarging the feasible region of the continuous relaxation.

This motivated the present work.
We aim to eliminate the reliance on the $\varepsilon$-approximation when applying HR to GDPs with quadratic disjunctive constraints, while preserving the original algebraic structure that modern MIQCP and conic solvers can exploit.
Concretely, we derive an exact HR for general quadratic constraints (including non-convex quadratics) that avoids fractional perspective terms and does not enlarge the continuous relaxation.
For the convex quadratic case ($Q\succeq 0$), we re-derive an exact conic formulation based on rotated second-order cone representability, which preserves convexity and enables conic-capable solvers to recognize and exploit the resulting structure directly.
We implement these reformulations in \texttt{Pyomo.GDP}~\cite{chenPyomoGDPEcosystemLogic2022} and evaluate their computational impact across several benchmark families.

The remainder of the paper is structured as follows.
Section~\ref{sec:gdp-strategies} reviews common GDP-to-MINLP reformulation strategies, including HR and the numerical issues associated with the $\varepsilon$-approximation.
Section~\ref{sec:hr_q-gdp} introduces the class of quadratically constrained GDPs considered in this work. 
Section~\ref{sec:quadratic-hull} presents the proposed General Exact Hull Reformulation (GEHR) together with an equivalence proof.
Section~\ref{sec:conic-gdp} derives the Conic Exact Hull Reformulation (CEHR) for convex quadratic disjunctive constraints and discusses its implications for convexity recognition.
Section~\ref{sec:experiments} reports computational experiments on randomly generated convex and non-convex instances, a CSTR network optimization problem, $k$-means clustering, and constrained layout problems.
Finally, Section~\ref{sec:conclusion} concludes with a summary of findings and directions for future work.

\section{GDP to MINLP Reformulation Strategies}
\label{sec:gdp-strategies}

GDP problems are commonly solved by reformulating them into MINLP problems.
Although several reformulation approaches exist, all typically share a common structure.
All MINLP reformulations introduce binary variables to represent Boolean conditions and convert propositional logic into linear constraints, but they differ in how they reformulate the constraints within the disjunctions.
The disjunction-independent part of the MINLP reformulation can be presented as follows:

\begingroup
\begin{equation}
\setlength{\jot}{0pt} 
\begin{aligned}
    \min_{\mathbf{x}, \mathbf{y}} \quad 
    & f(\mathbf{x}) \\
    \text{s.t.} \quad 
    & g(\mathbf{x}) \leq\mathbf{0}, \\
    & \sum_{i \in D_k} y_{ik} = 1 
        \quad \forall\, k \in K,\\
    & \mathbf{E} \mathbf{y} \geq \mathbf{e}, \\
    & \mathbf{x}^{\ell} \leq \mathbf{x} \leq \mathbf{x}^{u}, \\
    & \mathbf{x} \in \mathbb{R}^n, \\
    & y_{ik} \in \{0, 1\}
        \quad \forall\, k \in K,\, i \in D_k.
\end{aligned}
\qquad \text{(MINLP without Disjunctions)}
\label{eq:MINLP}
\end{equation}
\endgroup

In this formulation, the Boolean variables  \(Y_{ik}\) from the original GDP are converted to binary variables $y_{ik}\in\{0,1\}$, which activate or deactivate the corresponding disjunctive constraints.
The constraint $\sum_{i \in D_k} y_{ik}=1$ represents the exclusive OR (XOR) condition \(\underset{i \in D_k}{\underline{\bigvee}} Y_{ik}\), ensuring that exactly one disjunct from each disjunction set \( D_k\) is selected.
All logical relationships between disjuncts, originally denoted by \( \Omega(\mathbf{Y}) = \text{True}\) are translated into linear constraints \(\mathbf{E} \mathbf{y} \geq \mathbf{e}\) using established methods \cite{grossmannSystematicModelingDiscretecontinuous2013}.

However, the presented MINLP reformulation does not yet specify how the disjunctive constraints \( h_{ik}(\mathbf{x}) \leq \mathbf{0}\) should be handled.
Each MINLP reformulation strategy differs mainly in how it reformulates these constraints.
The following subsections address how this challenge is handled by reformulations used in this paper.

\subsection{Big-M Reformulation}

In the Big-M reformulation, constraints within each disjunct are activated conditionally using binary variables and large constants (``Big-M'')\cite{trespalaciosReviewMixedIntegerNonlinear2014}. Specifically, disjunctive constraints are reformulated as:

\begingroup
\begin{equation}
\begin{aligned}
    h_{ik}(\mathbf{x}) \leq \mathbf{M_{ik}}(1 - y_{ik})
        \quad \forall\, k \in K,\, i \in D_k, \\
\end{aligned}
\qquad \text{(Big-M Reformulation)}
\end{equation}
\endgroup

\noindent where \(\mathbf{M_{ik}}\) is a vector of sufficiently large constants chosen such that when $y_{ik}=0$ the constraint is always evaluated to values less than $\mathbf{M_{ik}}$ in the domain of the problem, and thus the constraint becomes trivially satisfied regardless of the value of $\mathbf{x}$.
On the other hand, when $y_{ik}=1$, the constraint reverts exactly to its original form $h_{ik}(\mathbf{x})\leq \mathbf{0}$, activating the corresponding disjunct.

The quality of continuous relaxation in the case of Big-M reformulation strongly depends on the choice of values of constants in $\mathbf{M_{ik}}$.
Excessively large values lead to weaker continuous relaxations and poorer lower bounds, thus increasing computational effort.
Conversely, selecting tighter values improves relaxation quality and solver performance but requires careful, problem-specific tuning.
Despite yielding potentially weaker relaxation compared to HR \cite{grossmannGeneralizedConvexDisjunctive2003}, the Big-M reformulation remains an attractive and widely used approach due to its simplicity, ease of implementation, and the fact that it does not increase the number of variables or constraints compared to the original GDP~\cite{kronqvistStepsIntermediateRelaxations2021}.

\subsection{Hull Reformulation}

The Hull Reformulation (HR) provides tighter continuous relaxations and stronger lower bounds for GDP problems than other common approaches, such as Big-M.
HR ensures that the continuous relaxation of each disjunction in the resulting MINLP corresponds exactly to the convex hull of all disjuncts in a disjunction, provided the constraints within disjunctions are convex \cite{grossmannGeneralizedConvexDisjunctive2003,bernalneiraConvexMixedintegerNonlinear2024, kronqvistStepsIntermediateRelaxations2021}.

\subsubsection{ Reformulation of Disjunctions}
HR of disjunctions can be expressed as follows:

\begingroup
\begin{equation}
\begin{aligned}
    & \mathbf{x} = \sum_{i \in D_k} \mathbf{v}_{ik}, \quad k \in K \\
    & \left( \text{cl} \, \widetilde{h}_{ik} \right)(\mathbf{v}_{ik}, y_{ik}) \leq \mathbf{0}, \quad k \in K,\, i \in D_k \\
    & \mathbf x^{\ell} y_{ik} \leq \mathbf{v}_{ik} \leq \mathbf{x}^u y_{ik} \\
    & \mathbf{v}_{ik} \in \mathbb{R}^n, \quad k \in K,\, i \in D_k \\
\end{aligned}
\qquad \text{(HR)}
\label{eq:HR}
\end{equation}
\endgroup

The reformulation introduces disaggregated variables \(\mathbf{v}_{ik}\) as separate copies of the original variables \(\mathbf{x}\), each associated with disjunct \(ik\) in disjunction \(D_k\). 
HR generally assumes finite bounds on all variables in the original GDP, and this assumption is adopted for all problems and reformulations considered in this paper. Under these conditions, the constraints
\(
\mathbf{x}^{\ell} y_{ik} \leq \mathbf{v}_{ik} \leq \mathbf{x}^{u} y_{ik}
\)
ensure that $\mathbf{v}_{ik} = 0$ whenever the associated binary variable $y_{ik} = 0$.
The disjunctive constraint functions \(h_{ik}(\mathbf{x})\) are replaced by the closure of the perspective function
\(
\left(\operatorname{cl}\,\widetilde{h}_{ik}\right)(\mathbf{v}_{ik},y_{ik})
\)~\cite{ceriaConvexProgrammingDisjunctive1999}.
Provided \(h_{ik}\) is defined and finite, and taking into account that the disaggregation bounds
\(\mathbf{x}^{\ell} y_{ik} \le \mathbf{v}_{ik} \le \mathbf{x}^{u} y_{ik}\)
imply \(\mathbf{v}_{ik}=\mathbf{0}\) whenever \(y_{ik}=0\), the closure of the perspective function in HR domain can be written as~\cite{stubbsBranchandcutMethod011999,grossmannGeneralizedConvexDisjunctive2003,gunlukPerspectiveReformulationApplications2012,furmanComputationallyUsefulAlgebraic2020}:

\begingroup
\begin{equation}
    \left( \text{cl} \, \widetilde{h} \right)(\mathbf{v}, y) = 
\begin{cases} 
    y h\left(\frac{\mathbf{v}}{y}\right) & \text{if } y > 0 \\[5pt]
    \mathbf{0}  & \text{if } y = 0 \\[1pt]
\end{cases}
\label{eq:closure}
\end{equation}
\endgroup

Clearly, when \(y_{ik}=1\), the original constraint \(h_{ik}(\mathbf{x})\leq \mathbf{0}\)  is recovered precisely.
On the other hand, when \(y_{ik}=0\), the reformulated constraint simplifies to \(\mathbf{0}\leq\mathbf{0}\), making the constraint trivially satisfied.

\subsubsection{Convex and Non-Convex Constraints}

When dealing with convex GDPs, where each disjunctive constraint is convex, the continuous relaxation of HR creates a convex hull of disjuncts in a disjunction \cite{grossmannGeneralizedConvexDisjunctive2003}, which is the tightest possible convex relaxation for the original disjunctive structure.
This property is the key theoretical motivation behind HR and is what makes it particularly attractive.

However, the behavior of the closure of the perspective function at binary values, described above, and thus the validity of reformulation, is independent of convexity, making it a valid reformulation to apply to non-convex GDPs as well (Appendix~\ref{app:bm-to-hull}).
Applying HR directly to non-convex GDPs yields non-convex MINLP problems that can be solved by global optimization solvers.

For non-convex disjunctive constraints, calling this reformulation a “Hull Reformulation (HR)” can be a misnomer, since the relaxation is generally not the convex hull of the disjunction.
Nevertheless, we use the term HR throughout the paper to refer to the same reformulation construction.
In this non-convex setting, the relaxation itself can be non-convex, and global solvers may construct additional convex relaxations internally (to compute lower bounds), so the convex region explored during the solution process may be larger than the non-convex relaxation obtained directly from HR.

\subsubsection{Numerical Challenges and $\varepsilon$-approximation}

Equation~\eqref{eq:HR} outlines the general strategy for handling constraints in HR.
In the case of linear constraints of the form \(A\mathbf{x}\leq b\), instead of using the definition of closure of the perspective function, the reformulation of linear constraints can be expressed simply as \(A\mathbf{v}\leq b y\)~\cite{balasDisjunctiveProgramming2018}, which preserves linearity and exactly analytically describes the convex hull of the disjunction.

However, nonlinear constraints present a numerical challenge since the closure of the perspective function is not defined analytically at \(y=0\).
To overcome this issue, \(\varepsilon\)-approximation method (where \(\varepsilon\) is a small number, e.g. \(10^{-4}\)) proposed by Furman et al.~\cite{furmanComputationallyUsefulAlgebraic2020}, is typically employed:

\begin{equation}
\left(\text{cl}\,\widetilde{h}\right)(\mathbf{x},y)\approx\left((1-\varepsilon)y+\varepsilon\right)h\left(\frac{\mathbf{x}}{(1-\varepsilon)y+\varepsilon}\right)-\varepsilon\,h(0)(1-y)
\label{eq:approx}
\end{equation}

Although computationally practical and exact at binary values of $y$, $\varepsilon$-approximation results in exact HR for relaxation only when \(\varepsilon\rightarrow0\), and however small, any values \(\varepsilon>0\) enlarge the feasible region of the continuous relaxation \cite{furmanComputationallyUsefulAlgebraic2020}, potentially weakening solver effectiveness.
Additionally, application of $\varepsilon$-approximation of the closure of the perspective function creates a more complex functional structure compared to the original disjunctive constraints, introducing numerical instabilities and making convexity detection more difficult, further motivating the development of exact HR approaches presented later in this work.

Figure~\ref{fig:epsilon_approx} illustrates the feasible region of continuous relaxation obtained from applying HR to convex and non-convex GDPs, providing some intuition about the relaxation's feasible region and its enlargement when $\varepsilon$-approximation is used.

\begin{figure}[htbp]
  \centering

  \begin{subfigure}[t]{0.48\textwidth}
    \centering
    \includegraphics[width=\textwidth]{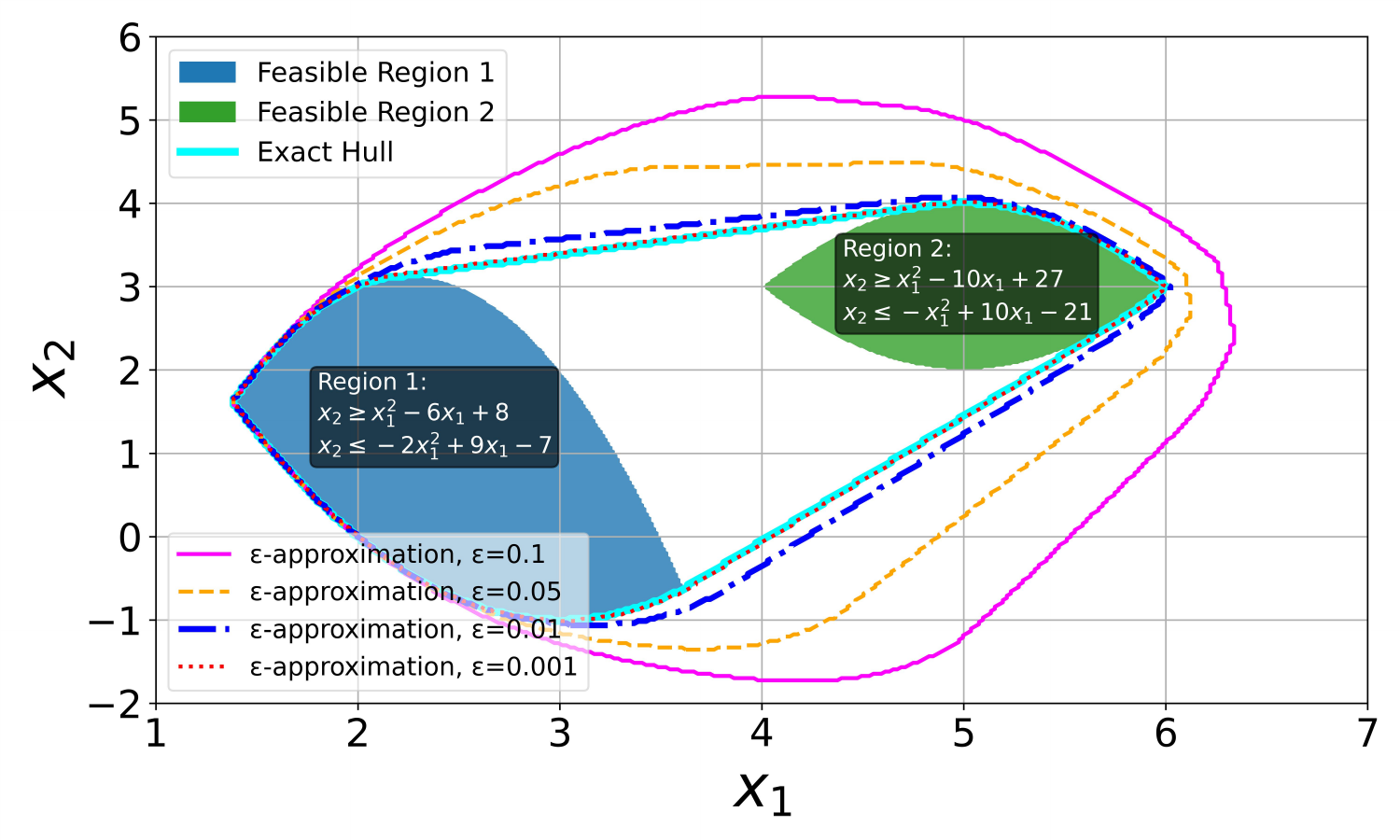}
    \caption{Convex GDP}
    \label{fig:convex_region}
  \end{subfigure}
  \hfill
  \begin{subfigure}[t]{0.48\textwidth}
    \centering
    \includegraphics[width=\textwidth]{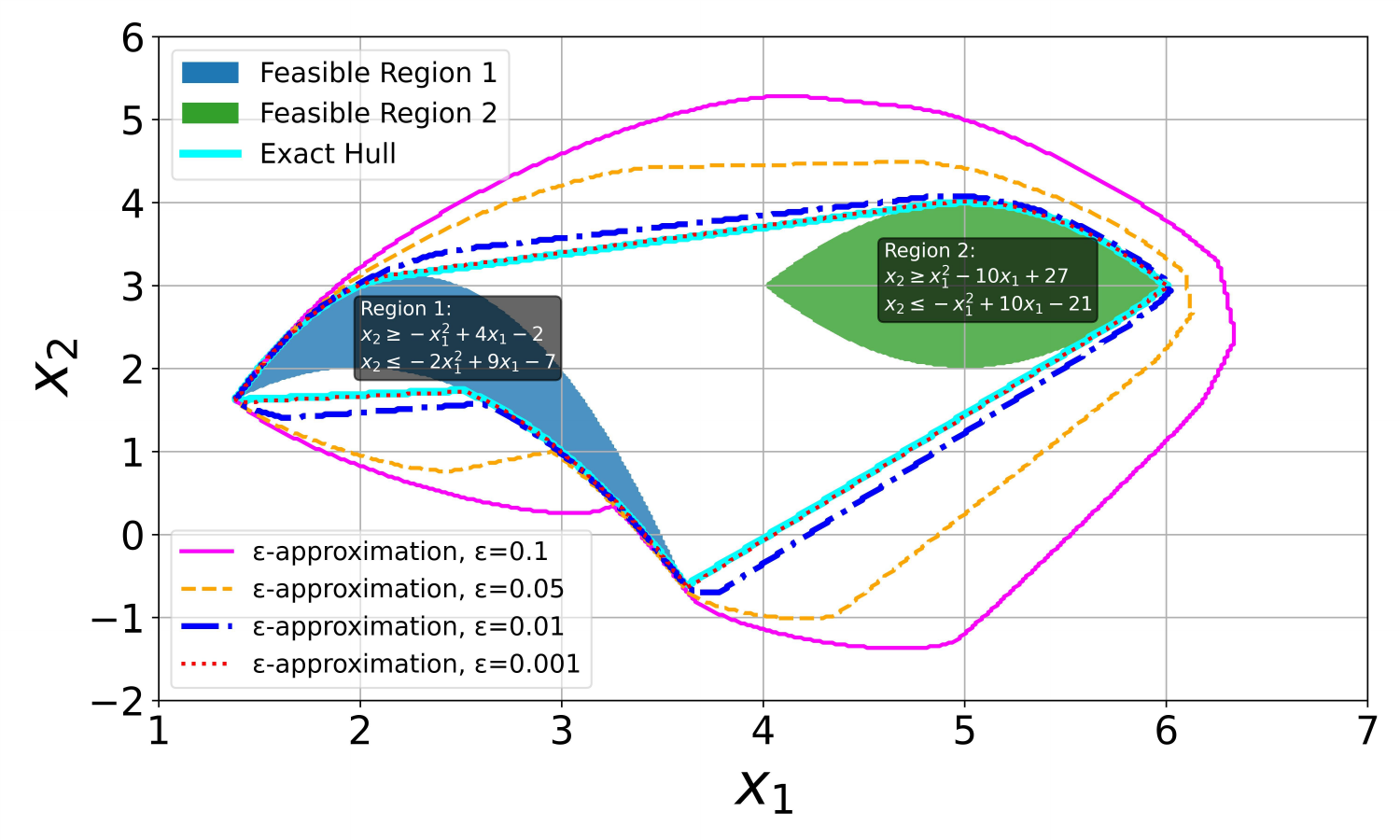}
    \caption{Non-convex GDP}
    \label{fig:epsilon_approx_sub}
  \end{subfigure}

  \caption{Illustration of the effect of the $\varepsilon$-approximation on the feasible region of the continuous relaxation resulting from HR applied to convex and non-convex GDP. The figure shows the exact feasible region with the enlarged region obtained using the $\varepsilon$-approximation, highlighting how the approximation leads to a relaxation that is larger than the exact one and weakens as $\varepsilon$ increases. The Big-M relaxation is not shown, as its feasible region would significantly exceed the plotted area, taking up most of the figure at this scale and axis limits.}

  \label{fig:epsilon_approx}
\end{figure}

\subsubsection{Trade-offs and Practical Considerations}

HR significantly improves the quality of the continuous relaxation, leading to tighter lower bounds.
However, this comes at the cost of increased model complexity due to the introduction of additional variables and constraints \cite{grossmannGeneralizedDisjunctiveProgramming2012}.
This trade-off between stronger relaxations and higher dimensionality makes it difficult to determine in general whether HR will outperform the Big-M reformulation in terms of solution time, and the results vary on a case-by-case basis \cite{kronqvistStepsIntermediateRelaxations2021, kronqvist50YearsMixedinteger2025}.

\subsection{Binary Multiplication Reformulation}

One of the simplest and most straightforward methods to transform the GDP problem into a MINLP formulation is the binary multiplication approach.
This method involves directly multiplying each constraint within a disjunct by the binary variable associated with the disjunct, which activates or deactivates the corresponding constraint.
The Binary Multiplication reformulation of disjunctions can be expressed as:

\begingroup
\begin{equation}
\begin{aligned}
    & y_{ik}h_{ik}(\mathbf{x}) \leq \mathbf{0}
        \quad \forall\, k \in K,\, i \in D_k, \\
\end{aligned}
\qquad \text{(Binary Multiplication Reformulation)}
\end{equation}
\endgroup

In this formulation, when the binary variable \(y_{ik} = 0\), the corresponding constraint becomes \(\mathbf{0}\leq\mathbf{0}\), which is trivially satisfied. When \(y_{ik} = 1\), the original constraint \(h_{ik}(\mathbf{x}) \leq \mathbf{0}\) is recovered.
Therefore, the disjunctive logic of the original GDP formulation is correctly enforced, and the reformulation is valid.

However, the main drawback of this approach is that the continuous relaxation of the reformulation is often non-convex.
In fact, when projected to $\mathbf{x}$, the continuous relaxation has the same feasible region as the original MINLP with binary values, providing no additional convex relaxation benefit to the solver.
As a result, for almost any problem, MINLP solvers must rely on internal convexification techniques to navigate the non-convex feasible space efficiently during the branch-and-bound procedure.

Despite its lack of favorable relaxation properties, the binary multiplication reformulation is valid and easy to implement.
It can be directly applied to GDP problems and passed to a general MINLP solver, which can handle non-convex problems.

\section{Hull Reformulations of Quadratically Constrained GDP}
\label{sec:hr_q-gdp}

In this work, we focus on HR of GDPs with disjunctive quadratic constraints.
Disjunctions of quadratically constrained GDP can be expressed as follows:

\begingroup
\begin{equation}
\setlength{\jot}{0pt}
\begin{aligned}
    &\bigvee_{i \in D_k} 
        \begin{bmatrix} 
            Y_{ik} \\
            \mathbf{x}^\top Q_{jik} \mathbf{x} + \mathbf{c}_{jik}^\top \mathbf{x} + d_{jik} \leq 0 ,\, j \in J_{ik}
        \end{bmatrix}, \quad k \in K \\
\end{aligned}
\label{eq:Q-GDP}
\end{equation}
\endgroup

\noindent where \(J_{ik}\) is a set of quadratic constraints in a disjunct, \(Q \in \mathbb{R}^{n \times n}, \mathbf{c} \in \mathbb{R}^n, d \in \mathbb{R}\).

The application of HR to the quadratically constrained disjunctions, using \(\varepsilon\)-approximation of the closure of the perspective function shown in Equation~\eqref{eq:approx}, after algebraic simplification yields:

\begingroup
\begin{equation}
\begin{aligned}
    & \left(\operatorname{cl}\,\widetilde{h}\right)(\mathbf{v},y) \approx \frac{\mathbf{v}^\top Q_{jik}  \mathbf{v}}{(1-\varepsilon)y + \varepsilon} + \mathbf{c}_{jik}^\top \mathbf{v} + d_{jik} y \leq 0, \quad j \in J_{ik},\, k \in K,\, i \in D_k \\
\end{aligned}
\label{eq:Q-HR}
\end{equation}
\endgroup

Thus, it is clear that application of \(\varepsilon\)-approximation to the quadratically constrained problem results in an MINLP with a more complicated functional structure as the quadratic term is replaced by a fractional term of the form \( \frac{\mathbf{v}^\top Q_{jik}  \mathbf{v}}{(1-\varepsilon)y + \varepsilon}  \).

\subsection{General Exact Hull Reformulation of Quadratically Constrained GDP}
\label{sec:quadratic-hull}

Consider a GDP in which the disjunctions contain quadratic, possibly non-convex, constraints of the form
\(
    h(\mathbf x)\;=\;\mathbf x^{\mathsf T}Q\mathbf x
                 + \mathbf{c}^{\mathsf T}\mathbf x + d \leq 0,
\)
\noindent with finite bounds \( \mathbf{x}^{\ell} \leq \mathbf{x} \leq \mathbf{x}^{u}\), and continuous relaxation obtained through the standard exact HR.  
In lifted space, the feasible set of the continuous relaxation associated with a single constraint is  
\begin{equation}
    \begin{aligned}
        S_{1}
        \;=\;
        \Bigl\{
            (\mathbf v,y) \in \mathbb{R}^n \times \mathbb{R} \mid y \in [0, 1], \;
            \operatorname{cl}\widetilde h(\mathbf v,y)\le 0,\;
            \mathbf x^{\ell}y\le \mathbf v\le \mathbf x^{u}y
        \Bigr\},
    \end{aligned}
    \label{eq:s1}
\end{equation}
where the closure of the perspective function simplifies to  
\begingroup
\begin{equation}
    \operatorname{cl}\widetilde h(\mathbf v,y)
    \;=\;
    \begin{cases}
        \dfrac{\mathbf v^{\mathsf T}Q\mathbf v}{y}
        + \mathbf{c}^{\mathsf T}\mathbf v
        + dy, & y>0, \\[6pt]
        0, & y=0.
    \end{cases}
    \label{eq:q-closure}
\end{equation}
\endgroup

Multiplying the function \( \dfrac{\mathbf v^{\mathsf T}Q\mathbf v}{y}
        + \mathbf c^{\mathsf T}\mathbf v
        + dy\)  by~\(y\) motivates the alternative
set:
\begin{equation}
\begin{aligned}
S_{2} & =  \left\{ (\mathbf{v}, y) \in \mathbb{R}^n \times \mathbb{R} \mid y \in [0, 1], \; \mathbf{v}^\top Q \mathbf{v} + (\mathbf{c}^\top \mathbf{v}) y + d y^2 \leq 0,  \; \mathbf x^{\ell} y \leq \mathbf{v} \leq \mathbf{x}^u y \right\} &
\end{aligned}
\label{eq:s2}
\end{equation}

We demonstrate the equivalence of these two sets in the following proposition:

\begin{proposition}
\label{prop:sets}
For a quadratic constraint \(
    h(\mathbf v)\;=\;\mathbf v^{\mathsf T}Q\mathbf v
                 + \mathbf{c}^{\mathsf T}\mathbf v + d \leq 0,
\) define the sets
\begin{equation*}
\begin{aligned}
S_1 & = \left\{ (\mathbf{v}, y) \in \mathbb{R}^n \times \mathbb{R} \,\middle|\, y \in [0, 1], \; \operatorname{cl} \widetilde{h}(\mathbf{v}, y) \leq 0, \; \mathbf x^{\ell} y \leq \mathbf{v} \leq \mathbf{x}^u y \right\}
\end{aligned}
\end{equation*}
where
\begingroup
\begin{equation*}
\begin{aligned}
    \left( \text{cl} \, \widetilde{h} \right)(\mathbf{v}, y) = 
\begin{cases} 
    \frac{\mathbf{v}^\top Q  \mathbf{v}}{y} + \mathbf{c}^\top \mathbf{v} + d y  & \text{if } y > 0 \\[5pt]
    0 & \text{if } y = 0 \\[1pt]
\end{cases}
\end{aligned}
\end{equation*}
\endgroup
and
\[
S_{2} =  \left\{ (\mathbf{v}, y) \in \mathbb{R}^n \times \mathbb{R} \mid y \in [0, 1], \; \mathbf{v}^\top Q \mathbf{v} + (\mathbf{c}^\top \mathbf{v}) y + d y^2 \leq 0,  \; \mathbf x^{\ell} y \leq \mathbf{v} \leq \mathbf{x}^u y \right\}
\]
Then, the two sets coincide: \( S_1 = S_2\) 
\end{proposition}

\begin{proof}

\textbf{Case \(y>0\):}  
Since \(y\) is positive, the constraint \(\frac{\mathbf{v}^\top Q \mathbf{v}}{y} + \mathbf{c}^\top \mathbf{v} + d y \leq 0\) can be multiplied by \(y\) to yield a constraint with the same feasible region:
\begin{equation*}
\frac{\mathbf{v}^\top Q \mathbf{v}}{y} + \mathbf{c}^\top \mathbf{v} + d y \leq 0 \quad \Longleftrightarrow \quad \mathbf{v}^\top Q \mathbf{v} + (\mathbf{c}^\top \mathbf{v}) y + d y^2 \leq 0, \quad \text{when } y>0.
\end{equation*}

\noindent
\textbf{Case \(y=0\):}   
 The constraints \(\mathbf x^{\ell}y\le\mathbf v\le\mathbf x^{u}y\) (finite bounds) force
\(\mathbf v=\mathbf 0\) at \(y=0\).

Substituting \((\mathbf 0,0)\) into \(\mathbf{v}^\top Q \mathbf{v} + (\mathbf{c}^\top \mathbf{v}) y + d y^2  \) yields~0, identical to
\(\operatorname{cl}\widetilde h(\mathbf 0,0)\).
Thus, the feasible points with
\(y=0\) are the same in both sets.

Since the feasible regions coincide in both subcases, the equivalence \( S_1 = S_2\) is established.
\end{proof}

Hence, HR for quadratically constrained GDP problems can be exactly expressed using set \(S_2\) for each quadratic constraint, avoiding the need for approximations and maintaining quadratic functional structure in the reformulation.
Consequently, when applying HR, it is possible to use \(S_2\) for any disjunctive quadratic constraint. In this case, the General Exact Hull Reformulation (GEHR) of quadratically constrained disjunctions shown in Equation~\eqref{eq:Q-GDP} can be expressed as:

\begingroup
\begin{equation}
\begin{aligned}
    & \mathbf{x} = \sum_{i \in D_k} \mathbf{v}_{ik}, \quad k \in K \\
    & \mathbf{v}_{ik}^\top Q_{jik} \mathbf{v}_{ik} + \mathbf{c}^\top_{jik} \mathbf{v}_{ik} y_{ik} 
      + d_{jik} y_{ik}^2 \leq 0, \quad j \in J_{ik},\, k \in K,\, i \in D_k \\
    & \mathbf x^{\ell} y_{ik} \leq \mathbf{v}_{ik} \leq \mathbf{x}^u y_{ik} \\
    & \mathbf{v}_{ik} \in \mathbb{R}^n, \quad k \in K,\, i \in D_k \\
\end{aligned}
\label{eq:our-HR}
\end{equation}
\endgroup

The proposed exact reformulation provides two major advantages over \(\varepsilon\)-approximation approaches.
First, because it is exact, it avoids the feasible-region enlargement introduced by the $\varepsilon$-approximation and can therefore strengthen the relaxation and the resulting lower bounds.
Additionally, it preserves the quadratic structure of the problem, potentially enhancing computational efficiency and eliminating numerical instabilities associated with the \(\varepsilon\)-approximation.
This structural preservation enables the use of advancements in Mixed-Integer Quadratically Constrained Programming (MIQCP) solvers when the original problem is quadratically constrained, whereas the \(\varepsilon\)-approximation alters the problem structure.

However, the key drawback of this approach arises when this reformulation is applied to \emph{convex} quadratic constraints. 
Although the resulting feasible region of the continuous relaxation is identical to that of the standard exact HR and is therefore convex, the inequality
\[
\mathbf{v}^\top Q \mathbf{v} + (\mathbf{c}^\top \mathbf{v})\,y + d\,y^{2} \le 0
\]
is not necessarily convex as a function of $(\mathbf{v},y)$. 
As a result, solvers may classify these constraints as non-convex, which can significantly degrade performance compared to formulations that preserve explicit convexity, including the $\varepsilon$-approximation. 
We address this issue in the next section by re-deriving CEHR for the convex quadratic case.

Additionally, we note that this approach can be further generalized to polynomial constraints of arbitrary degree, enabling the construction of a HR in which the resulting expressions retain the same polynomial degree (Appendix~\ref{app:poly-hull}).
Alternatively, polynomial constraints can be reformulated into an equivalent problem with quadratic constraints using established techniques \cite{kariaAssessmentTwostepApproach2022}, to which the proposed GEHR for quadratic constraints can be applied.
Thus, using either strategy, the proposed GEHR approach can be extended to polynomial constraints of any degree.

\subsection{Conic Exact Hull Reformulation of Convex Quadratically Constrained GDP}
\label{sec:conic-gdp}
Convex GDP usually refers to a problem in which the objective function and all constraints in the original GDP formulation are convex.
Specifically, for quadratically constrained GDPs, convexity implies that the matrices \(Q\)  associated with quadratic constraints are positive semi-definite (\(Q \succeq 0 \)). 

It has been shown that HR of GDPs with second-order cone constraints admits an exact second-order cone representation \cite{bernalneiraConvexMixedintegerNonlinear2024}.
However, to the best of our knowledge, the literature does not provide an explicit, solver-friendly algebraic formulation for HR of convex quadratic disjunctive constraints that avoids the perspective division and preserves a conic-representable structure.
Moreover, this conic reformulation has not been integrated into commonly used GDP modeling frameworks such as \texttt{Pyomo.GDP}~\cite{chenPyomoGDPEcosystemLogic2022}.
Thus, this subsection focuses on GDP disjunctions whose active constraints are \emph{convex} quadratic
inequalities.

Consider the following disjunctive constraint:
\begin{equation}
\label{eq:convex-q-gdp}
    h(\mathbf x)
    \;:=\;
    \mathbf x^{\mathsf T}Q\mathbf x
    + \mathbf c^{\mathsf T}\mathbf x
    + d
    \;\le\;0,
    \qquad
    Q\succeq 0 .
\end{equation}
As it was mentioned before, the standard HR can be produced by the following constraint
\begin{equation}
\label{eq:convex-perspective}
    \bigl(\operatorname{cl}\widetilde h\bigr)(\mathbf v,y)
    \;=\;
    \begin{cases}
        \dfrac{\mathbf v^{\mathsf T}Q\mathbf v}{y}
        +\mathbf c^{\mathsf T}\mathbf v
        +dy, & y>0,\\[6pt]
        0, & y=0,
    \end{cases}
    \qquad \le 0,
\end{equation}
together with the standard disaggregation bounds
$\mathbf x^{\ell}y\le \mathbf v\le \mathbf x^u y$.

For convex quadratics, however, one can derive an \emph{exact} algebraic formulation that avoids division and preserves quadratic structure, and is representable with a conic constraint.

\subsubsection{ Conic Exact Hull (without factorization)}
The closure of the perspective function constraint \eqref{eq:convex-perspective} can be written (for $y>0$) as
\begin{equation}
\label{eq:persp-ypos}
    \frac{\mathbf v^{\mathsf T}Q\mathbf v}{y}
    + \mathbf c^{\mathsf T}\mathbf v
    + d\,y
    \;\le\; 0.
\end{equation}
A convenient way to avoid the division is to separate the quadratic-over-$y$ term from the affine terms
via an epigraph variable. Introduce a scalar $t$ and impose
\begin{equation}
\label{eq:epi-split}
    \frac{\mathbf v^{\mathsf T}Q\mathbf v}{y} \;\le\; t,
    \qquad
    t + \mathbf c^{\mathsf T}\mathbf v + d\,y \;\le\; 0.
\end{equation}
Because $Q\succeq 0$ implies $\mathbf v^{\mathsf T}Q\mathbf v \ge 0$, we may additionally enforce $t\ge 0$.

For $y>0$, the first inequality in \eqref{eq:epi-split} is equivalent to
\begin{equation}
\label{eq:rsoc-core}
    \mathbf v^{\mathsf T}Q\mathbf v \;\le\; t\,y,
    \qquad t\ge 0,\; y>0,
\end{equation}
which is a (rotated) conic inequality. The second inequality in \eqref{eq:epi-split} is linear:
\begin{equation}
\label{eq:affine-part}
    t + \mathbf c^{\mathsf T}\mathbf v + d\,y \;\le\; 0.
\end{equation}

At $y=0$, the disaggregation bounds enforce $\mathbf v=\mathbf 0$, so \eqref{eq:rsoc-core} holds trivially.
Moreover, \eqref{eq:affine-part} together with $t\ge 0$ forces $t=0$, matching the closure value
$\bigl(\operatorname{cl}\widetilde h\bigr)(\mathbf 0,0)=0$.
Therefore, the projection of the system
\eqref{eq:rsoc-core}--\eqref{eq:affine-part} (with $t\ge 0$ and the standard disaggregation bounds)
onto $(\mathbf v,y)$ coincides exactly with the closure of the perspective function constraint \eqref{eq:convex-perspective}.
In particular, this yields an exact reformulation that avoids the $\varepsilon$-approximation and removes the fractional term.

The validity of this reformulation is formally shown in the following proposition:
\begin{proposition}[Conic epigraph reformulation of HR]
\label{prop:conic-epi-equiv}
Let $Q\succeq 0$ and define
\begin{equation}
\label{eq:conic-set-S1}
S_{1}
=
\Bigl\{(\mathbf v,y)\in\mathbb R^{n}\times\mathbb R \ \Bigm|\ 
y\in[0,1],\ 
\bigl(\operatorname{cl}\widetilde h\bigr)(\mathbf v,y)\le 0,\ 
\mathbf x^{\ell}y\le \mathbf v\le \mathbf x^{u}y
\Bigr\},
\end{equation}
where $\bigl(\operatorname{cl}\widetilde h\bigr)$ is given by \eqref{eq:convex-perspective}. Define also
\begin{equation}
\label{eq:conic-set-S2}
\begin{aligned}
S_{2}^{\mathrm{conic}}
=
\Bigl\{(\mathbf v,y)\in\mathbb R^{n}\times\mathbb R \;\Bigm|\;&
y\in[0,1],\ \exists\, t\in\mathbb R,\ 
\mathbf v^{\mathsf T}Q\mathbf v \le t\,y,\\
& t + \mathbf c^{\mathsf T}\mathbf v + d\,y \le 0,\ 
t\ge 0,\ 
\mathbf x^{\ell}y\le \mathbf v\le \mathbf x^{u}y
\Bigr\}.
\end{aligned}
\end{equation}
Then $S_{1} = S_{2}^{\mathrm{conic}}$.
\end{proposition}

\begin{proof}
We prove this by mutual containment of the sets.

First, take $(\mathbf v,y)\in S_{1}$.
If $y=0$, the bounds $\mathbf x^{\ell}y\le \mathbf v\le \mathbf x^{u}y$ imply $\mathbf v=\mathbf 0$.
Then the choice $t=0$ satisfies $\mathbf v^{\mathsf T}Q\mathbf v = 0 \le t\,y=0$ and
$t+\mathbf c^{\mathsf T}\mathbf v + d\,y = 0$. Moreover, for any feasible $t$ we must have
$t+\mathbf c^{\mathsf T}\mathbf 0 + d\cdot 0 \le 0$ and $t\ge 0$, hence $t=0$ is forced.
Therefore $(\mathbf v,y)\in S_{2}^{\mathrm{conic}}$.
If $y>0$, define $t := \frac{\mathbf v^{\mathsf T}Q\mathbf v}{y}\ge 0$ (since $Q\succeq 0$).
Then $\mathbf v^{\mathsf T}Q\mathbf v \le t\,y$ holds at equality, and
$\bigl(\operatorname{cl}\widetilde h\bigr)(\mathbf v,y)\le 0$ implies
$t+\mathbf c^{\mathsf T}\mathbf v + d\,y \le 0$. Hence $(\mathbf v,y)\in S_{2}^{\mathrm{conic}}$.

Second, take $(\mathbf v,y)\in S_{2}^{\mathrm{conic}}$ with some $t\ge 0$ satisfying \eqref{eq:conic-set-S2}.
If $y=0$, the bounds imply $\mathbf v=\mathbf 0$, and then
$\bigl(\operatorname{cl}\widetilde h\bigr)(\mathbf 0,0)=0\le 0$, so $(\mathbf v,y)\in S_{1}$.
If $y>0$, the inequality $\mathbf v^{\mathsf T}Q\mathbf v \le t\,y$ implies
$\frac{\mathbf v^{\mathsf T}Q\mathbf v}{y}\le t$, and combining with
$t + \mathbf c^{\mathsf T}\mathbf v + d\,y \le 0$ yields
$\frac{\mathbf v^{\mathsf T}Q\mathbf v}{y} + \mathbf c^{\mathsf T}\mathbf v + d\,y \le 0$,
which is exactly $\bigl(\operatorname{cl}\widetilde h\bigr)(\mathbf v,y)\le 0$ for $y>0$.
Therefore $(\mathbf v,y)\in S_{1}$.

Thus $S_{1} = S_{2}^{\mathrm{conic}}$.
\end{proof}

Putting indices back and applying this construction to every convex quadratic constraint in each disjunct, the Conic Exact Hull Reformulation (CEHR) of the quadratically constrained disjunctions \eqref{eq:Q-GDP} is:
\begingroup
\begin{equation}
\label{eq:convex-conic-hull-nochol}
\begin{aligned}
    &\mathbf x = \sum_{i\in D_k}\mathbf v_{ik}, && k\in K,\\
    &\mathbf x^{\ell}y_{ik}\le \mathbf v_{ik}\le \mathbf x^{u}y_{ik},
        && k\in K,\; i\in D_k,\\
    &\mathbf v_{ik}^{\mathsf T}Q_{jik}\mathbf v_{ik} \le t_{jik}\,y_{ik},
        && j\in J_{ik},\; k\in K,\; i\in D_k,\\
    &t_{jik} + \mathbf c_{jik}^{\mathsf T}\mathbf v_{ik} + d_{jik}\,y_{ik} \le 0,
        && j\in J_{ik},\; k\in K,\; i\in D_k,\\
    &t_{jik}\ge 0,
        && j\in J_{ik},\; k\in K,\; i\in D_k,
\end{aligned}
\qquad \text{(CEHR)}
\end{equation}
\endgroup
This formulation is exact for convex quadratic disjunctive constraints ($Q_{jik}\succeq 0$) and avoids the division by $y_{ik}$ present in the perspective form \eqref{eq:convex-perspective}.
Moreover, the only nonlinearities appear in the quadratic-over-product constraints $\mathbf v_{ik}^{\mathsf T}Q_{jik}\mathbf v_{ik} \le t_{jik}y_{ik}$, which are rotated second-order cone representable.
As a result, the HR can be handled by conic-capable solvers without explicitly factoring
$Q_{jik}$ at the modeling layer.

\subsubsection{Explicit Rotated Second Order Cone form via factorization}

In our main implementation and experiments, we keep CEHR in the algebraic form shown in \eqref{eq:convex-conic-hull-nochol}.
Modern conic-capable solvers (notably \texttt{Gurobi} and \texttt{SCIP}) can often recognize a rotated second-order cone (RSOC) structure in \eqref{eq:rsoc-core} during presolve, so an explicit factorization of $Q$ is not required according to our observations.
Additional information on the performance of different algebraic representations of conic constraints is provided in Appendix~\ref{app:conic-repr}.

Although we do not use the explicit conic form in our main computational study, we present it here for
completeness. Since $Q\succeq 0$, there exists a matrix $L$ such that
\[
    Q = L L^{\mathsf T}.
\]
For example, when $Q\succ 0$, $L$ can be obtained via a Cholesky factorization.
When $Q$ is positive semidefinite (possibly singular), one option is a spectral factorization using an eigen-decomposition $Q = U\Lambda U^{\mathsf T}$ with $\Lambda\succeq 0$, yielding $L := U\Lambda^{1/2}$.

Then
\[
    \mathbf v^{\mathsf T}Q\mathbf v
    = \mathbf v^{\mathsf T} L L^{\mathsf T}\mathbf v
    = \|L^{\mathsf T}\mathbf v\|_2^2,
\]
and \eqref{eq:rsoc-core} is equivalent to
\[
    \|L^{\mathsf T}\mathbf v\|_2^2 \le t\,y,
    \qquad t\ge 0,\; y\ge 0.
\]
This is a rotated second-order cone representable set via
\[
    \left(\frac{t}{2},\, y,\, L^{\mathsf T}\mathbf v\right)\in\mathcal Q_r
    \;:=\;\left\{(\alpha,\beta,\mathbf w)\ \middle|\ 2\alpha\beta \ge \|\mathbf w\|_2^2,\ \alpha\ge 0,\ \beta\ge 0\right\}.
\]

This yields an explicit mixed-integer second-order cone (MISOCP) formulation of HR for convex quadratic disjunctive constraints, re-deriving the known result that hull reformulation of second-order cone constraints admit second-order cone representations \cite{bernalneiraConvexMixedintegerNonlinear2024}.

\subsubsection{Extension to general second-order cone representable constraints}
\label{sec:conic-extension-soc}

It has been shown in the literature that HR remains representable in the same cone families as those used within the original disjunctions \cite{bernalneiraConvexMixedintegerNonlinear2024}.
Our convex quadratic case fits naturally within this framework.
Although a generic convex quadratic inequality
\[
\mathbf x^{\mathsf T}Q\mathbf x+\mathbf c^{\mathsf T}\mathbf x+d\le 0,
\qquad Q\succeq 0,
\]
is not itself written as a cone membership constraint; it is second-order cone representable after introducing lifting.

In this paper, we focus on the conic reformulation of convex quadratic disjunctive constraints.
However, the similar reformulation principle extends to any disjunctive constraint that admits an SOC representation, either in the original variables or after lifting.
Thus, while our computational study for CEHR is restricted to convex quadratic constraints, the approach applies more broadly to SOC-representable disjunctive constraints, possibly after introducing auxiliary variables. 
The presented CEHR for convex functions can therefore be viewed as a special case of a more general result for conic-representable disjunctive constraints.

\section{Computational experiments}
\label{sec:experiments}

The primary objective of the computational experiments presented in this paper is to evaluate the performance of the proposed GEHR and CEHR (for convex problems) in comparison to \(\varepsilon\)-approximation method, which is usually used for any nonlinear constraints in HR.
Although the theoretical foundations of HR have been developed for convex GDPs~\cite{ceriaConvexProgrammingDisjunctive1999,grossmannGeneralizedConvexDisjunctive2003}, and most existing literature concentrates on convex GDP formulations~\cite{bernalneiraConvexMixedintegerNonlinear2024, grossmannGeneralizedConvexDisjunctive2003, grossmannGeneralizedDisjunctiveProgramming2012}, as highlighted above, HR is a valid approach to apply to both convex and non-convex problems.
Consequently, this study investigates both convex and non-convex GDPs to provide a comprehensive comparison.

The existing literature on non-convex GDPs proposes convexifying the constraints within each disjunct prior to applying techniques developed for convex GDPs, such as HR or Big-M, to obtain a convex MINLP~\cite{ruizGlobalOptimizationNonconvex2017, trespalaciosCuttingPlanesImproved2016}.
A comparison between this approach and the direct application of HR or Big-M reformulation to non-convex MINLPs represents a compelling direction for future research, although it lies beyond the scope of this paper.

Similarly, a detailed assessment of the relative advantages and disadvantages of alternative reformulation strategies, such as Big-M or Binary Multiplication, compared to HR, particularly in the context of nonconvex problems, is deferred to future studies. Performance results for the non-hull reformulations (Big-M and Binary Multiplication) are reported as obtained directly from \texttt{Pyomo.GDP} (i.e., without Big-M parameter tuning)~\cite{chenPyomoGDPEcosystemLogic2022}. These results are included solely to provide a broader context.
We emphasize that a comprehensive comparison of the proposed hull reformulations with other reformulation methods is not the objective of this work.

In this section, we present results from computational experiments on various quadratically constrained GDPs, solved using state-of-the-art general MINLP solvers after applying different reformulation approaches.
All test problems were initially formulated as GDP and reformulated to MINLP using \texttt{Pyomo} (version 6.9.2) \cite{bynumPyomoOptimizationModeling2021}. 
As part of this work, we implemented GEHR and CEHR for quadratic constraints within \texttt{Pyomo.GDP}~\cite{chenPyomoGDPEcosystemLogic2022} by extending its existing  HR framework. 
All computational experiments in the main part of the paper used the default \texttt{Pyomo} setting of  \(\varepsilon = 10^{-4}\) for  \(\varepsilon\)-approximation approach, since the values of \(\varepsilon \) were found not to influence results qualitatively. For details on the effect of different values of \(\varepsilon\), refer to Appendix \ref{app:epsilon}.

The resulting MINLP formulations were subsequently solved in \texttt{GAMS} (version 52.2.0) \cite{GAMS} using \texttt{Gurobi} (version 13.0) \cite{GurobiOptimizationLLC}, \texttt{BARON} (version 25.11.17.) \cite{OptimizationFirmBARON}, and \texttt{SCIP} (version 9.2.4) \cite{BolusaniEtal2024OO}.

All problems were solved with a wall-clock time limit of 5 minutes (300 seconds).

To reduce the influence of solver stopping criteria on the reported outcomes across solvers, we attempted to align key tolerance parameters as closely as possible, including feasibility tolerances, integrality tolerances, and optimality gap tolerances.
In our implementation, we used a common target set of tolerances (relative MIP gap $10^{-6}$, absolute MIP gap $10^{-10}$, feasibility tolerance $10^{-6}$, optimality tolerance $10^{-6}$, and integrality tolerance $10^{-5}$) and mapped them to the closest available options in each solver interface.

Nevertheless, exact matching of numerical settings across \texttt{Gurobi}, \texttt{SCIP}, and \texttt{BARON} is not always possible because solvers expose different parameterizations and may interpret tolerances differently. Importantly, this work is not intended to compare solvers against each other; rather, it compares reformulation strategies within each solver. Consequently, the reported data should not be used to draw conclusions about the relative performance of the solvers themselves.

All experiments were run single-threaded on an AMD EPYC 7643 CPU with 1 TB RAM, under Red Hat Enterprise Linux 9.5. Solver parallelism was disabled by setting the solver thread count to 1.

The following subsections summarize the main findings from computational experiments on different quadratically constrained GDP problems.
These experiments compare reformulation strategies using different solvers.
Detailed numerical results and problem generation scripts are available online.
\footnote{Available at: \url{https://github.com/SECQUOIA/exact_quadratic_hull}.}

\subsection{Random Instances (convex and non-convex GDPs)}

To comprehensively evaluate the effectiveness of the proposed Hull Reformulations, we conducted computational experiments on randomly generated instances of quadratically constrained GDPs with quadratic disjunctive constraints, no global constraints, and a quadratic objective function.
These problems featured dense matrices and vectors of coefficients, meaning all coefficients were nonzero.
Coefficients were randomly generated from a uniform distribution over the interval \( [-1,1]\), and all continuous variables were bounded by the interval \( [-1,1]\).
For the convex problem instances, eigenvalues of the generated random matrices \(Q\) were adjusted to be non-negative to ensure positive semidefiniteness of \(Q\), thereby guaranteeing convexity of the problem.

Recall that in the formulation presented in Equation~\eqref{eq:Q-GDP}, $D_k$ is a set of disjuncts in a disjunction $k$, $J_{ik}$ is a set of constraints in a disjunct $ik$, $n$ is the number of dimensions of variable $\mathbf{x}$.

The problem generation code allowed for systematic variation in parameters such as the number of dimensions \(n\), disjunctions \(|K|\), disjuncts per disjunction \(|D_k|\), constraints per disjunct \(|J_{ik}|\), and the number of explicitly guaranteed feasible points.
In the experiments carried out, each instance included 10 randomly generated points explicitly ensured to be feasible by adjusting the constant \(d_{jik}\) in the constraints.

A total of 240 convex instances were generated with parameters in the following ranges:
\(
3 \leq |K| \leq 10, \; 10 \leq |D_k| \leq 15, \; |J_{ik}| = 10, \; 3 \leq n \leq 7
\) and 100 non-convex instances with parameters in the following ranges: \(
3 \leq |K| \leq 10, \; 10 \leq |D_k| \leq 15, \; |J_{ik}| = 10, \; 3 \leq n \leq 9
\).
Generated GDP models were saved to enable execution with different reformulations and solvers.

Figures~\ref{fig:conv_random} and~\ref{fig:non-conv_random} report cumulative performance profiles (instances solved versus wall-clock time) for the convex and non-convex problems, respectively.
Table~\ref{tab:solver-performance} summarizes solver outcomes, including timeouts and cases where a solver returned an infeasibility declaration (although the problems were created that ensured that they are feasible) or an ``Objective mismatch'' solution (these are treated as incorrect and excluded from the profile curves).
A solution was classified as ``Objective mismatch'' if the solver, without reaching the time limit, reported an optimal solution with an objective value relatively larger by more than \(10^{-4}\) compared to those obtained by other reformulations, despite the solver’s optimality tolerance being set to \(10^{-6}\). 
The threshold of \(10^{-4}\) was chosen to ensure that such cases represent definitively incorrect solutions rather than results only marginally exceeding the specified tolerance.

Objective mismatch arises from an incorrect (overly large) reported lower bound, which can cause the solver to certify a suboptimal incumbent as optimal.
To illustrate this, Table~\ref{tab:obj-mismatch-all} reports the objective and solver-reported lower bound on one random convex quadratic instance.
As shown in Table~\ref{tab:obj-mismatch-all}, solvers sometimes reported matching lower and upper bounds; however, other reformulations or solvers with the same reformulation reported significantly better solutions.

\begin{figure}[htbp]
  \centering
  \begin{subfigure}[t]{0.32\textwidth}
    \centering
    \includegraphics[width=\linewidth]{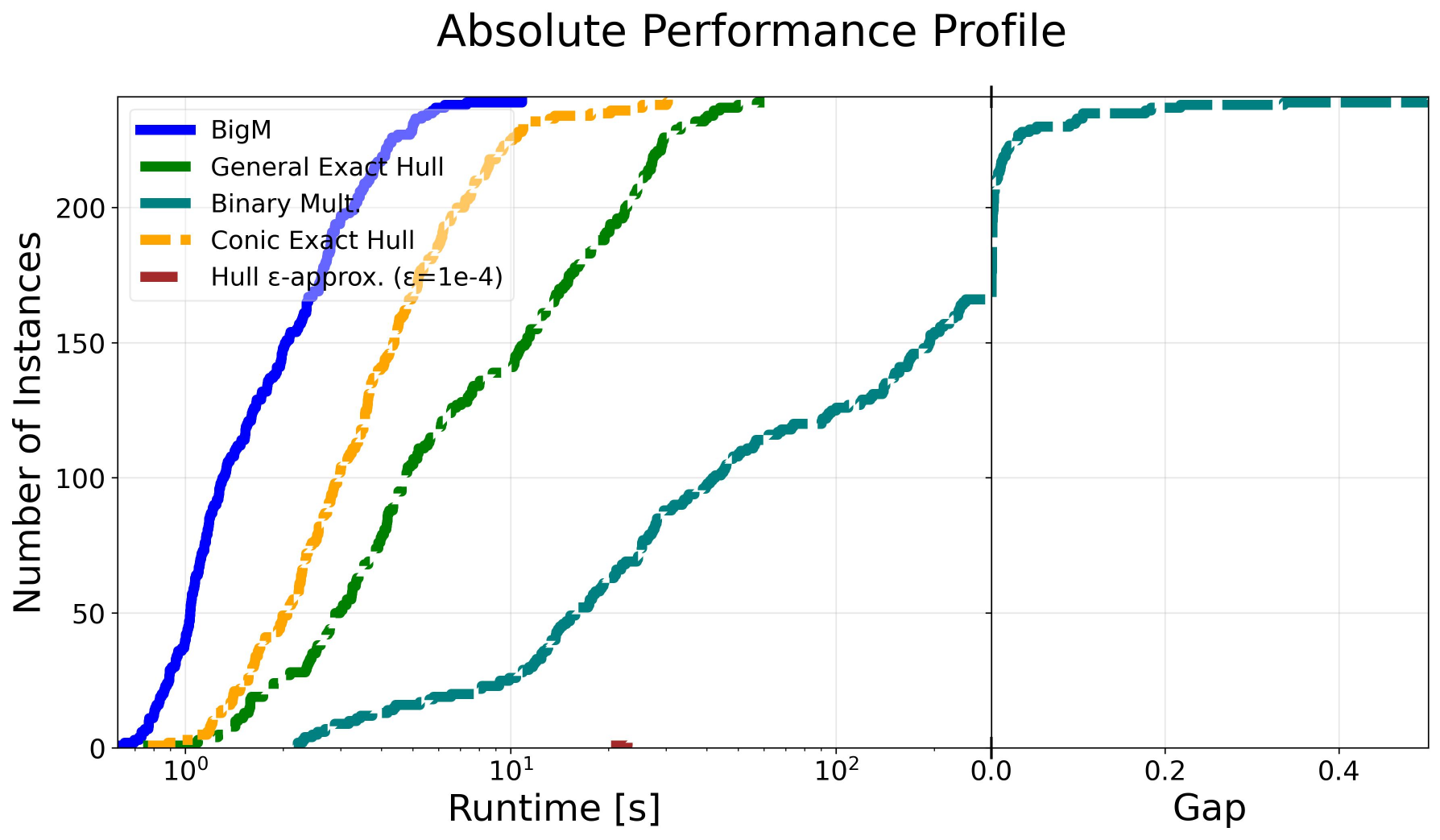}
    \caption{\texttt{Gurobi}}
    \label{fig:conv_random-gurobi}
  \end{subfigure}%
  \hspace{0.01\textwidth}
  \begin{subfigure}[t]{0.32\textwidth}
    \centering
    \includegraphics[width=\linewidth]{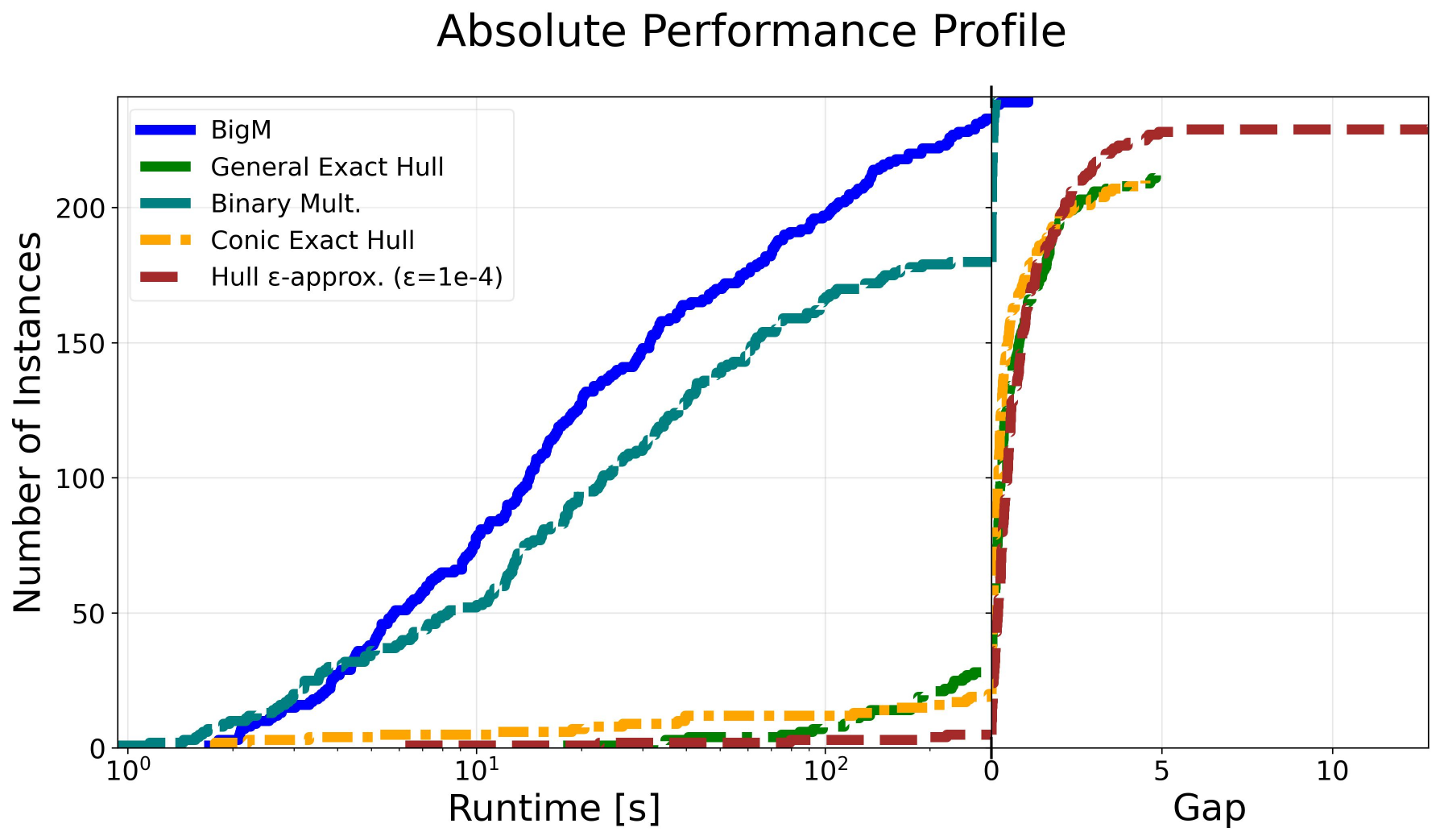}
    \caption{\texttt{BARON}}
    \label{fig:conv_random-baron}
  \end{subfigure}%
  \hspace{0.01\textwidth}
  \begin{subfigure}[t]{0.32\textwidth}
    \centering
    \includegraphics[width=\linewidth]{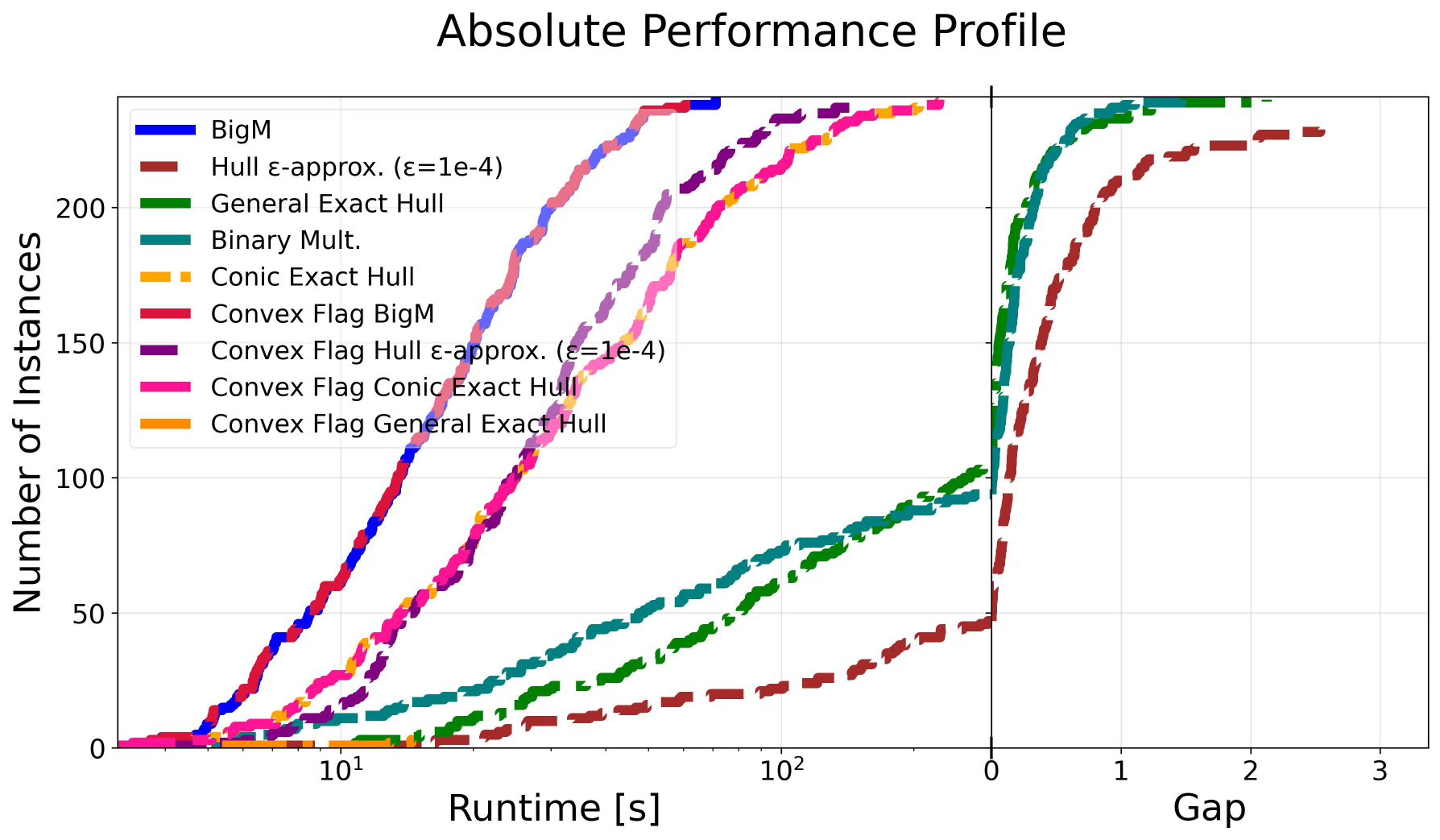}
    \caption{\texttt{SCIP}}
    \label{fig:conv_random-scip}
  \end{subfigure}
  \caption{Cumulative number of instances solved versus solution time for different reformulations applied to random convex quadratically constrained GDP problems. Results are shown separately for \texttt{Gurobi}, \texttt{BARON}, and \texttt{SCIP}. Each curve represents a distinct reformulation strategy.}
  \label{fig:conv_random}
\end{figure}

\begin{figure}[htbp]
  \centering
  \begin{subfigure}[t]{0.32\textwidth}
    \centering
    \includegraphics[width=\linewidth]{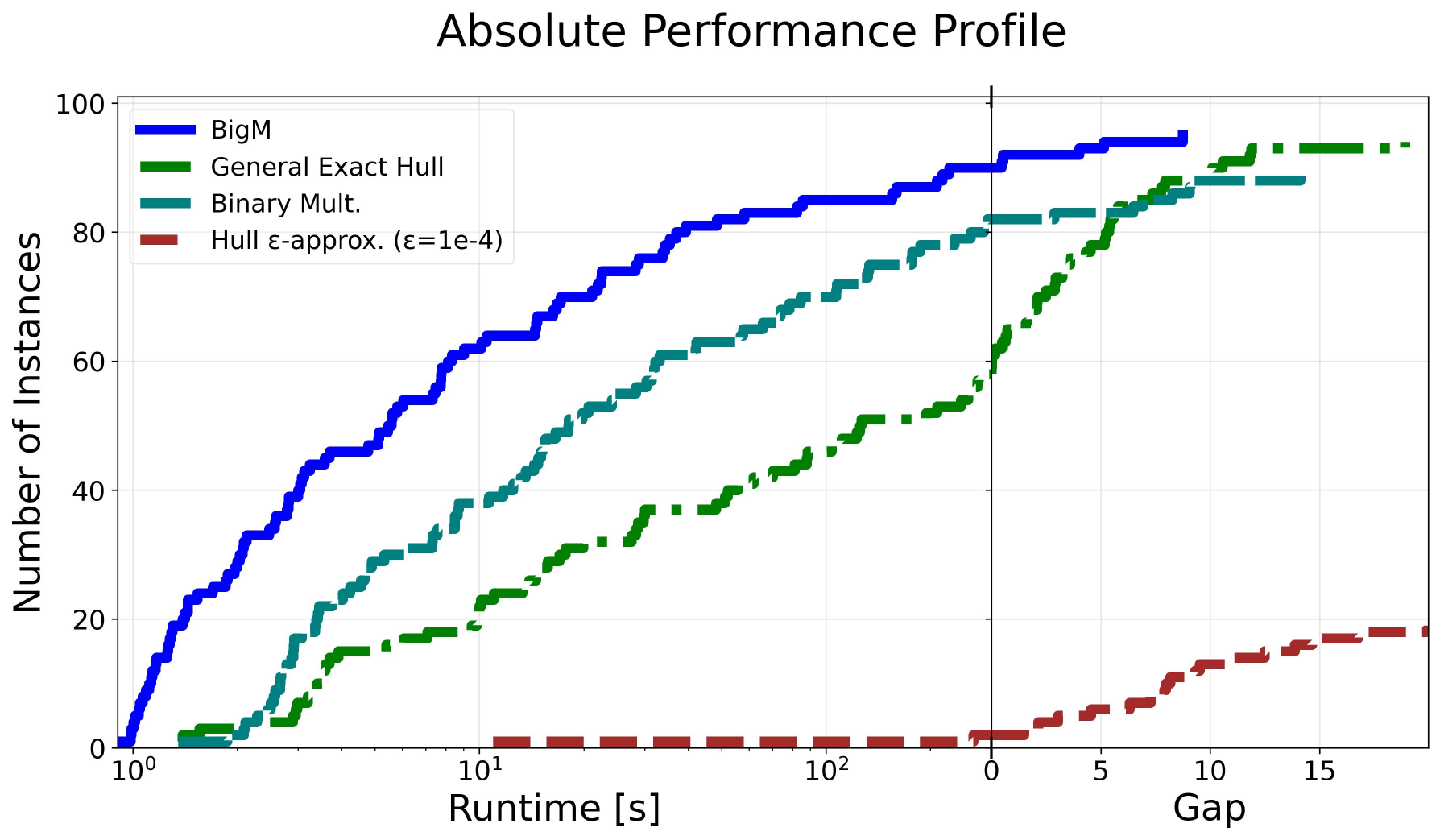}
    \caption{\texttt{Gurobi}}
    \label{fig:nonconv-gurobi}
  \end{subfigure}%
  \hspace{0.01\textwidth}
  \begin{subfigure}[t]{0.32\textwidth}
    \centering
    \includegraphics[width=\linewidth]{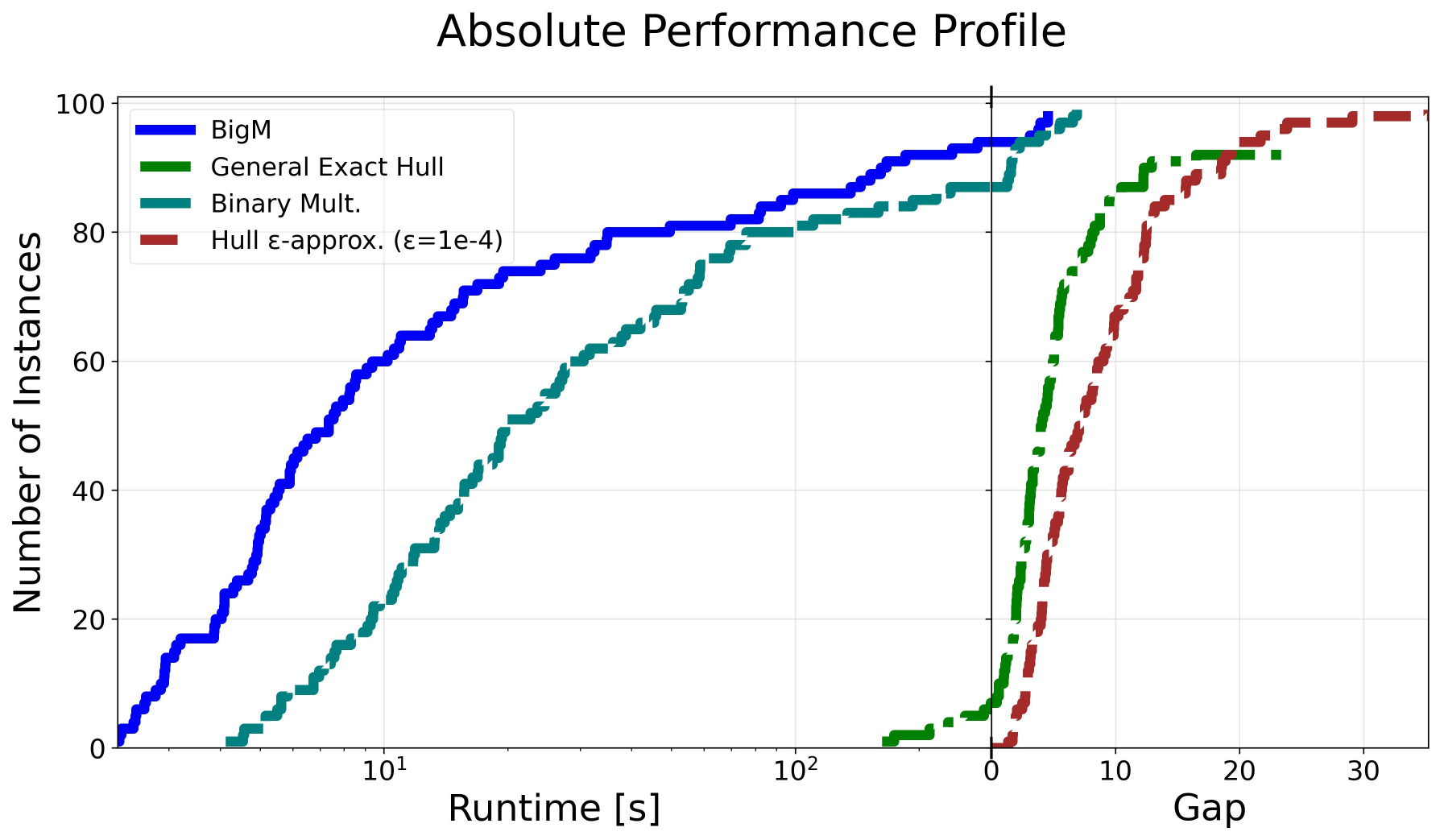}
    \caption{\texttt{BARON}}
    \label{fig:nonconv-baron}
  \end{subfigure}%
  \hspace{0.01\textwidth}
  \begin{subfigure}[t]{0.32\textwidth}
    \centering
    \includegraphics[width=\linewidth]{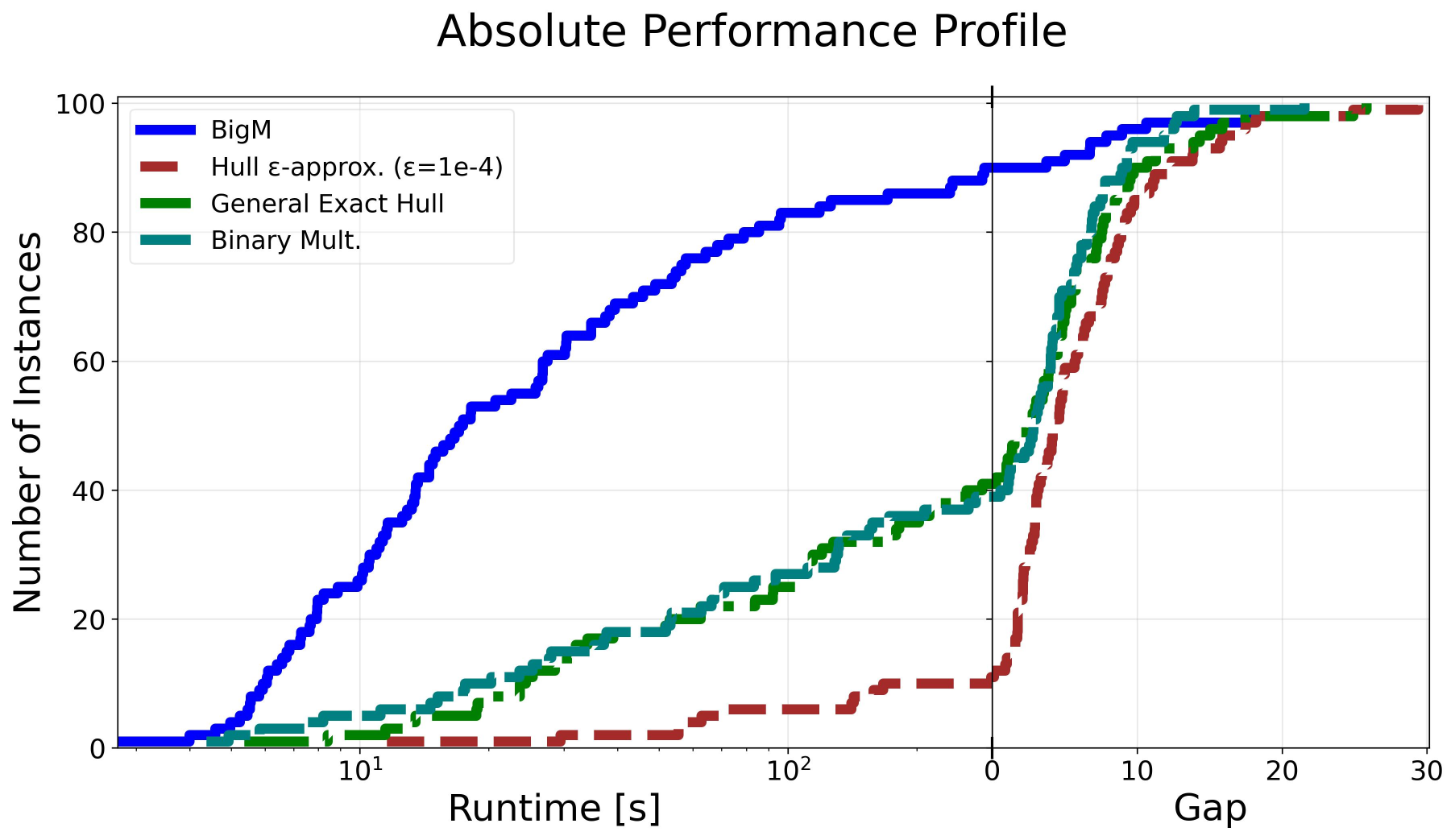}
    \caption{\texttt{SCIP}}
    \label{fig:nonconv-scip}
  \end{subfigure}
  \caption{Cumulative number of instances solved versus solution time for different reformulations applied to random non-convex quadratically constrained GDP problems. Results are shown separately for \texttt{Gurobi}, \texttt{BARON}, and \texttt{SCIP}. Each curve represents a distinct reformulation strategy.}
  \label{fig:non-conv_random}
\end{figure}


\begin{table}[t]
\centering
\caption{Solver outcomes across reformulation strategies in convex and non-convex cases for random quadratic GDPs}
\label{tab:solver-performance}

\begin{subtable}[t]{0.49\textwidth}
\centering
\footnotesize
\caption{Convex case}
\label{tab:outcomes-convex}
\resizebox{\linewidth}{!}{%
\begin{tabular}{lccccc}
\toprule
\textbf{Strategy} & \textbf{Optimal} & \textbf{Timeout} & \textbf{Infeasible} & \textbf{\shortstack{Objective \\ mismatch}} & \textbf{Total} \\
\midrule
\multicolumn{6}{l}{\textbf{Solver: \texttt{Gurobi}}} \\
BigM & 240 & 0   & 0   & 0  & 240 \\
Hull $\varepsilon$-approx.\ ($\varepsilon=10^{-4}$) & 2   & 0   & 200 & 38 & 240 \\
General Exact Hull & 240 & 0   & 0   & 0  & 240 \\
CEHR & 240 & 0   & 0   & 0  & 240 \\
Binary Mult. & 167 & 73  & 0   & 0  & 240 \\
\midrule
\multicolumn{6}{l}{\textbf{Solver: \texttt{BARON}}} \\
BigM & 233 & 7   & 0   & 0  & 240 \\
Hull $\varepsilon$-approx.\ ($\varepsilon=10^{-4}$) & 5   & 233 & 0   & 2  & 240 \\
GEHR & 29  & 183 & 0   & 28 & 240 \\
CEHR & 20  & 190 & 0   & 30 & 240 \\
Binary Mult. & 180 & 60  & 0   & 0  & 240 \\
\midrule
\multicolumn{6}{l}{\textbf{Solver: \texttt{SCIP}}} \\
BigM & 240 & 0   & 0   & 0  & 240 \\
Hull $\varepsilon$-approx.\ ($\varepsilon=10^{-4}$) & 47  & 182 & 0   & 11 & 240 \\
GEHR & 104 & 136 & 0   & 0  & 240 \\
CEHR & 240 & 0   & 0   & 0  & 240 \\
Binary Mult. & 94  & 146 & 0   & 0  & 240 \\
\midrule
\multicolumn{6}{l}{\textbf{Solver: \texttt{SCIP} (convex flag)}} \\
BigM & 240 & 0 & 0   & 0  & 240 \\
Hull $\varepsilon$-approx.\ ($\varepsilon=10^{-4}$) & 238 & 0 & 0   & 2  & 240 \\
GEHR & 4   & 0 & 217 & 19 & 240 \\
CEHR & 240 & 0 & 0   & 0  & 240 \\
Binary Mult. & 0   & 0 & 219 & 21 & 240 \\
\bottomrule
\end{tabular}}
\end{subtable}
\hfill
\begin{subtable}[t]{0.49\textwidth}
\centering
\footnotesize
\caption{Non-convex case}
\label{tab:outcomes-nonconvex}
\resizebox{\linewidth}{!}{%
\begin{tabular}{lccccc}
\toprule
\textbf{Strategy} & \textbf{Optimal} & \textbf{Timeout} & \textbf{Infeasible} & \textbf{\shortstack{Objective \\ mismatch}} & \textbf{Total} \\
\midrule
\multicolumn{6}{l}{\textbf{Solver: \texttt{Gurobi}}} \\
BigM & 91 & 9  & 0 & 0  & 100 \\
Hull $\varepsilon$-approx.\ ($\varepsilon=10^{-4}$) & 2  & 86 & 0 & 12 & 100 \\
GEHR & 59 & 41 & 0 & 0  & 100 \\
Binary Mult. & 83 & 17 & 0 & 0  & 100 \\
\midrule
\multicolumn{6}{l}{\textbf{Solver: \texttt{BARON}}} \\
BigM & 95 & 4  & 0 & 1  & 100 \\
Hull $\varepsilon$-approx.\ ($\varepsilon=10^{-4}$) & 0  & 99 & 1 & 0  & 100 \\
GEHR & 7  & 86 & 0 & 7  & 100 \\
Binary Mult. & 88 & 12 & 0 & 0  & 100 \\
\midrule
\multicolumn{6}{l}{\textbf{Solver: \texttt{SCIP}}} \\
BigM & 91 & 9  & 0 & 0  & 100 \\
Hull $\varepsilon$-approx.\ ($\varepsilon=10^{-4}$) & 11 & 89 & 0 & 0  & 100 \\
GEHR & 41 & 59 & 0 & 0  & 100 \\
Binary Mult. & 39 & 61 & 0 & 0  & 100 \\
\bottomrule
\end{tabular}}
\end{subtable}

\end{table}

\begin{table}[t]
\centering
\caption{Objective value, solver-reported lower bound, and wall-clock time for all reformulations on a random convex quadratic instance with parameters 
\( |K| = 3,\; |D_k| = 10,\; \; n = 4 \). Objective mismatch corresponds to cases where a solver returns \texttt{optimal}, but the reported lower bound is inconsistent with the best objective values obtained by other reformulations on the same instance, indicating an invalid lower bound.}
\label{tab:obj-mismatch-all}
\footnotesize
\setlength{\tabcolsep}{6pt}
\begin{tabular}{lrrr}
\toprule
\textbf{Strategy} & \textbf{Time (s)} & \textbf{Objective} & \textbf{Lower bound} \\
\midrule
\multicolumn{4}{l}{\textbf{Solver: \texttt{Gurobi}}} \\
BigM & 0.7864  & -1.503530467 & -1.503531622 \\
Hull $\varepsilon$-approx.\ ($\varepsilon=10^{-4}$) & 29.6530 & -1.209200146 & -1.209200146 \\
GEHR & 1.2564  & -1.503530467 & -1.503530784 \\
CEHR & 1.3155  & -1.503530107 & -1.503530472 \\
Binary Mult. & 3.3902  & -1.503531523 & -1.503536828 \\
\midrule
\multicolumn{4}{l}{\textbf{Solver: \texttt{BARON}}} \\
BigM & 1.8185  & -1.503530473 & -1.503530415 \\
Hull $\varepsilon$-approx.\ ($\varepsilon=10^{-4}$) & 301.7905 & -1.382069355 & -2.172538773 \\
GEHR & 196.1422 & -1.439272090 & -1.439272090 \\
CEHR & 279.6740 & -1.451448855 & -1.451448859 \\
Binary Mult. & 4.9766  & -1.503530473 & -1.503530490 \\
\midrule
\multicolumn{4}{l}{\textbf{Solver: \texttt{SCIP}}} \\
BigM & 4.9693  & -1.503530165 & -1.503530994 \\
Hull $\varepsilon$-approx.\ ($\varepsilon=10^{-4}$) & 4.7891  & -1.070302464 & -1.070302474 \\
GEHR & 27.2644 & -1.503530585 & -1.503531213 \\
CEHR & 8.4097  & -1.503530345 & -1.503530886 \\
Binary Mult. & 23.7996 & -1.503530282 & -1.503531068 \\
\midrule
\multicolumn{4}{l}{\textbf{Solver: \texttt{SCIP} (convex flag)}} \\
BigM & 4.9459  & -1.503530165 & -1.503530994 \\
Hull $\varepsilon$-approx.\ ($\varepsilon=10^{-4}$) & 4.0870  & -1.503530309 & -1.503530995 \\
GEHR & 0.9600  & n/a & $\infty$ \\
CEHR & 8.3809  & -1.503530345 & -1.503530886 \\
Binary Mult. & 0.8000  & n/a & $\infty$ \\
\bottomrule
\end{tabular}
\end{table}

Failures to produce a correct solution, whether by incorrectly declaring the problem infeasible or by reporting a solution as optimal with an incorrect lower bound, are attributed to numerical instabilities, which can arguably be particularly pronounced when using the  \(\varepsilon\)-approximation.

\subsubsection{Convex instances}

As can be seen from the results on the convex benchmark set presented in Figure~\ref{fig:conv_random} and Table~\ref{tab:outcomes-convex}, CEHR is, in most cases, the fastest and most reliable HR approach across solvers, with the clearest gains for solvers that recognize conic structure (\texttt{Gurobi} and \texttt{SCIP}).
This outcome is consistent with the formulation's design: it preserves the exact hull relaxation while expressing the perspective term via a rotated-cone-representable constraint, enabling the solver to certify and exploit convexity directly.

In \texttt{BARON}, all hull formulations performed relatively poorly overall. Nevertheless, CEHR and GEHR showed similar performance and were notably better than the \(\varepsilon\)-approximation in terms of runtime.
Both exact formulations exhibited an increased number of ``Objective mismatch'' outcomes; however, interpreting numerical robustness for the \(\varepsilon\)-approximation is difficult in \texttt{BARON} because it timed out on 233 of 240 instances in the case of $\varepsilon$-approximation, leaving too few completed runs to support a meaningful comparison.

Finally, GEHR substantially improved solution time and numerical robustness relative to the \(\varepsilon\)-approximation across all solvers, but it was significantly outperformed by CEHR for solvers with conic recognition (\texttt{Gurobi} and \texttt{SCIP}).

To isolate the role of convexity recognition, we additionally tested the convex random instances with \texttt{SCIP} under a setting that forces \texttt{SCIP} to assume that all constraints are mixed-integer convex functions (reported as ``convex flag'' results in Table~\ref{tab:outcomes-convex}).

When \texttt{SCIP} is guided to treat the constraints as mixed-integer convex, the $\varepsilon$-approximation hull behaves much more like CEHR in terms of speed and reliability.
This suggests that only with explicit guidance, the branch-and-bound performance of these two formulations is comparable, and they target essentially the same convex relaxation.
However, without explicit guidance, the solvers tested here did \emph{not} reliably recognize or exploit the convexity of the $\varepsilon$-approximation constraints on these instances, which largely explains its substantially worse performance relative to CEHR.

Additionally, the fact that most solutions were incorrect for GEHR under the convex-flag setting underscores its main limitation.
Although the feasible region of the relaxation is convex, the reformulated inequality is not necessarily convex as a function of the lifted variables.
If the solver is forced to treat these constraints as convex, it may generate invalid cuts and return incorrect solutions. Therefore, declaring GEHR as convex is generally inappropriate, even when the underlying relaxation forms a convex feasible region.

Based on these differences in \texttt{SCIP}, we infer that the performance of the $\varepsilon$-approximation could improve substantially when solvers correctly recognize and exploit the convexity of the resulting constraints. This highlights a key drawback of the $\varepsilon$-approximation: it complicates the functional structure of constraints and can hinder convexity detection, thereby preventing solvers from fully leveraging convex relaxations.

In practice, it is often impractical, and sometimes impossible, to supply annotations for constraints that arise after GDP reformulation, particularly when solvers do not provide the necessary support (especially if only some constraints are convex).
For this reason, we restrict the convexity-annotation experiment to the random convex benchmark set. For the remaining benchmarks, we report results under the default solver behavior, which better reflects typical use cases in which users either do not provide convexity flags or cannot.

\subsubsection{Non-convex instances}
On the non-convex GDPs, CEHR is not applicable, so we focus on comparing GEHR with the $\varepsilon$-approximation. 
As shown in Figures~\ref{fig:non-conv_random} and Table~\ref{tab:outcomes-nonconvex}, overall, GEHR outperforms the $\varepsilon$-approximation, yielding shorter solution times across all solvers.

In terms of numerical robustness, the improvement is most pronounced for \texttt{Gurobi}: GEHR produced no incorrect solutions, whereas the $\varepsilon$-approximation returned 12 incorrect solutions.
For \texttt{BARON}, GEHR exhibited more solver failures, but once again, this comparison is not very informative because the $\varepsilon$-approximation timed out on 99 out of 100 instances, leaving almost no runs completed within the time limit to assess its numerical stability.
For \texttt{SCIP}, no reformulation exhibited numerical issues on this benchmark set.

\subsection{Continuously Stirred Tank Reactor network problem (non-convex GDP)}

Consider the following superstructure optimization problem involving a network of Continuously Stirred Tank Reactors (CSTRs).
The objective is to minimize the total volume of CSTRs in series, in which the autocatalytic reaction \(A + B \rightarrow 2B\) takes place.
The superstructure allows a recycle stream originating from the last reactor to be returned to any preceding reactor within the series.
An illustration of this problem is provided in Figure~\ref{fig:CSTR}.
The details and formulations of the problem were previously described in the literature \cite{ovalleLogicBasedDiscreteSteepestDescent2025} and are summarized in Appendix \ref{app:cstr}.

Originally, the constraints defining this optimization problem were linear or cubic.
Since any polynomial constraint can be systematically represented using quadratic constraints, for the purposes of this work, cubic constraints are reformulated into equivalent quadratic forms, using quadratic equality constraints.
Consequently, due to the presence of quadratic equality constraints, the resulting GDP formulation is inherently non-convex.
Therefore, this problem allows for testing reformulations in an application, where constraints are linear and non-convex quadratic.

\begin{figure}[htbp]
  \centering
  \includegraphics[width=0.8\textwidth]{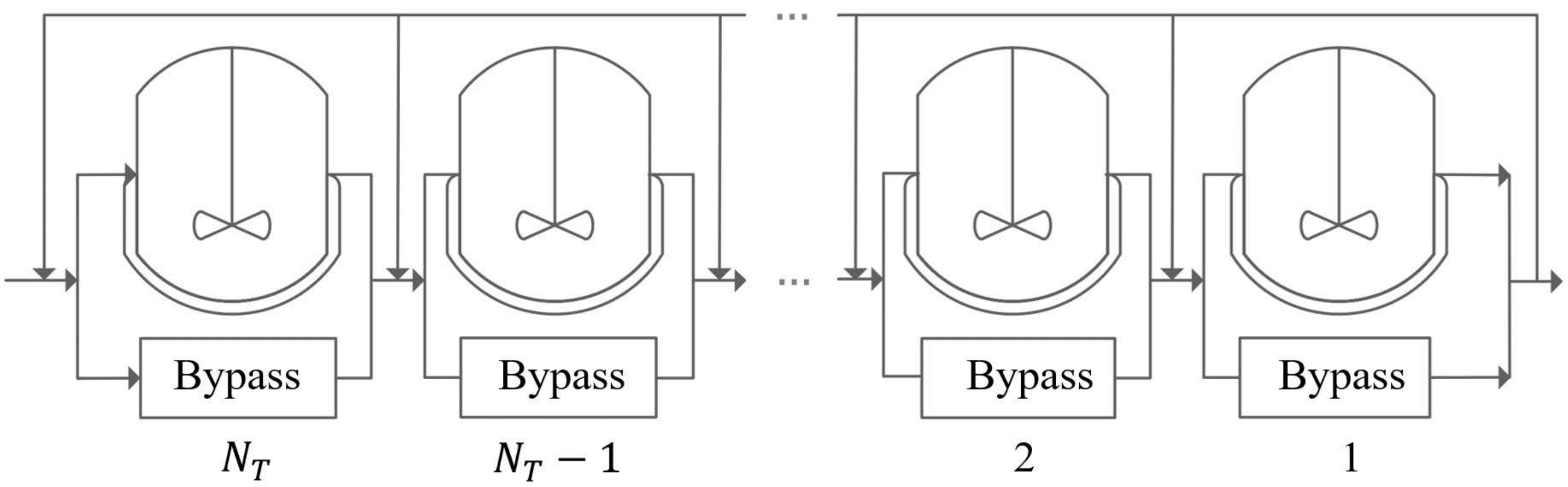}  
  \caption{Superstructure representation for the continuously stirred tank reactor (CSTR) network problem. 
The configuration includes \(N_T\) reactor stages arranged in series, where each stage can either be a reactor or a bypass.
The flow proceeds from left to right, and a recycle stream from the last reactor is allowed to return to any upstream stage.
This flexible structure enables the optimizer to choose the optimal number of reactors and the location of recycle to minimize the total reactor volume \cite{ovalleLogicBasedDiscreteSteepestDescent2025}.}

  \label{fig:CSTR}
\end{figure}

To vary problem size and difficulty, we changed the maximum allowable number of reactor stages while keeping all other model parameters fixed.
Figure~\ref{fig:combined} summarizes the results across reformulations and solvers.
No numerical issues were observed for this benchmark across the tested solvers and reformulations, so we omit a table of solution outcomes.

\begin{figure}[htbp]
  \centering
  \begin{subfigure}[t]{0.32\textwidth}
    \centering
    \includegraphics[width=\linewidth]{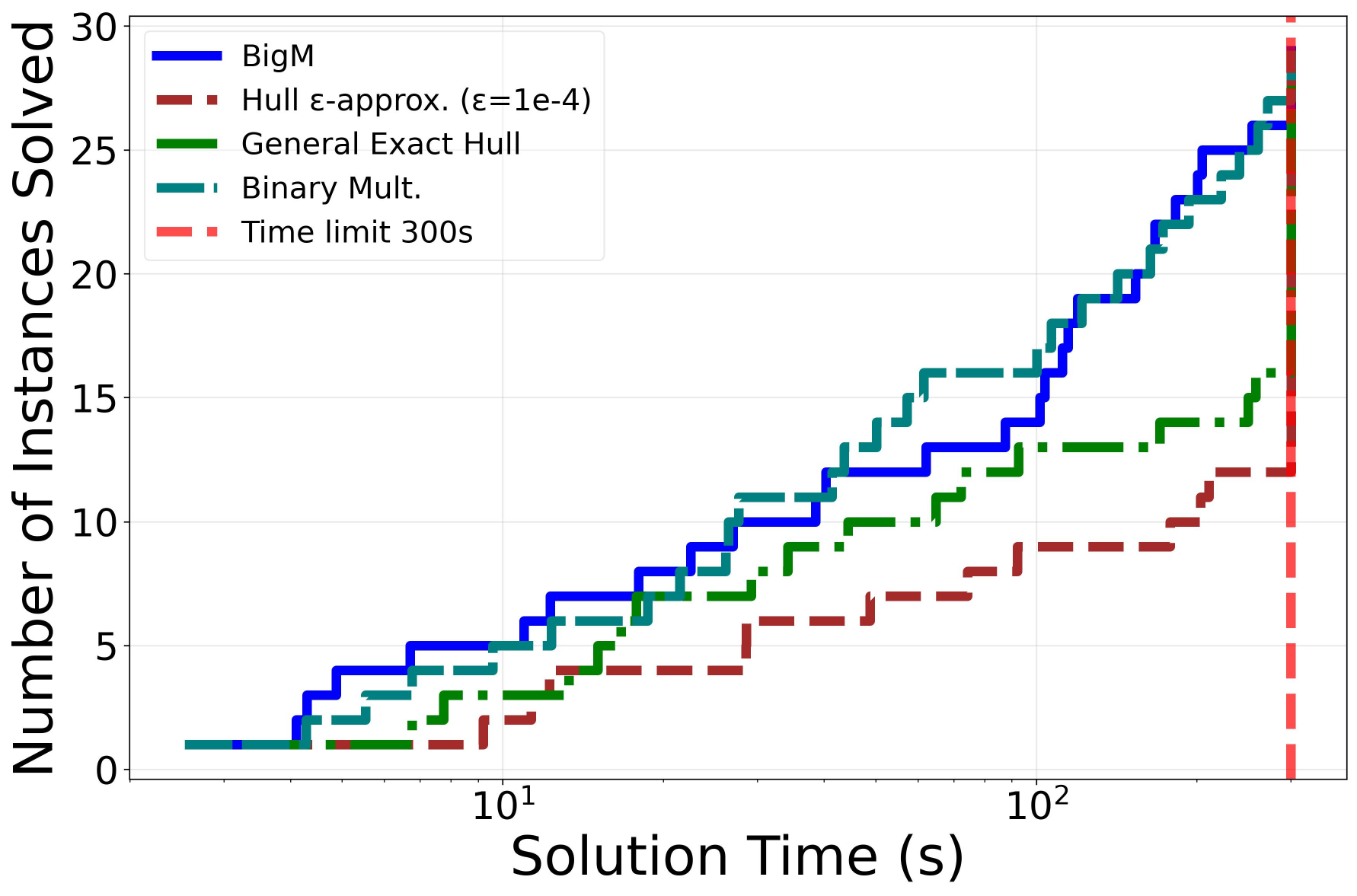}
    \caption{\texttt{Gurobi}}
    \label{fig:cstr-gurobi}
  \end{subfigure}%
  \hspace{0.01\textwidth}
  \begin{subfigure}[t]{0.32\textwidth}
    \centering
    \includegraphics[width=\linewidth]{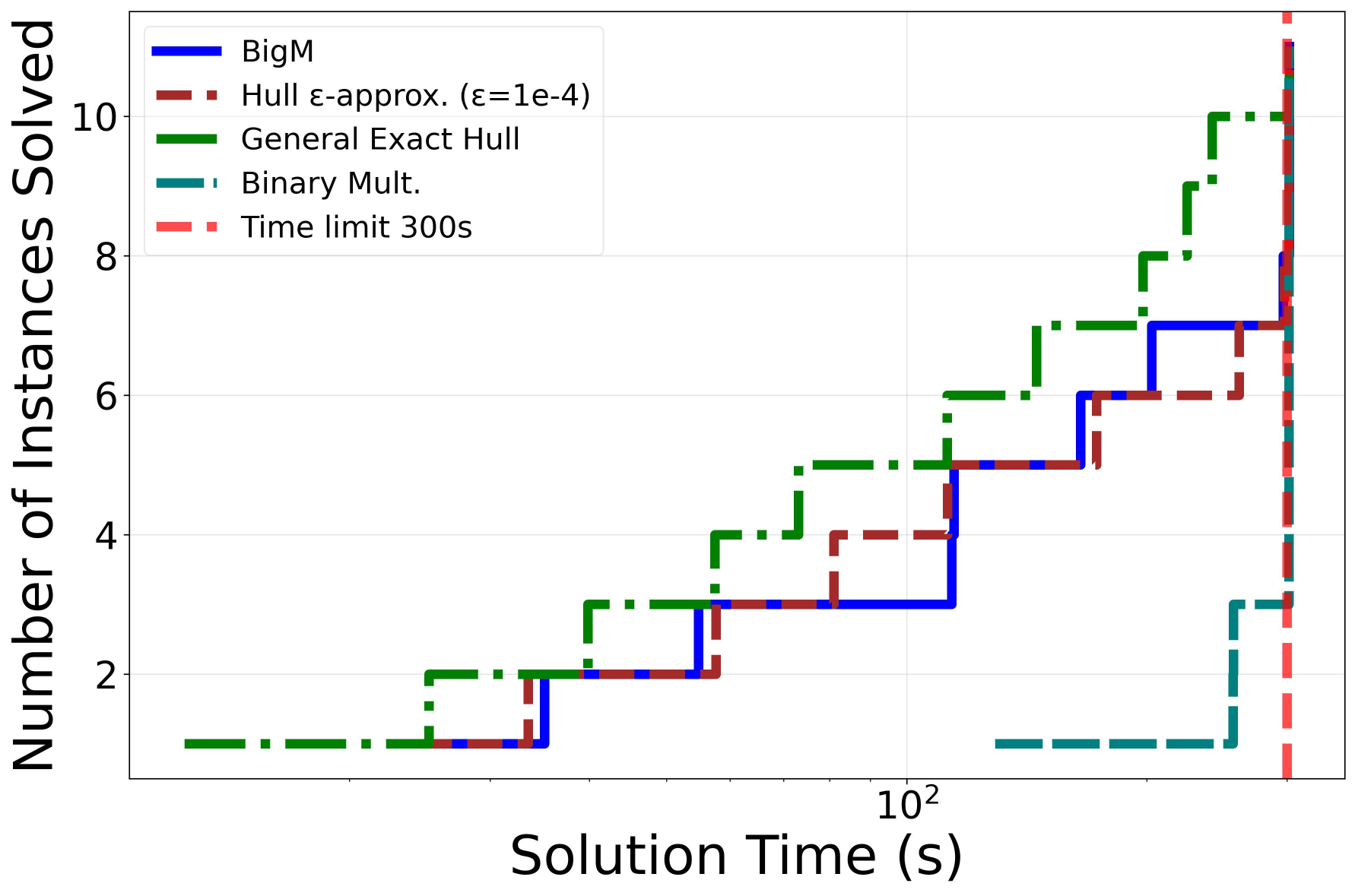}
    \caption{\texttt{BARON}}
    \label{fig:cstr-baron}
  \end{subfigure}%
  \hspace{0.01\textwidth}
  \begin{subfigure}[t]{0.32\textwidth}
    \centering
    \includegraphics[width=\linewidth]{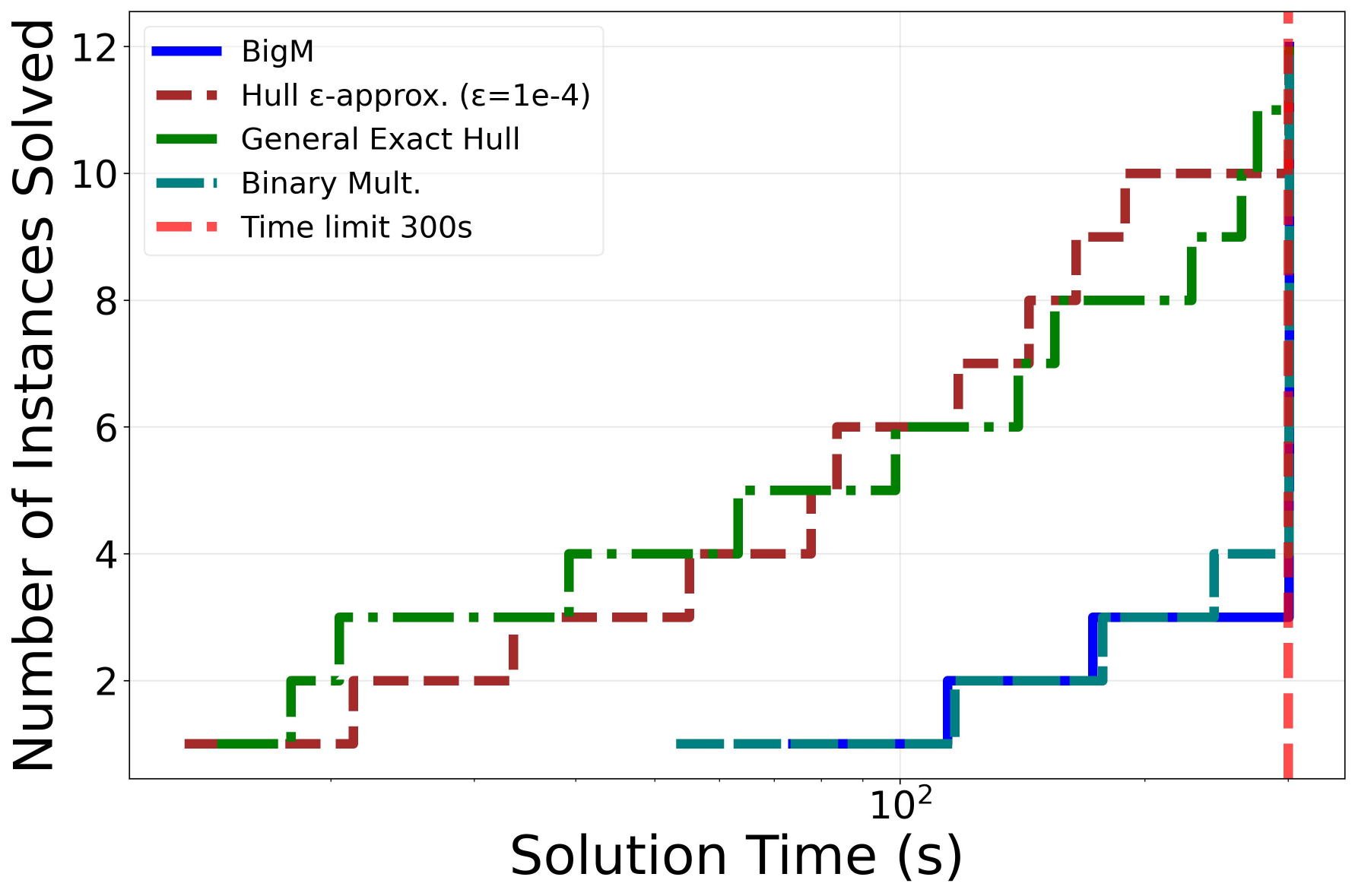}
    \caption{\texttt{SCIP}}
    \label{fig:cstr-scip}
  \end{subfigure}
  \caption{Cumulative number of instances solved versus solution time for different reformulations applied to the CSTR network problem formulated as GDP. Results are shown separately for \texttt{Gurobi}, \texttt{BARON}, and \texttt{SCIP}. Each curve represents a distinct reformulation strategy.}
  \label{fig:combined}
\end{figure}

GEHR performed similarly to the \(\varepsilon\)-approximation in \texttt{SCIP} and was solved clearly faster in \texttt{Gurobi} and \texttt{BARON}, demonstrating an overall advantage for GEHR.

Although a detailed comparison between Big-M and HRs for non-convex GDPs is outside the scope of this paper, it is worth noting that, for \texttt{SCIP}, both HRs were substantially faster than the Big-M reformulation.
For \texttt{BARON}, the \(\varepsilon\)-approximation performed similarly to Big-M, while GEHR was clearly faster.
These results demonstrate that a systematic study of Big-M versus HRs applied to non-convex GDPs could be a valuable direction for future work.

\subsection{k-Means Clustering (convex GDP)}

The $k$-means clustering problem is a classic unsupervised learning task in which $|N|$ data points in $|D|$ dimensions are partitioned into $|K|$ clusters in a way that minimizes the total squared Euclidean distance from each point to its assigned cluster center.
Although it is typically solved using heuristics, the problem can also be formulated as a convex GDP with linear and quadratic constraints~\cite{papageorgiouPseudoBasicSteps2018} and solved to global optimality with established reformulation and optimization techniques.
This makes it a suitable benchmark problem within the GDP framework \cite{bernalneiraConvexMixedintegerNonlinear2024, kronqvistStepsIntermediateRelaxations2021}.

The GDP formulation of the $k$-means clustering problem can be expressed as follows:

\begingroup
\begin{equation}
\setlength{\jot}{0pt}
\begin{aligned}
\min_{\mathbf{c}, \mathbf{d}, Y} \quad & \sum_{i \in N} d_i \\
\text{s.t.} \quad & c_{k-1,1} \leq c_{k,1}, \quad k \in \{2, \ldots, |K|\} \\
& \bigvee_{k \in K} \left[ Y_{ik} \land \left( d_i \geq \sum_{j \in D} (p_{ij} - c_{kj})^2 \right) \right], \quad i \in N \\
& \underset{k \in K}{\underline{\bigvee}} Y_{ik}, \quad i \in N \\
& \mathbf{d} \in \mathbb{R}_+^{|N|} \\
& \mathbf{c} \in \mathbb{R}^{|K| \times |D|} \\
& Y_{ik} \in \{ \text{False}, \text{True} \}, \quad i \in N, k \in K
\end{aligned}
\label{eq:k-means-GDP}
\end{equation}
\endgroup

Here, \( p_{ij} \) is the coordinate of point \( i \) in dimension \( j \), \( c_{kj} \) is the coordinate of cluster center \( k \) in dimension \( j \), and \( d_i \) is the squared distance from point \( i \) to its assigned center.
The constraint \( c_{k-1,1} \leq c_{k,1} \) is a symmetry-breaking condition that ensures consistent cluster ordering.

We generated 96 instances of this problem with parameters in the following ranges: \( 3 \leq |K| \leq 5 \), \( 10 \leq |N| \leq 17 \), and \( 2 \leq |D| \leq 5 \).
The coordinates of the points were generated randomly by sampling a uniform distribution in the interval \([-1, 1]\).

Figure~\ref{fig:k-means} presents cumulative performance profiles, extended by the information on the absolute optimality gap. 
A summary of the solver outcomes is provided in Table~\ref{tab:kmeans-solver-performance}.

Consistent with the random benchmark results, CEHR delivered the best overall performance on the $k$-means clustering instances across all tested solvers.
In addition, no numerical issues were observed for CEHR in any solver.

For \texttt{Gurobi} and \texttt{SCIP}, GEHR performed nearly identically to CEHR and significantly better than $\varepsilon$-approximation.
This indicates that, for this problem class, the solvers treat the two formulations similarly.
A likely explanation is that, unlike the dense random instances, GEHR constraints arising in $k$-means are recognized as convex or conic-representable without explicit lifting, rather than being handled as generic non-convex quadratic expressions.

For \texttt{BARON}, the $\varepsilon$-approximation achieved slightly faster solve times than GEHR on the subset of instances where both reformulations produced correct solutions, although CEHR remained substantially superior overall.
However, GEHR was markedly more numerically robust than the $\varepsilon$-approximation in \texttt{BARON}: the $\varepsilon$-approximation produced 66 incorrect outcomes (infeasible declarations or incorrect optimal solutions), compared to 13 for GEHR.

\begin{figure}[htbp]
  \centering
  \begin{subfigure}[t]{0.32\textwidth}
    \centering
    \includegraphics[width=\linewidth]{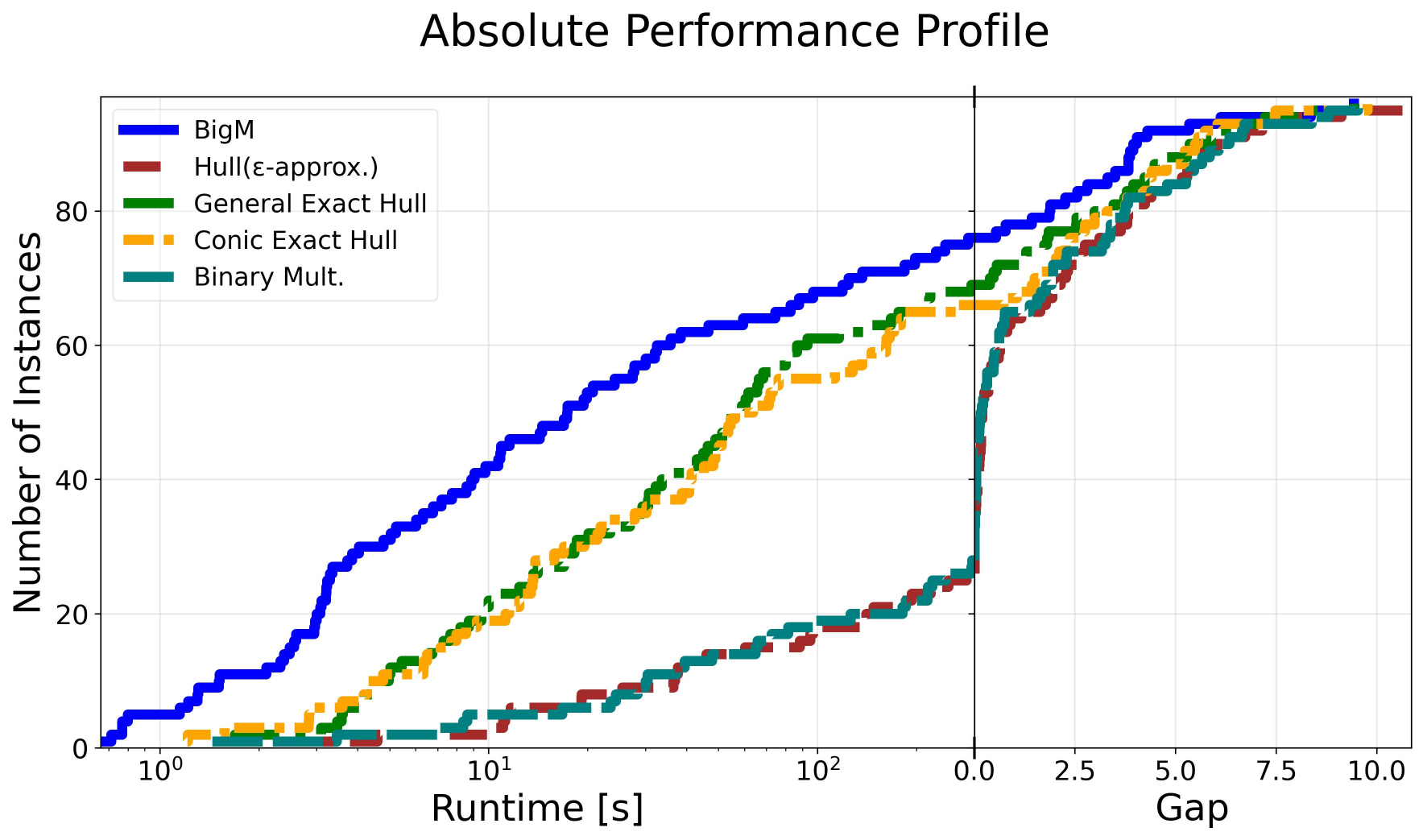}
    \caption{\texttt{Gurobi}}
    \label{fig:k-means-gurobi}
  \end{subfigure}%
  \hspace{0.01\textwidth}
  \begin{subfigure}[t]{0.32\textwidth}
    \centering
    \includegraphics[width=\linewidth]{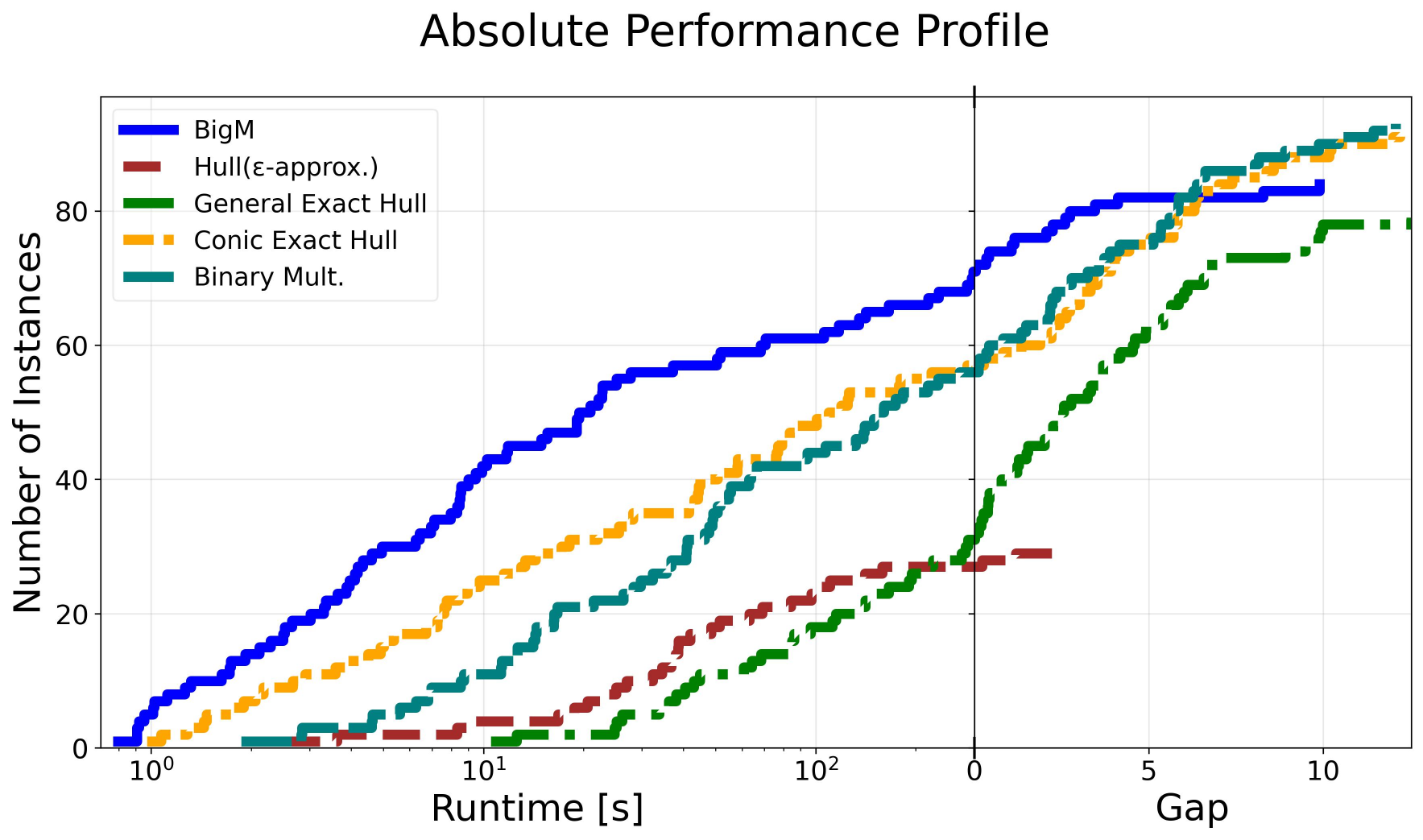}
    \caption{ \texttt{BARON}}
    \label{fig:k-means-baron}
  \end{subfigure}%
  \hspace{0.01\textwidth}
  \begin{subfigure}[t]{0.32\textwidth}
    \centering
    \includegraphics[width=\linewidth]{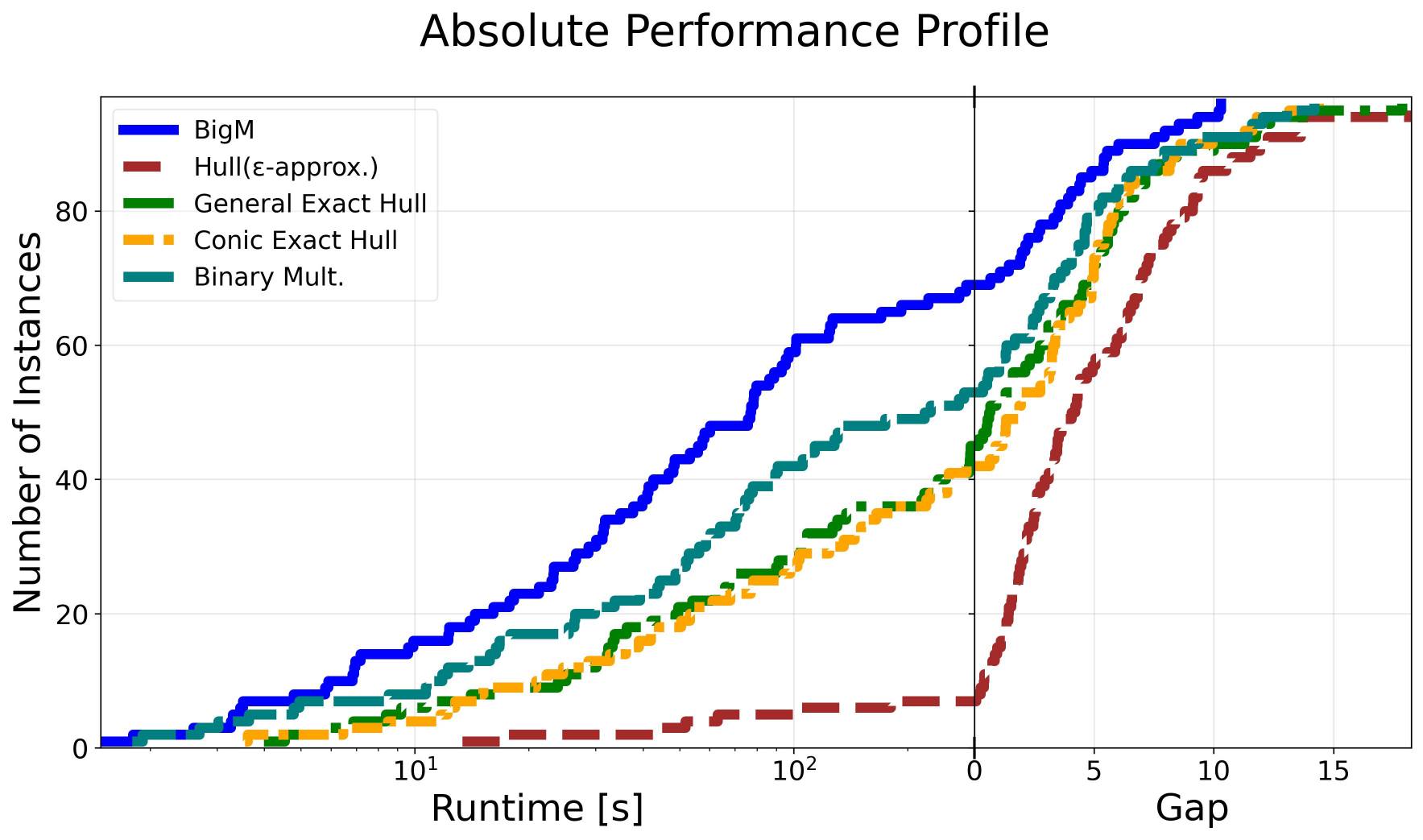}
    \caption{\texttt{SCIP}}
    \label{fig:k-means-scip}
  \end{subfigure}
  \caption{Cumulative number of instances solved versus solution time for different reformulations applied to $k$-means clustering problem. For instances that reached the time limit, cumulative plots are continued with values of absolute optimality gap. Results are shown separately for \texttt{Gurobi}, \texttt{BARON}, and \texttt{SCIP}. Each curve represents a distinct reformulation strategy.}
  \label{fig:k-means}
\end{figure}

\begin{table}[htbp]
\centering
\caption{Solver outcomes across reformulation strategies for the $k$-means clustering problem}
\label{tab:kmeans-solver-performance}
\setlength{\tabcolsep}{4pt}

\begin{minipage}[t]{0.7\textwidth}
\centering
\small
\begin{tabular}{lccccc}
\toprule
\textbf{Strategy} & \textbf{Optimal} & \textbf{Timeout} & \textbf{Infeasible} & \textbf{\shortstack{Objective \\ mismatch}}& \textbf{Total} \\
\midrule
\multicolumn{6}{l}{\textbf{Solver: \texttt{Gurobi}}} \\
BigM                & 76 & 20 & 0  & 0  & 96 \\
Hull($\varepsilon$-approx.) & 26 & 70 & 0  & 0  & 96 \\
GEHR  & 69 & 27 & 0  & 0  & 96 \\
CEHR    & 66 & 30 & 0  & 0  & 96 \\
Binary Mult.        & 28 & 68 & 0  & 0  & 96 \\
\midrule
\multicolumn{6}{l}{\textbf{Solver: \texttt{BARON}}} \\
BigM                & 83 & 13 & 0  & 0  & 96 \\
Hull($\varepsilon$-approx.) & 27 & 3  & 18 & 48 & 96 \\
GEHR  & 35 & 48 & 0  & 13 & 96 \\
CEHR    & 61 & 35 & 0  & 0  & 96 \\
Binary Mult.        & 59 & 37 & 0  & 0  & 96 \\
\midrule
\multicolumn{6}{l}{\textbf{Solver: \texttt{SCIP}}} \\
BigM                & 69 & 27 & 0 & 0 & 96 \\
Hull($\varepsilon$-approx.) & 7  & 88 & 0 & 1 & 96 \\
GEHR  & 45 & 51 & 0 & 0 & 96 \\
CEHR    & 42 & 54 & 0 & 0 & 96 \\
Binary Mult.        & 53 & 43 & 0 & 0 & 96 \\
\bottomrule
\end{tabular}
\end{minipage}
\end{table}

\subsection{Constrained layout problem (convex GDP)}

The constrained layout involves placing a set of non-overlapping rectangular components within a bounded area to minimize a given objective, the total cost of connections, while satisfying geometric and containment constraints.
This problem can be modeled as a convex GDP with both linear and quadratic constraints and is frequently used as a benchmark in GDP literature \cite{ruizHierarchyRelaxationsNonlinear2012,furmanComputationallyUsefulAlgebraic2020, bernalneiraConvexMixedintegerNonlinear2024}.

In the version of the problem considered here, rectangular 2D components must be positioned within one of several enclosing circular regions, such that the rectangles do not overlap and the total pairwise connection cost is minimized.
The objective function employs either an \(\ell_1\) or \(\ell_2\) norm to penalize distances between rectangles.

The GDP formulation can be expressed as follows:

\begingroup
\begin{equation}
\label{eq:clay}
\begin{aligned}
\min_{\delta_x, \delta_y, \mathbf{x}, \mathbf{y}, \mathbf{W}, \mathbf{Y}} \quad
& \sum_{i,j \in N,\, i < j} c_{ij}\,\|(\delta_{x_{ij}}, \delta_{y_{ij}})\|_p
\\[0.3em]
\text{s.t.} \quad
& \delta_{x_{ij}} \ge x_i - x_j,
\qquad \forall\, i,j \in N,\ i<j,
\\
& \delta_{x_{ij}} \ge x_j - x_i,
\qquad \forall\, i,j \in N,\ i<j,
\\
& \delta_{y_{ij}} \ge y_i - y_j,
\qquad \forall\, i,j \in N,\ i<j,
\\
& \delta_{y_{ij}} \ge y_j - y_i,
\qquad \forall\, i,j \in N,\ i<j,
\\[0.3em]
&
\left[
\begin{array}{c}
Y_{ij}^1 \\
x_i + L_i/2 \le x_j - L_j/2
\end{array}
\right]
\vee
\left[
\begin{array}{c}
Y_{ij}^2 \\
x_j + L_j/2 \le x_i - L_i/2
\end{array}
\right]
\\
& \vee
\left[
\begin{array}{c}
Y_{ij}^3 \\
y_i + H_i/2 \le y_j - H_j/2
\end{array}
\right]
\vee
\left[
\begin{array}{c}
Y_{ij}^4 \\
y_j + H_j/2 \le y_i - H_i/2
\end{array}
\right],
\\
& \hspace{2em} \forall\, i,j \in N,\ i<j,
\\[0.3em]
&
\bigvee_{t \in T}
\left[
\begin{array}{c}
W_{it} \\
(x_i + L_i/2 - x_{ct})^2 + (y_i + H_i/2 - y_{ct})^2 \le r_t^2 \\
(x_i + L_i/2 - x_{ct})^2 + (y_i - H_i/2 - y_{ct})^2 \le r_t^2 \\
(x_i - L_i/2 - x_{ct})^2 + (y_i + H_i/2 - y_{ct})^2 \le r_t^2 \\
(x_i - L_i/2 - x_{ct})^2 + (y_i - H_i/2 - y_{ct})^2 \le r_t^2
\end{array}
\right],
\\
& \hspace{2em} \forall\, i \in N,
\\[0.3em]
& Y_{ij}^1 \underline{\vee} Y_{ij}^2 \underline{\vee} Y_{ij}^3 \underline{\vee} Y_{ij}^4,
\qquad \forall\, i,j \in N,\ i<j,
\\
& \underset{t \in T}{\underline{\bigvee}} W_{it},
\qquad \forall\, i \in N,
\\
& x_i^l \le x_i \le x_i^u,
\qquad \forall\, i \in N,
\\
& y_i^l \le y_i \le y_i^u,
\qquad \forall\, i \in N,
\\
& \delta_{x_{ij}}, \delta_{y_{ij}} \in \mathbb{R}_+,
\qquad \forall\, i,j \in N,\ i<j,
\\
& x_i, y_i \in \mathbb{R},
\qquad \forall\, i \in N,
\\
& Y_{ij}^1, Y_{ij}^2, Y_{ij}^3, Y_{ij}^4 \in \{\mathrm{False}, \mathrm{True}\},
\qquad \forall\, i,j \in N,\ i<j,
\\
& W_{it} \in \{\mathrm{False}, \mathrm{True}\},
\qquad \forall\, i \in N,\ t \in T.
\end{aligned}
\end{equation}
\endgroup

Here, \(N\) is the set of rectangles. The variables \(x_i\) and \(y_i\) denote the coordinates of the center of rectangle \(i\), while \(\delta_{x_{ij}}\) and \(\delta_{y_{ij}}\) represent the horizontal and vertical distances between rectangles \(i\) and \(j\).
The parameters \(L_i\) and \(H_i\) denote the width and height of rectangle \(i\), respectively.

The set \(T\) contains all enclosing circles, each with center coordinates \((x_{ct}, y_{ct})\) and radius \(r_t\).
The binary variables \(Y_{ij}^t\) encode the relative position of rectangles \(i\) and \(j\) to enforce non-overlap, while \(W_{it}\) indicates whether rectangle \(i\) is assigned to circle \(t\).
The coefficient \(c_{ij}\) is the cost of the distance between connected rectangle pairs in the objective.

The objective function minimizes the total pairwise connection cost based on the \(\ell_p\)-norm of the distance between rectangle centers, where \(p \in \{1,2\}\).
Therefore, depending on the value of \(p\), the objective function takes the form:

\[
\min \sum_{i,j \in N,\, i < j} c_{ij} \cdot 
\begin{cases}
\delta_{x_{ij}} + \delta_{y_{ij}} & \text{if } p = 1 \\
\sqrt{\delta_{x_{ij}}^2 + \delta_{y_{ij}}^2} & \text{if } p = 2
\end{cases}
\]

To assess reformulation performance, we tested six instances of this problem implemented in the \texttt{Pyomo.GDP} example library\footnote{\url{https://github.com/Pyomo/pyomo/tree/main/examples/gdp/constrained_layout}} using both \(\ell_1\) and \(\ell_2\) norm.
A summary of results is provided in Figures~\ref{fig:layout_profiles} and Table~\ref{tab:layout-solver-performance}.

While the limited size and number of benchmark problems prevent drawing general conclusions, the observed results suggest that, overall, both CEHR and GEHR performed on par with the \(\varepsilon\)-approximation for constrained layout problems.

\begin{figure}[htbp]
  \centering
  \begin{subfigure}[t]{0.32\textwidth}
    \centering
    \includegraphics[width=\linewidth]{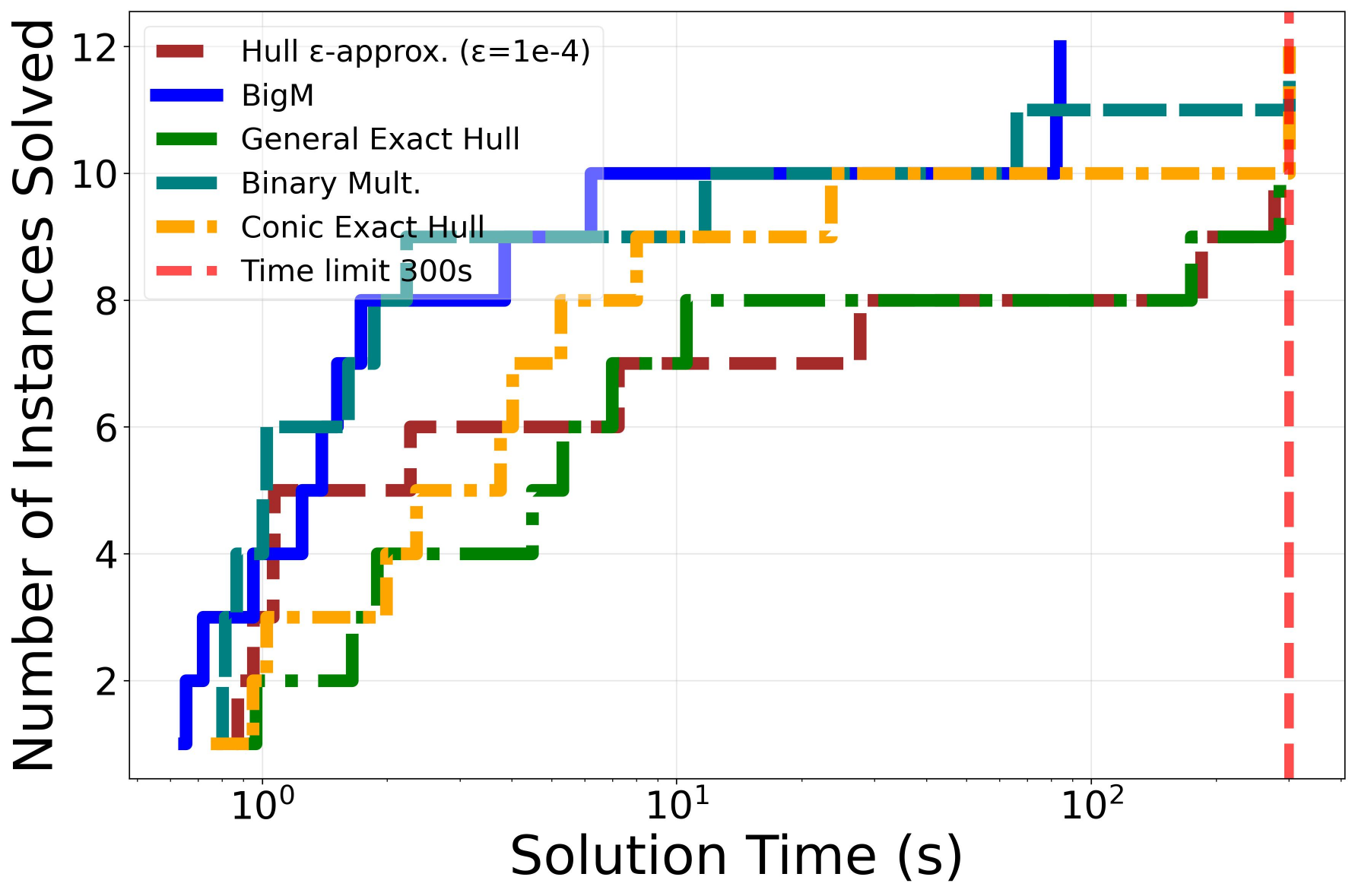}
    \caption{\texttt{Gurobi}}
    \label{fig:layout-gurobi}
  \end{subfigure}%
  \hspace{0.01\textwidth}
  \begin{subfigure}[t]{0.32\textwidth}
    \centering
    \includegraphics[width=\linewidth]{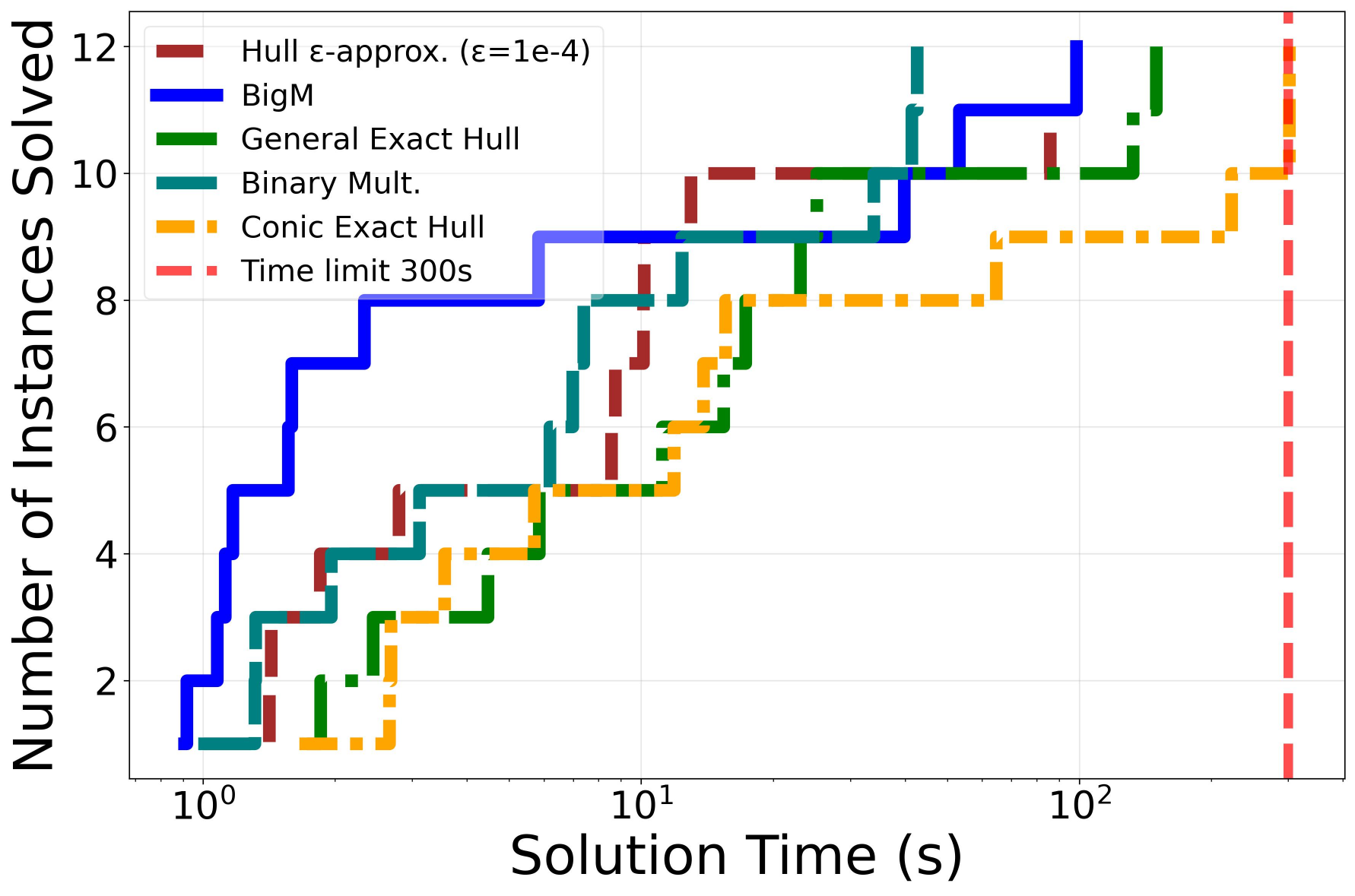}
    \caption{\texttt{BARON}}
    \label{fig:layout-baron}
  \end{subfigure}%
  \hspace{0.01\textwidth}
  \begin{subfigure}[t]{0.32\textwidth}
    \centering
    \includegraphics[width=\linewidth]{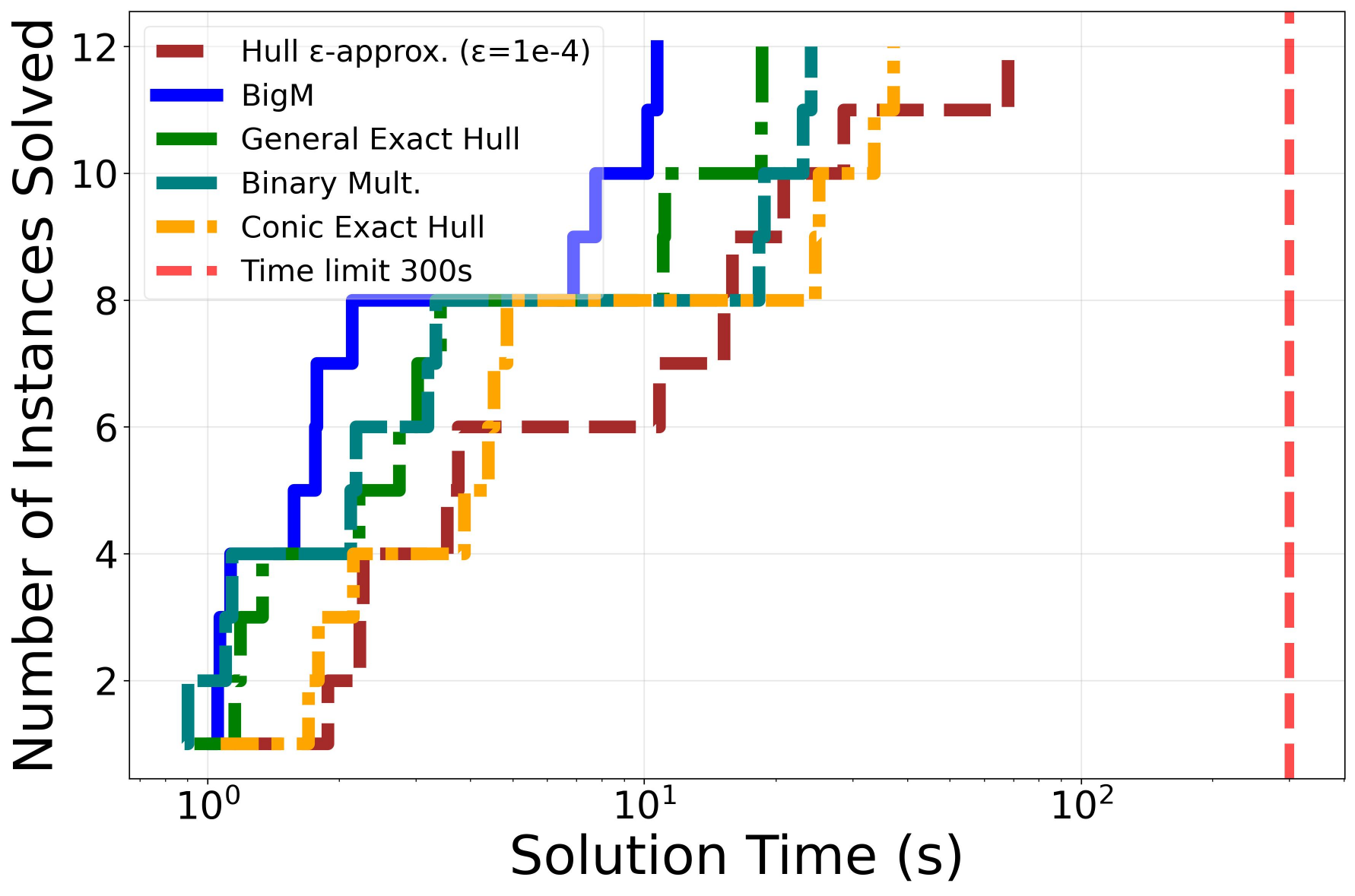}
    \caption{\texttt{SCIP}}
    \label{fig:layout-scip}
  \end{subfigure}
  \caption{Cumulative number of constrained layout problem instances solved versus solution time, comparing the \(\varepsilon\)-approximation and the proposed exact Hull reformulation. Results are shown separately for \texttt{Gurobi}, \texttt{BARON}, and \texttt{SCIP}. Each curve represents a reformulation strategy, highlighting their relative performance in terms of solution speed and number of instances solved.}
  \label{fig:layout_profiles}
\end{figure}


\begin{table}[htbp]
\centering
\caption{Solver outcomes across reformulation strategies for the constrained layout problem}
\label{tab:layout-solver-performance}
\setlength{\tabcolsep}{4pt}

\begin{minipage}[t]{0.7\textwidth}
\centering
\small
\begin{tabular}{lcccc}
\toprule
\textbf{Strategy} & \textbf{Optimal} & \textbf{Timeout} & \textbf{\shortstack{Objective \\ mismatch}} & \textbf{Total} \\
\midrule
\multicolumn{5}{l}{\textbf{Solver: \texttt{Gurobi}}} \\
BigM                         & 12 & 0 & 0 & 12 \\
Hull $\varepsilon$-approx.\ ($\varepsilon=10^{-4}$) & 10 & 0 & 2 & 12 \\
GEHR           & 10 & 0 & 2 & 12 \\
CEHR             & 10 & 2 & 0 & 12 \\
Binary Mult.                 & 11 & 1 & 0 & 12 \\
\midrule
\multicolumn{5}{l}{\textbf{Solver: \texttt{BARON}}} \\
BigM                         & 12 & 0 & 0 & 12 \\
Hull $\varepsilon$-approx.\ ($\varepsilon=10^{-4}$) & 11 & 0 & 1 & 12 \\
GEHR           & 12 & 0 & 0 & 12 \\
CEHR             & 10 & 2 & 0 & 12 \\
Binary Mult.                 & 12 & 0 & 0 & 12 \\
\midrule
\multicolumn{5}{l}{\textbf{Solver: \texttt{SCIP}}} \\
BigM                         & 12 & 0 & 0 & 12 \\
Hull $\varepsilon$-approx.\ ($\varepsilon=10^{-4}$) & 12 & 0 & 0 & 12 \\
GEHR           & 12 & 0 & 0 & 12 \\
CEHR             & 12 & 0 & 0 & 12 \\
Binary Mult.                 & 12 & 0 & 0 & 12 \\
\bottomrule
\end{tabular}
\end{minipage}
\end{table}

\section{Conclusions}
\label{sec:conclusion}

This paper investigated HR for GDPs with quadratic disjunctive constraints, with the goal of avoiding the convexity recognition, numerical and relaxation-quality drawbacks introduced by the standard $\varepsilon$-approximation of the closure of the perspective function.

\paragraph{Hull Reformulation for non-convex GDPs.}
HR is a valid reformulation technique even when a GDP contains non-convex constraints. The reformulation preserves the intended on/off behavior at binary values, recovering the original constraints when $y=1$ and rendering them redundant when $y=0$.

In the non-convex setting, however, the continuous relaxation produced by HR is not guaranteed to be convex. Moreover, while HR is designed to represent the convex hull of the feasible region of each individual disjunction when the constraints in the disjuncts are convex, this interpretation does not hold when the disjunctive constraints are non-convex. Consequently, the resulting relaxation should not be interpreted as producing the convex hull of the disjunction.

As a result, HR of non-convex GDP is naturally targeted to global MINLP solvers, which may additionally construct convex underestimators internally when computing bounds.

\paragraph{General Exact Hull Reformulation for quadratic constraints.}
For general quadratic constraints, possibly non-convex, we showed GEHR that preserves quadratic degree and avoids the fractional expressions introduced by the $\varepsilon$-approximation.
Also, because the reformulation is exact, it avoids feasible-region enlargement associated with $\varepsilon>0$ and removes the need to tune an approximation parameter.

\paragraph{Conic Exact Hull Reformulation for convex quadratics.}
For convex quadratic disjunctive constraints ($Q\succeq 0$), we showed derivation of CEHR, which is expressed by a rotated second-order cone representable inequality of the form $\mathbf v^\top Q \mathbf v \le t\,y$, together with a linear inequality linking $(t,\mathbf v,y)$.
This provides a solver-friendly reformulation for convex quadratics with a conic-representable structure, enabling direct exploitation by conic-capable solvers.

\paragraph{Summary of computational findings.}
Across random convex and non-convex quadratic GDPs, convex $k$-means clustering, a non-convex CSTR network benchmark, and constrained layout instances, the proposed exact formulations reduced numerical failures relative to the $\varepsilon$-approximation and frequently improved or had comparable runtime.

In our experiments, CEHR was consistently the most reliable hull formulation on convex benchmarks and achieved the best overall performance, especially for solvers that recognize and exploit conic structure. 
Its substantial advantage over the $\varepsilon$-approximation is largely attributable to the fact that the $\varepsilon$-approximation can obscure convexity and hinder convexity detection by the solver.

On both convex and non-convex benchmarks, GEHR was more numerically stable and typically faster than the $\varepsilon$-approximation unless it was a convex GDP and a solver was explicitly informed about convexity.

\paragraph{Implications for modeling.}
Overall, our results support using CEHR for convex quadratically constrained GDPs and GEHR for non-convex instances over $\varepsilon$-approximation when HR is applied to GDP.

\paragraph{Future work.}
This work motivates several directions for future research. First, the polynomial generalization in Appendix~\ref{app:poly-hull} should be benchmarked systematically to understand when preserving polynomial degree yields practical gains over alternative lifting strategies. 
Second, for non-convex GDPs, it would be valuable to compare direct application of hull reformulations with approaches that convexify disjunctive constraints prior to reformulation, and to benchmark hull reformulations against other formulations, such as Big-M, when applied directly to non-convex GDPs.
Finally, improving solver-side convexity detection and handling of perspective-type (i.e., $\varepsilon$-approximation) constraints remains an important opportunity for future solver development.

\section*{Acknowledgments}

The authors acknowledge and express sincere gratitude to the Davidson School of Chemical Engineering at Purdue University for providing a supportive research environment and funding.

The authors thank the GAMS Development Corporation for providing a personal license and access to the \texttt{BARON} solver for our group.

The authors thank \texttt{Gurobi} Optimization for providing access to software through their academic licensing program.

The authors thank Prof. Can Li for providing access to the hardware.

The authors thank Kevin Furman,  Nicolas Sawaya, Ignacio Grossmann,  Nikolaos Sahinidis for their helpful feedback provided through discussions of the preprint and poster presentations.

This material is based upon work supported by the National Science Foundation under Award No. 2430617.

\begin{appendices}
\normalsize

\vspace{-2ex}
\section{From Binary Multiplication to the Hull Reformulation}
\label{app:bm-to-hull}

We outline how HR for a single disjunction can be systematically derived from the simpler Binary Multiplication reformulation by taking a convex combination of disjuncts.
This derivation does not rely on convexity assumptions, demonstrating that HR is applicable to both convex and non-convex GDPs. 
It also offers additional insight into the interpretation of the reformulation when applied to non-convex problems.
This appendix is adapted from Appendix A of Lee and Grossmann~\cite{leeNewAlgorithmsNonlinear2000}.

\vspace{5pt}
Consider one disjunction
\begingroup
\begin{equation*}
\begin{aligned}
    & \bigvee_{D}
        \begin{bmatrix} 
            Y_{i} \\
            h_{i}(\mathbf{x}) \leq 0 
        \end{bmatrix}\\
    & \underset{i \in D}{\underline{\bigvee}} Y_{i}\\
\end{aligned}
\label{eq:Disjunction}
\end{equation*}
\endgroup

\noindent where every $h_i:\mathbb R^n\to\mathbb R$ is proper (defined and finite on its domain) and continuously
extended at the boundary of the feasible region, and where continuous variables are bounded:
$\mathbf x^{\ell}\le \mathbf x\le \mathbf x^{u}$.

\subsection{Binary Multiplication reformulation}
The Binary Multiplication reformulation of the disjunction, shown above, activates the constraints by multiplying it with the corresponding binary~$y_i$:

\begingroup
\begin{equation*}
\begin{aligned}
y_i\,h_i(\mathbf x) &\le 0 &&\forall\; i\in D,
\\
\sum_{i\in D} y_i &= 1,
& y_i &\in\{0,1\}\quad \forall\; i\in D.
\end{aligned}
\end{equation*}
\endgroup

Because $y_i=1$ restores $h_i(\mathbf x)\le 0$ and $y_i=0$  makes the constraint trivially satisfied, the reformulation is exact for binary~$y$, and therefore valid.
However, its continuous relaxation ($y_i\in[0,1]$) provides the same feasible region as the MINLP problem,  when projected to $\mathbf{x}$, hence it offers no relaxation advantages and the continuous relaxation problem is equivalent to the original.
\subsection{Taking a convex combination}
Let $\mathbf x_i$ be a copy of the original variable associated with the feasible region of disjunct $i$. We now form a convex combination of the disjuncts, which yields the convex hull of the disjunction if all $h_i(\mathbf x)$ are convex.
Note, however, that we don't have to assume convexity and convex combinations can also be taken for non-convex sets, but the resulting set will not, in general, be convex, which is illustrated in Figure~\ref{fig:epsilon_approx}.
Thus, these steps remain valid for non-convex problems, although they no longer guarantee the convex hull of the disjunction.
The convex combination of disjuncts can then be defined as follows:

\begingroup
\begin{equation*}
\begin{aligned}
& \mathbf{x} = \sum_{i \in D}y_i\mathbf{x}_i, \\
& \sum_{i \in D}y_i = 1, \\
& y_i\,h_i(\mathbf x_i) \le 0 &&\forall\; i\in D,\\
& y_i\in[0,1]
\end{aligned}
\end{equation*}
\endgroup

\subsection{Hull Reformulation}

Let  $\mathbf v_i=y_i\mathbf x_i\in\mathbb R^n$ be a disaggregated copy, scaled by~$y_i$. Because $\mathbf x_i$ is bounded by $\mathbf x^{\ell}\le \mathbf x_i\le \mathbf x^{u}$, $\mathbf v_i$ is bounded by $\mathbf x^{\ell}y_i\le \mathbf v_i\le \mathbf x^{u}y_i$.

The constraint function \( y_i\, h_i(\mathbf{x}_i) \) that defines the feasible region of each disjunct reduces to zero when \( y_i = 0 \), since \( h_i(\mathbf{x}_i) \) is assumed to be finite. When \( y_i > 0 \), we can apply the change of variables \( \mathbf{x}_i = \frac{\mathbf{v}_i}{y_i} \). Also it should be noted that $ \mathbf{x} = \sum_{i \in D}y_i\mathbf{x}_i = \sum_{i \in D}\mathbf{v}_i$

Collecting all constraints yields HR of the disjunction:

\begingroup
\begin{equation*}
\begin{aligned}
    & \mathbf{x} = \sum_{i \in D_k} \mathbf{v}_{i},\\
    & \left( \text{cl} \, \widetilde{h}_{i} \right)(\mathbf{v}_{i}, y_{i}) \leq 0, i \in D \\
    & \mathbf x^{\ell} y_{i} \leq \mathbf{v}_{i} \leq \mathbf{x}^u y_{i} \\
    & \mathbf{v}_{i} \in \mathbb{R}^n, i \in D\\
\end{aligned}
\end{equation*}
\endgroup

where in the domain of reformulation
\begingroup
\begin{equation*}
    \left( \text{cl} \, \widetilde{h} \right)(\mathbf{v}, y) = 
\begin{cases} 
    y h\left(\frac{\mathbf{v}}{y}\right) & \text{if } y > 0 \\[5pt]
    0  & \text{if } y = 0 \\[1pt]
\end{cases}
\end{equation*}
\endgroup

Thus, HR can be systematically obtained from Binary Multiplication by taking a convex combination of disjuncts in lifted space.
Note that convexity is not assumed at any step, and for non-convex function $h_i(\mathbf{x})$, the construction remains valid, but the continuous relaxation is no longer guaranteed to be convex.

\section{Extension of the Exact Hull Reformulation to Polynomial Constraints of Degree $d$}
\label{app:poly-hull}

Consider a polynomial disjunctive constraint \(h(\mathbf x) \leq 0\), such that:
\[
    h(\mathbf x)\;=\;\sum_{k=0}^{d} p_k(\mathbf x),
\]
where a scalar polynomial of total degree $d\ge 2$, expressed as a sum of its homogeneous components (parts of the polynomial consisting of all monomials of the same total degree)
$p_k:\mathbb R^{n}\!\to\!\mathbb R$ of degree $k$ ($p_0$ is the constant term).

The continuous relaxation of the standard HR for this polynomial constraint introduces the following set in the lifted space:  
\[
    S_{d}^{(1)}
    \;=\;
    \Bigl\{
        (\mathbf v,y)\in\mathbb R^{n}\!\times\![0,1]
        \;\bigm|\;
        \operatorname{cl}\widetilde h(\mathbf v,y)\le 0,
        \;
        \mathbf x^{\ell}y\le\mathbf v\le \mathbf x^{u}y
    \Bigr\},
\]
where the closure of the perspective function in the domain of the reformulation can be expressed as:  
\begingroup
\begin{equation*}
\setlength{\jot}{0pt}
\begin{aligned}
    \bigl(\operatorname{cl}\widetilde h\bigr)(\mathbf v,y)
    \;=\;
    \begin{cases}
        y\,h\!\bigl(\mathbf v/y\bigr), & y>0,\\[2pt]
        0, & y=0.
    \end{cases}
\end{aligned}
\end{equation*}
\endgroup
which, in the case of polynomial constraints simplifies to:
\begingroup
\begin{equation*}
\setlength{\jot}{0pt}
\begin{aligned}
    \bigl(\operatorname{cl}\widetilde h\bigr)(\mathbf v,y)
    \;=\;
    \begin{cases}
    \sum_{k=0}^{d} p_k(\mathbf v)\,y^{\,1-k}, & y>0,\\[2pt]
        0, & y=0.
    \end{cases}
\end{aligned}
\end{equation*}
\endgroup

Multiplying  \( \sum_{k=0}^{d} p_k(\mathbf v)\,y^{\,1-k}\) by $y^{\,d-1}$  motivates the alternative set
\[
    S_{d}^{(2)}
    \;=\;
    \Bigl\{
        (\mathbf v,y)\in\mathbb R^{n}\!\times\![0,1]
        \;\bigm|\;
        \sum_{k=0}^{d} p_k(\mathbf v)\,y^{\,d-k}\le 0,
        \;
        \mathbf x^{\ell}y\le \mathbf v\le \mathbf x^{u}y
    \Bigr\}.
\]

We demonstrate the equivalence of these two sets in the following proposition:

\begin{proposition}\label{prop:poly-sets}

For any polynomial constraint \(h(\mathbf v)\;=\;\sum_{k=0}^{d} p_k(\mathbf v) \leq 0\) of degree~$d\!\ge\!2$, define the sets
\[
    S_{d}^{(1)}
    \;=\;
    \Bigl\{
        (\mathbf v,y)\in\mathbb R^{n}\!\times\![0,1]
        \;\bigm|\;
        \operatorname{cl}\widetilde h(\mathbf v,y)\le 0,
        \;
        \mathbf x^{\ell}y\le\mathbf v\le \mathbf x^{u}y
    \Bigr\},
\]
where
\begingroup
\begin{equation*}
\setlength{\jot}{0pt}
\begin{aligned}
    \bigl(\operatorname{cl}\widetilde h\bigr)(\mathbf v,y)
    \;=\;
    \begin{cases}
    \sum_{k=0}^{d} p_k(\mathbf v)\,y^{\,1-k}, & y>0,\\[2pt]
        0, & y=0.
    \end{cases}
\end{aligned}
\end{equation*}
\endgroup
and 
\[
    S_{d}^{(2)}
    \;=\;
    \Bigl\{
        (\mathbf v,y)\in\mathbb R^{n}\!\times\![0,1]
        \;\bigm|\;
        \sum_{k=0}^{d} p_k(\mathbf v)\,y^{\,d-k}\le 0,
        \;
        \mathbf x^{\ell}y\le \mathbf v\le \mathbf x^{u}y
    \Bigr\}.
\]

Then, 
the two sets coincide: \(\displaystyle S_{d}^{(1)} = S_{d}^{(2)}\).
\end{proposition}

\begin{proof}
\textbf{Case 1, \(y>0\):}  
Because \(y^{d-1}>0\), multiplying the defining inequality of
\(S_{d}^{(1)}\) by \(y^{d-1}\) preserves the feasible region and produces the
polynomial inequality that defines \(S_{d}^{(2)}\).
Hence \(S_{d}^{(1)}\cap\{y>0\}=S_{d}^{(2)}\cap\{y>0\}\).

\medskip\noindent
\textbf{Case 2: \(y=0\):}  
From the constraints \(\mathbf x^{\ell}y\le\mathbf v\le \mathbf x^{u}y\)
it follows that \(y=0\) implies \(\mathbf v=\mathbf 0\).
Substituting \((\mathbf v,y)=(\mathbf 0,0)\) into the polynomial expression
gives \(\sum_{k=0}^{d}p_k(\mathbf 0)\,0^{d-k}=0\), matching
\(\operatorname{cl}\widetilde h(\mathbf 0,0)=0\).
Therefore, the two sets also coincide on \(\{y=0\}\).

\medskip\noindent
Since both cases produce identical feasible regions, the sets are equal.
\end{proof}

Therefore, when applying HR, any polynomial disjunctive constraint
\(\sum_{k=0}^{d} p_k(\mathbf x) \leq 0\), can be reformulated in a way that forms exact hull and preserves polynomial structure as \( \sum_{k=0}^{d} p_k(\mathbf v)\,y^{\,d-k}\le 0\) (sum of monomials of degree $d$), instead of using $\varepsilon$‐approximation.

As with the quadratic case ($d=2$), addressed in the paper, the reformulation of polynomials for $d\ge 2$
\emph{(i)}~is exact, guaranteeing a strictly tighter continuous relaxation than $\varepsilon$‐approximation, when  $\varepsilon>0$ 
\emph{(ii)}~retains the original polynomial degree, preserving
algebraic structure of the constraints, and  
\emph{(iii)}~avoids potential numerical instabilities that arise from dividing by small values in $\varepsilon$‐approximations of the closure of the perspective function.

\section{Algebraic second-order cone representations for the Conic Exact Hull}
\label{app:conic-repr}

Section~\ref{sec:conic-gdp} shows that HR of convex quadratic disjunctive constraints admits an exact second-order cone representation (CEHR).
In \texttt{Pyomo.GDP}, constraints are passed to solvers in algebraic form, so conic-capable solvers can exploit the structure only if they recognize it from the provided algebraic representation. Since solver interfaces and presolve routines may prefer different but equivalent representations of the same cone constraint, we tested several equivalent forms of the Conic Exact Hull on the random convex benchmark set.

\paragraph{Derivation of the conic perspective constraint.}
Consider a convex quadratic inequality in a disjunct:
\begin{equation}
\label{eq:app-conic-base-q}
h(\mathbf{x}) \;=\; \mathbf{x}^{\mathsf T}Q\,\mathbf{x} + \mathbf{c}^{\mathsf T}\mathbf{x} + d \;\le\; 0,
\qquad Q \succeq 0.
\end{equation}
In HR, the closure of the perspective function constraint (for $y>0$) is
\begin{equation}
\label{eq:app-conic-persp}
\frac{\mathbf{v}^{\mathsf T}Q\,\mathbf{v}}{y} + \mathbf{c}^{\mathsf T}\mathbf{v} + d\,y \;\le\; 0,
\end{equation}
together with the standard disaggregation bounds $\mathbf{x}^{\ell}y \le \mathbf{v} \le \mathbf{x}^{u}y$
(and $y\in[0,1]$). Introduce an epigraph variable $t$ to isolate the quadratic-over-$y$ term:
\begin{equation}
\label{eq:app-conic-epi}
\mathbf{v}^{\mathsf T}Q\,\mathbf{v} \;\le\; t\,y,
\qquad
t + \mathbf{c}^{\mathsf T}\mathbf{v} + d\,y \;\le\; 0,
\qquad
t \ge 0,\; y \ge 0.
\end{equation}
For $y>0$, \eqref{eq:app-conic-epi} is equivalent to \eqref{eq:app-conic-persp} by dividing the first
inequality by $y$ and substituting into the second.
At $y=0$, the disaggregation bounds force $\mathbf{v}=\mathbf{0}$, so $\mathbf{v}^{\mathsf T}Q\mathbf{v}=0$ and \eqref{eq:app-conic-epi} implies $t=0$, matching the closure value $\bigl(\operatorname{cl}\widetilde h\bigr)(\mathbf{0},0)=0$.

\paragraph{Factorized form.}
Since $Q\succeq 0$, there exists a matrix $L$ such that
\begin{equation}
\label{eq:app-conic-factor}
Q \;=\; L\,L^{\mathsf T}.
\end{equation}
Therefore,
\begin{equation}
\label{eq:app-conic-norm}
\mathbf{v}^{\mathsf T}Q\,\mathbf{v} \;=\; \|L^{\mathsf T}\mathbf{v}\|_2^2,
\end{equation}
and the rotated-cone part of \eqref{eq:app-conic-epi} can be written as
$\|L^{\mathsf T}\mathbf{v}\|_2^2 \le t\,y$ with $t\ge 0$ and $y\ge 0$.

\paragraph{Tested reformulation representations (names match the plot legend).}
All representations below impose the same linear constraint
\(
t + \mathbf{c}^{\mathsf T}\mathbf{v} + d\,y \le 0
\),
the same nonnegativity conditions $t\ge 0$, $y\ge 0$, and the same disaggregation bounds
$\mathbf{x}^{\ell}y \le \mathbf{v} \le \mathbf{x}^{u}y$.
They differ only in the algebraic form used for the rotated second-order cone component.

\begin{itemize}
\item \textbf{Conic Exact Hull (No factorization).}
Direct quadratic-over-product inequality:
\begin{equation}
\label{eq:app-conic-nofact}
\mathbf{v}^{\mathsf T}Q\,\mathbf{v} \;\le\; t\,y.
\end{equation}

\item \textbf{Conic Exact Hull (Only factorized).}
Using \eqref{eq:app-conic-factor}--\eqref{eq:app-conic-norm}:
\begin{equation}
\label{eq:app-conic-onlyfact}
\|L^{\mathsf T}\mathbf{v}\|_2^2 \;\le\; t\,y.
\end{equation}

\item \textbf{Conic Exact Hull (No Sqrt, No Extra Var).}
Starting from \eqref{eq:app-conic-onlyfact}, use the identity
$(t+y)^2-(t-y)^2 = 4ty$ to obtain an equivalent quadratic inequality (for $t\ge 0$, $y\ge 0$):
\begin{equation}
\label{eq:app-conic-nosqrt-noextra}
4\|L^{\mathsf T}\mathbf{v}\|_2^2 + (t-y)^2 \;\le\; (t+y)^2,
\qquad t\ge 0,\; y\ge 0.
\end{equation}

\item \textbf{Conic Exact Hull (Sqrt, No Extra Var).}
Taking square roots of \eqref{eq:app-conic-nosqrt-noextra} yields the standard SOC form:
\begin{equation}
\label{eq:app-conic-sqrt-noextra}
\sqrt{\,4\|L^{\mathsf T}\mathbf{v}\|_2^2 + (t-y)^2\,} \;\le\; t+y,
\qquad t\ge 0,\; y\ge 0.
\end{equation}

\item \textbf{Conic Exact Hull (Sqrt, Extra Var).}
Introduce auxiliary variables
\begin{equation}
\label{eq:app-conic-extra-defs}
\mathbf{z} = 2L^{\mathsf T}\mathbf{v},
\qquad
w = t-y,
\qquad
s = t+y,
\qquad
s \ge 0,
\end{equation}
and impose the SOC
\begin{equation}
\label{eq:app-conic-sqrt-extra}
\sqrt{\|\mathbf{z}\|_2^2 + w^2} \;\le\; s.
\end{equation}

\item \textbf{Conic Exact Hull (No Sqrt, Extra Var).}
Using the same auxiliaries \eqref{eq:app-conic-extra-defs}, impose the squared SOC form
\begin{equation}
\label{eq:app-conic-nosqrt-extra}
\|\mathbf{z}\|_2^2 + w^2 \;\le\; s^2,
\qquad s\ge 0.
\end{equation}
\end{itemize}

\paragraph{Explicit rotated-cone membership (optional).}
For reference, \eqref{eq:app-conic-onlyfact} is equivalent to rotated second-order cone membership:
\[
\left(\frac{t}{2},\, y,\, L^{\mathsf T}\mathbf{v}\right)\in\mathcal{Q}_r
\;:=\;
\left\{(\alpha,\beta,\mathbf{w}) \ \middle|\ 2\alpha\beta \ge \|\mathbf{w}\|_2^2,\ \alpha\ge 0,\ \beta\ge 0\right\}.
\]

\paragraph{Computational comparison and choice of default representation.}
Figure~\ref{fig:app-conic-profiles} compares these Conic Exact Hull representations on the random convex
benchmark set using \texttt{Gurobi}, \texttt{BARON}, and \texttt{SCIP}. The figure also includes \textbf{BigM} as a baseline reference.

For \texttt{Gurobi} and \texttt{SCIP}, which have explicit conic capabilities, \textbf{Conic Exact Hull (No factorization)}, denoted as just Conic Exact Hull in Figure~\ref{fig:app-conic-profiles}, performs among the best representations, suggesting that both solvers recognize and exploit the rotated-cone structure directly from this algebraic form.

For \texttt{BARON}, \textbf{Conic Exact Hull (No factorization)} performs comparably to most other Conic Exact Hull representations on this benchmark. One representation performs noticeably better here, namely \textbf{Conic Exact Hull (Sqrt, No Extra Var)}, but the same representation performs poorly for \texttt{Gurobi} and \texttt{SCIP}.

Taken together, these results support using \textbf{Conic Exact Hull (No factorization)} in the direct form \eqref{eq:app-conic-nofact} as the overall default representation: it is consistently competitive and is the strongest option for solvers that actively exploit conic structure (\texttt{Gurobi} and \texttt{SCIP}), while remaining comparable to alternatives for \texttt{BARON}.

\begin{figure}[htbp]
  \centering
  \begin{subfigure}[t]{0.32\textwidth}
    \centering
    \includegraphics[width=\linewidth]{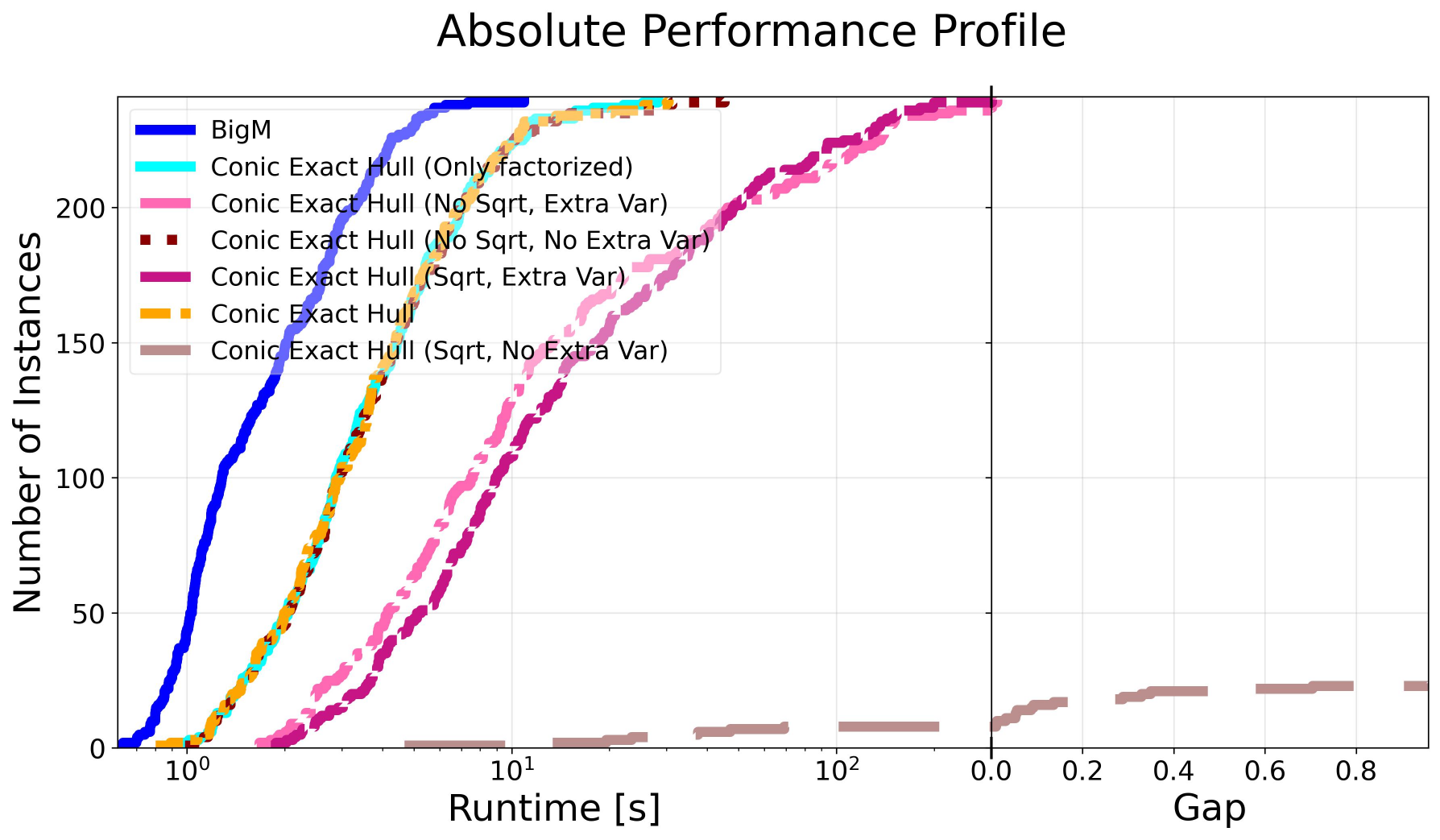}
    \caption{\texttt{Gurobi}}
    \label{fig:app-conic-gurobi}
  \end{subfigure}\hfill
  \begin{subfigure}[t]{0.32\textwidth}
    \centering
    \includegraphics[width=\linewidth]{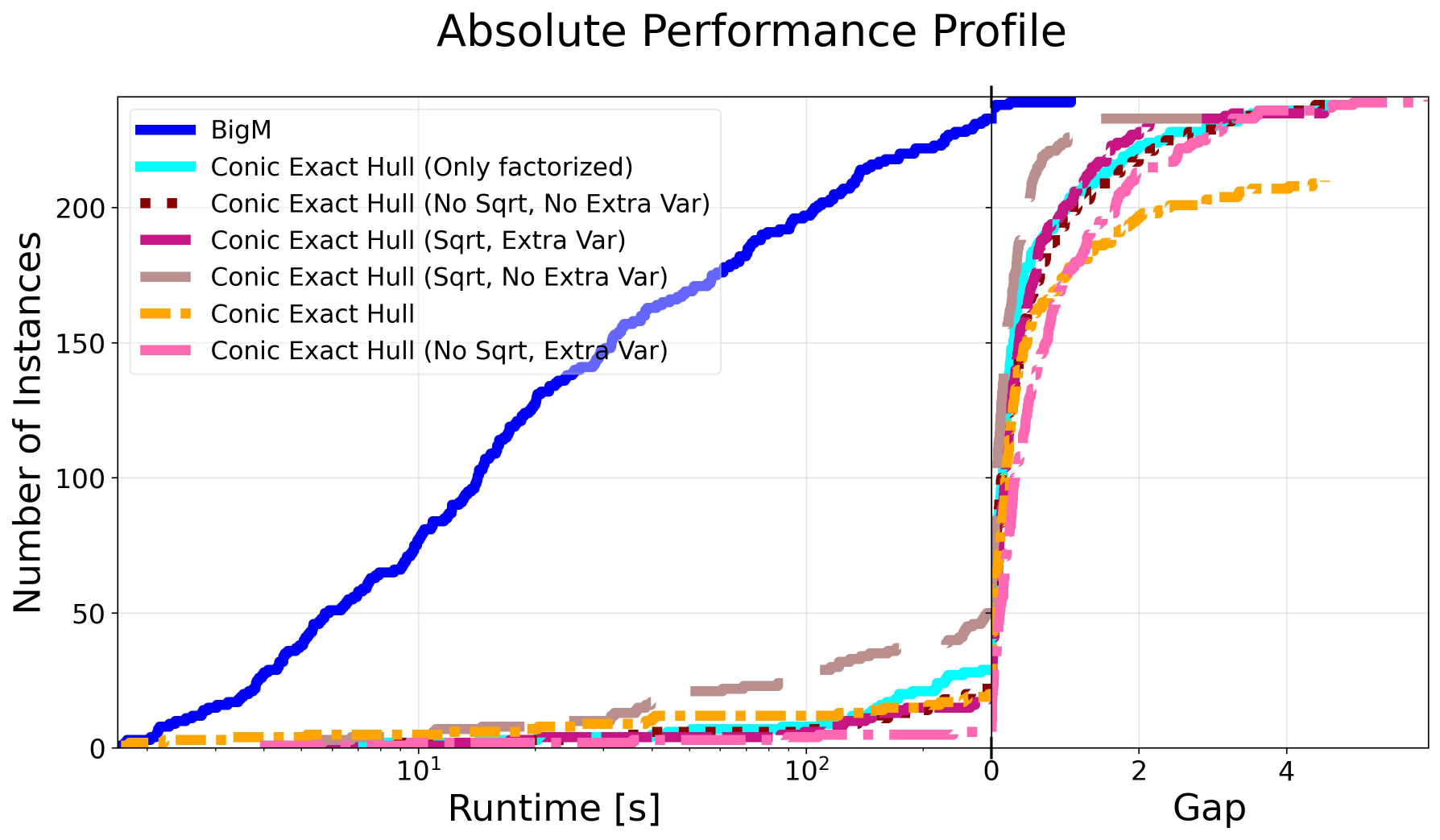}
    \caption{\texttt{BARON}}
    \label{fig:app-conic-baron}
  \end{subfigure}\hfill
  \begin{subfigure}[t]{0.32\textwidth}
    \centering
    \includegraphics[width=\linewidth]{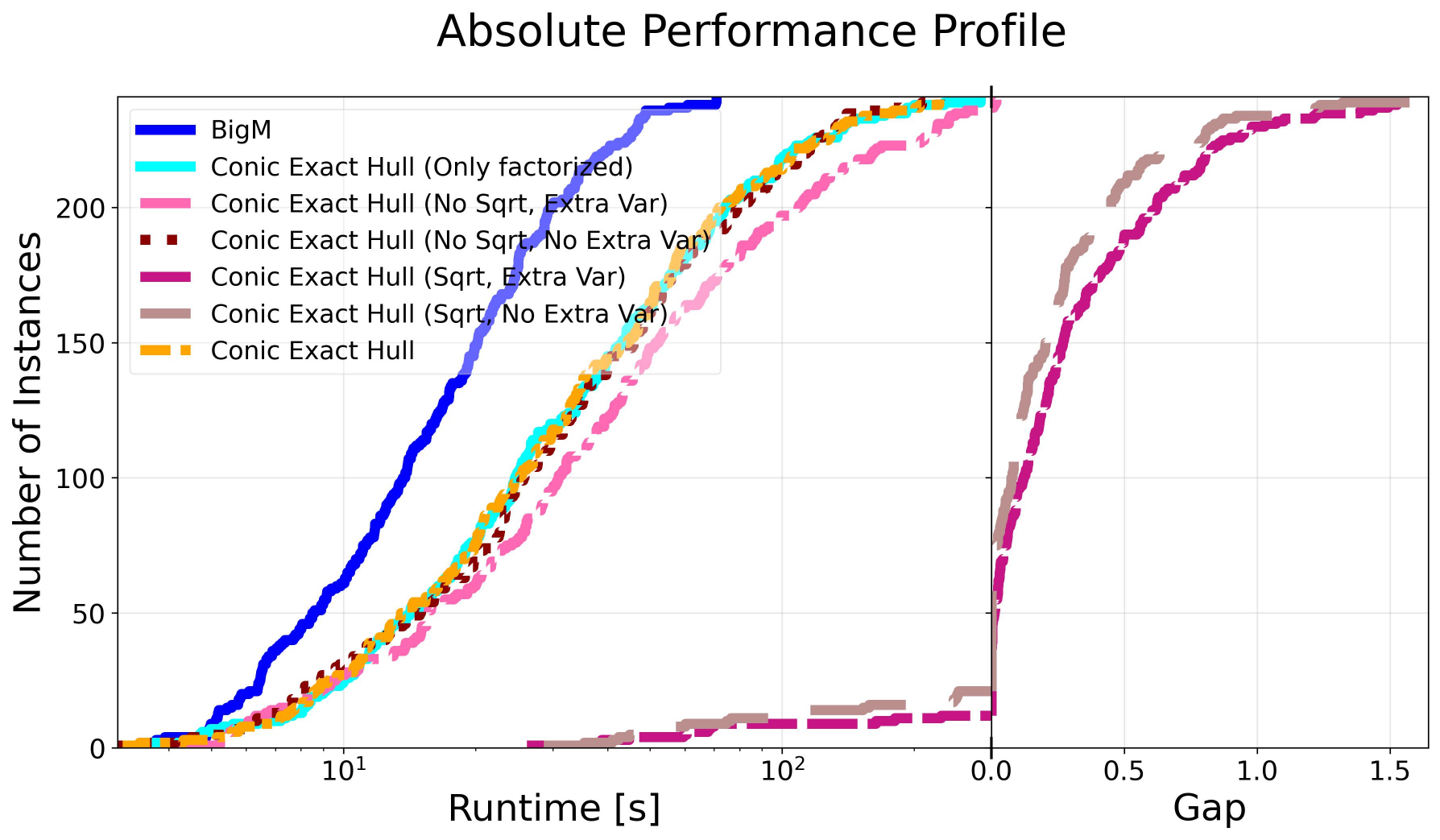}
    \caption{\texttt{SCIP}}
    \label{fig:app-conic-scip}
  \end{subfigure}
  \caption{Absolute performance profiles on the random convex benchmark set comparing different Conic Exact Hull
  reformulation representations. Legend names match the representations listed in Appendix~\ref{app:conic-repr}.}
  \label{fig:app-conic-profiles}
\end{figure}

\section{Effect of the \texorpdfstring{$\varepsilon$}{epsilon}-approximation parameter}
\label{app:epsilon}

The standard HR for nonlinear constraints uses the closure of the perspective function,
which is not defined analytically at $y=0$.
In practice, this is commonly handled with the
$\varepsilon$-approximation (Eq.~\eqref{eq:approx}), which replaces the division by $y$ with a strictly
positive denominator $(1-\varepsilon)y+\varepsilon$.
The parameter $\varepsilon$ controls a trade-off:
larger values typically improve numerical conditioning but enlarge the continuous relaxation, while smaller values make the approximation closer to the exact closure, but can worsen scaling when $y$ is small.

In the main experiments, we used the default \texttt{Pyomo} value $\varepsilon=10^{-4}$.
This appendix summarizes a sensitivity study on $\varepsilon$ using the values $\varepsilon\in\{10^{-2},10^{-3},10^{-4}\}$ for both convex (Figure~\ref{fig:eps-psd} ) and non-convex random  (Figure~\ref{fig:eps-nonconv}) benchmark sets, keeping all other modeling and solver settings identical (including the same time limits and tolerances).

Overall, changing $\varepsilon$ affected the solution time and numeric stability.
However, in most cases, the performance is qualitatively similar across the tested values, and differences are not expected materially change the main conclusions of the paper.

Additionally, the best choice of $\varepsilon$ may depend on the problem and the solver.
Identifying an ``optimal'' $\varepsilon$ for a given problem is itself a challenging tuning task with no broadly applicable, principled selection rule.
A brute-force search over $\varepsilon$ values significantly increases computational burden and makes this parameter tuning impractical.
This drawback of the $\varepsilon$-approximation approach further motivates exact reformulations that avoid introducing $\varepsilon$ altogether.


\begin{figure}[htbp]
  \centering
    \begin{subfigure}[t]{0.32\textwidth}
    \centering
    \includegraphics[width=\linewidth]{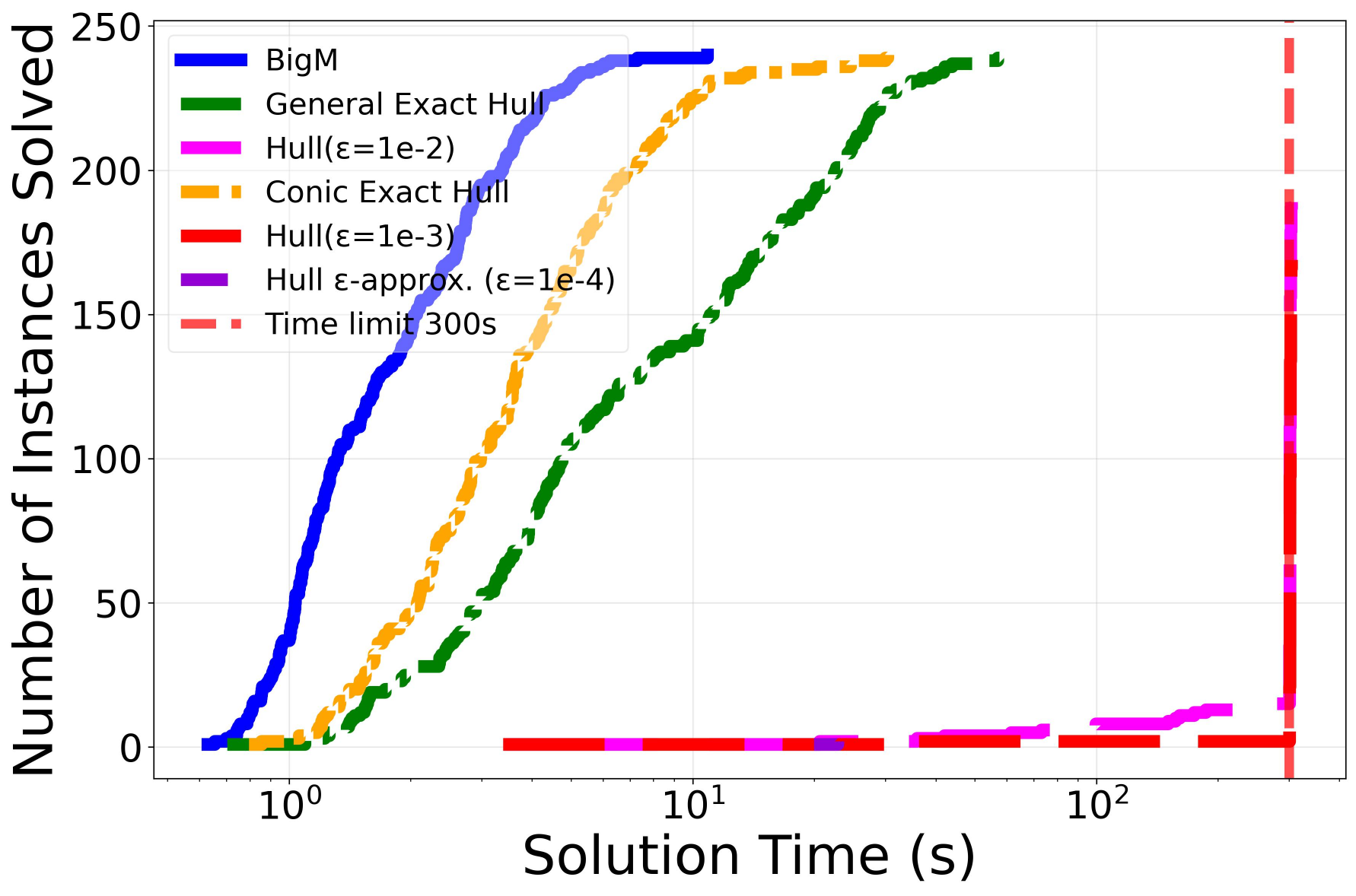}
    \caption{\texttt{Gurobi}}
    \label{fig:eps-psd-gurobi}
  \end{subfigure}\hfill
  \begin{subfigure}[t]{0.32\textwidth}
    \centering
    \includegraphics[width=\linewidth]{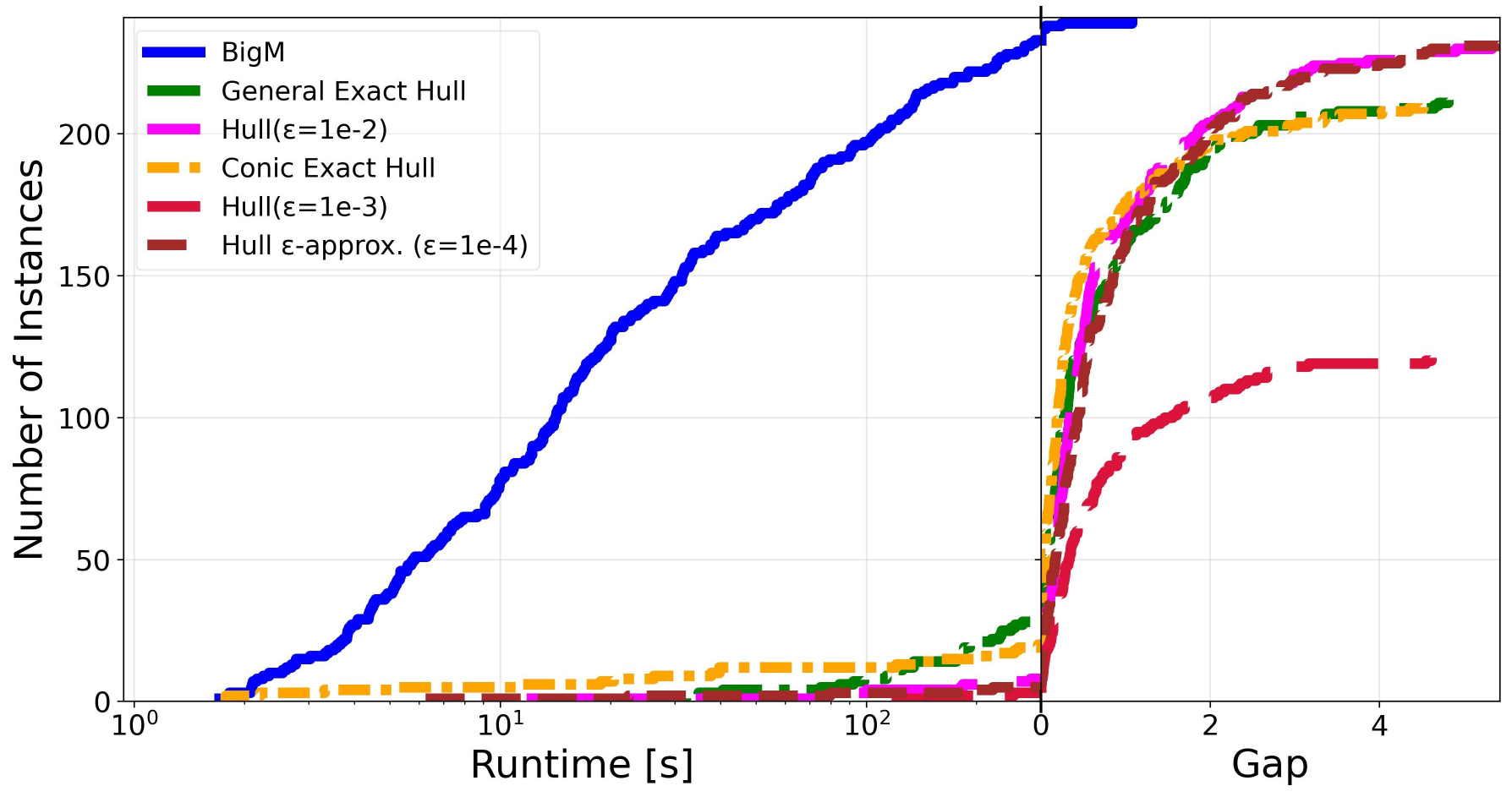}
    \caption{\texttt{BARON}}
    \label{fig:eps-psd-baron}
  \end{subfigure}\hfill
  \begin{subfigure}[t]{0.32\textwidth}
    \centering
    \includegraphics[width=\linewidth]{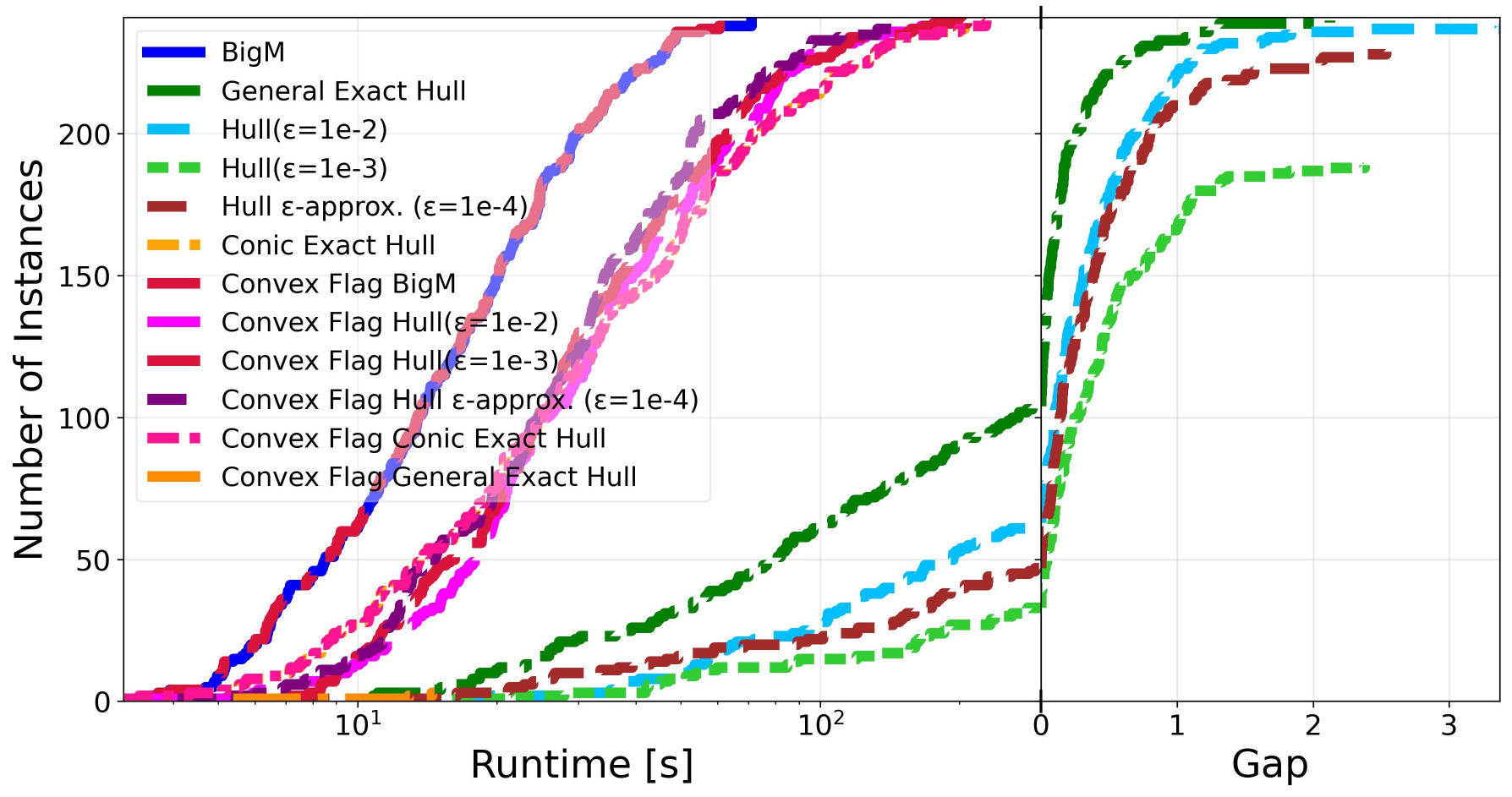}
    \caption{\texttt{SCIP}}
    \label{fig:eps-psd-scip}
  \end{subfigure}
  \caption{Sensitivity to the values of $\varepsilon$  in $\varepsilon$-approximation on the convex random benchmark set.}
  \label{fig:eps-psd}
\end{figure}

\begin{figure}[htbp]
  \centering
    \begin{subfigure}[t]{0.32\textwidth}
    \centering
    \includegraphics[width=\linewidth]{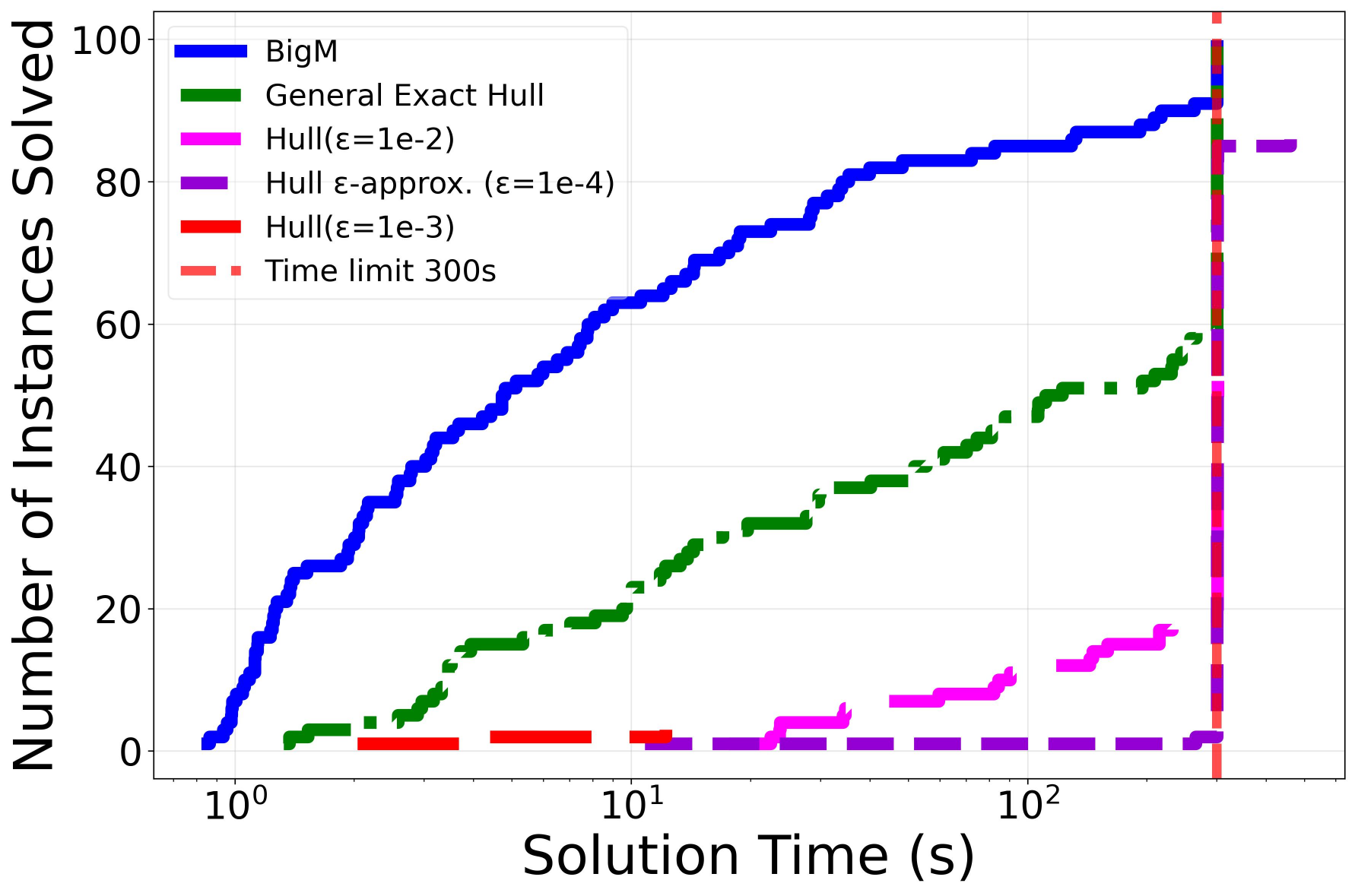}
    \caption{\texttt{Gurobi}}
    \label{fig:eps-nonconv-gurobi}
  \end{subfigure}\hfill
  \begin{subfigure}[t]{0.32\textwidth}
    \centering
    \includegraphics[width=\linewidth]{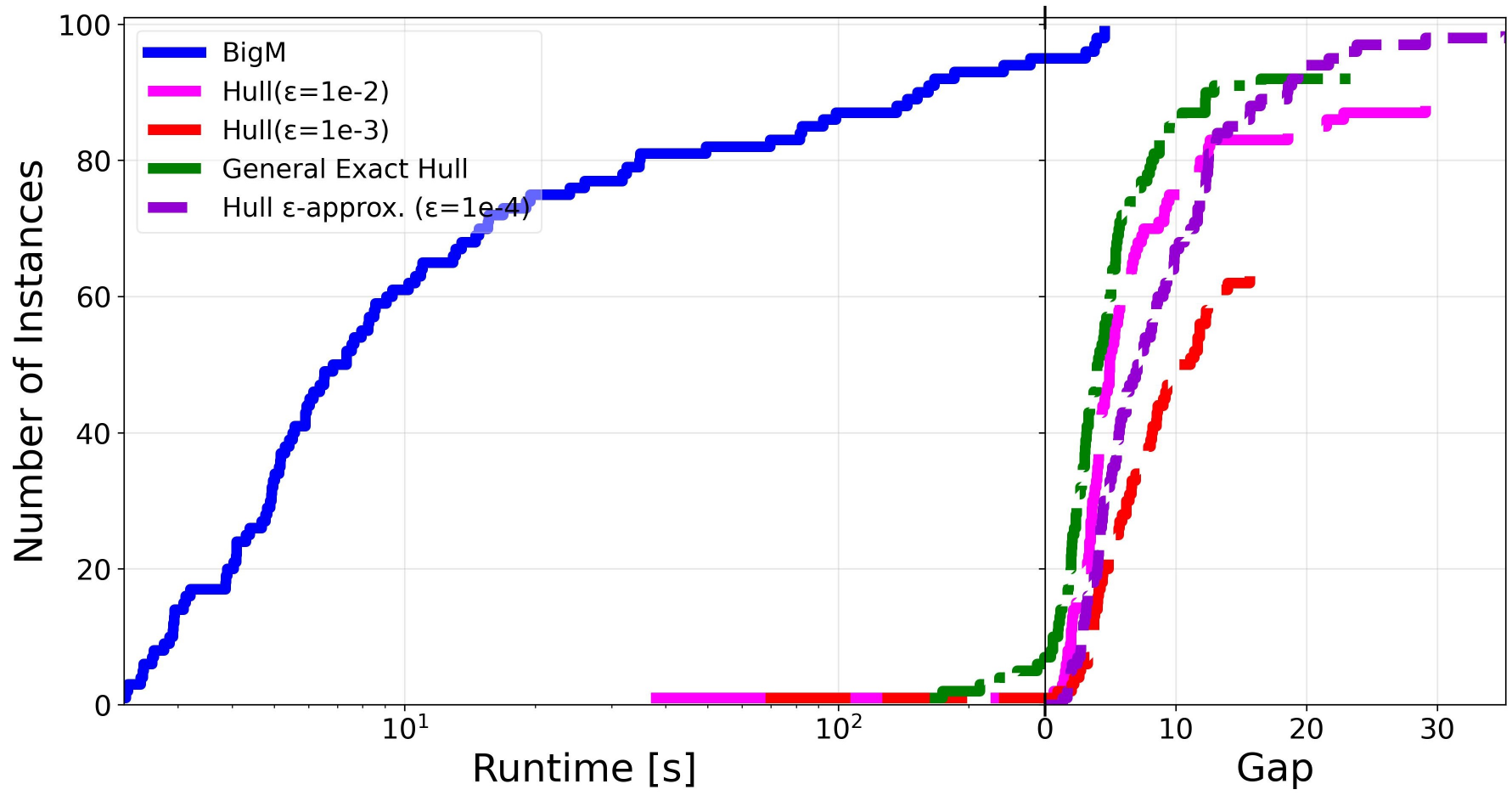}
    \caption{\texttt{BARON}}
    \label{fig:eps-nonconv-baron}
  \end{subfigure}\hfill
  \begin{subfigure}[t]{0.32\textwidth}
    \centering
    \includegraphics[width=\linewidth]{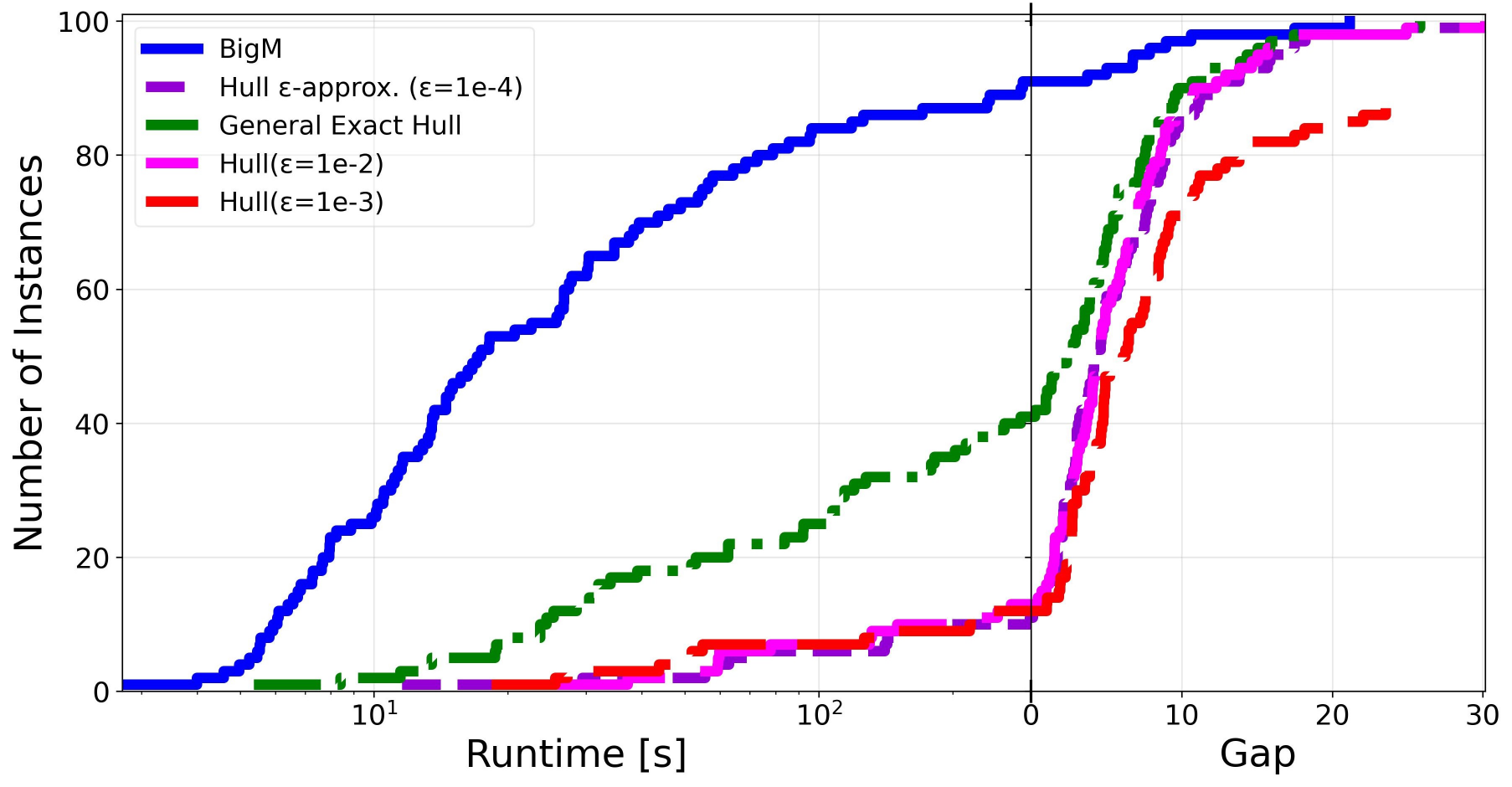}
    \caption{\texttt{SCIP}}
    \label{fig:eps-nonconv-scip}
  \end{subfigure}
  \caption{Sensitivity to the values of $\varepsilon$  in $\varepsilon$-approximation on the non-convex random benchmark set.}
  \label{fig:eps-nonconv}
\end{figure}

\section{Continuously Stirred Tank Reactor (CSTR) Network}
\label{app:cstr}

This appendix provides a description and the set of equations used to formulate the Continuously Stirred Tank Reactor (CSTR) network optimization problem considered in Section~\ref{sec:experiments}, adapted from prior formulations in the literature \cite{ovalleLogicBasedDiscreteSteepestDescent2025}. 

The system consists of $N_T$ units arranged in series, each of which can be either a reactor or a bypass.
The autocatalytic reaction $A + B \rightarrow 2B$ occurs in active reactors. 

Each stage \(n\in N=\{1,\dots,N_T\}\) may operate as a CSTR or be bypassed; exactly one stage receives the fresh feed, and exactly one stage receives the recycle stream returned from the outlet of the last stage.  
Component indices are collected in \(I=\{A,B\}\).  
Variables include molar flow rates \(F_{i,n}\) and \(FR_{i,n}\), total flow rates \(Q_n\) and \(Q_{FR,n}\), reactor volumes \(V_n\) (with volume cost \(c_n\)), reaction rates \(r_{i,n}\), the recycle flow \((R_i,Q_R)\), the product flow \((P_i,Q_P)\), and an auxiliary variables \(t_n\) are used to reduce cubic terms to quadratic form.  
Boolean variables \(YF_n\), \(YR_n\), and \(YP_n\) denote, respectively, that fresh feed enters stage \(n\), recycle enters stage \(n\), and stage \(n\) operates as a CSTR.  
All continuous variables are non-negative.

The complete GDP formulation can be presented as follows:

\begingroup
\begin{equation}
\begin{aligned}
\min \quad
& \sum_{n\in N} c_n
\\
\text{w.r.t.} \quad
& F_{i,n},\, FR_{i,n},\, r_{i,n},\, t_n,\, Q_n,\, V_n,\, c_n,
\\
& YF_n,\, YR_n,\, YP_n
\\[0.4em]
\text{s.t.} \quad
& F_{i,NT}-F0_i-FR_{i,NT}-r_{i,NT}V_{NT}=0,
\qquad \forall\, i\in I,
\\
& Q_{NT}-Q_{F0}-Q_{FR,NT}=0,
\\
& F_{i,n}-F_{i,n+1}-FR_{i,n}-r_{i,n}V_n=0,
\qquad \forall\, n\in N\setminus\{N_T\},\ i\in I,
\\
& Q_{n}-Q_{n+1}-Q_{FR,n}=0,
\qquad \forall\, n\in N\setminus\{N_T\},
\\
& F_{i,1}-P_i-R_i=0,
\qquad \forall\, i\in I,
\\
& Q_{1}-Q_P-Q_R=0,
\\
& P_iQ_{1}-F_{i,1}Q_P=0,
\qquad \forall\, i\in I,
\\
& 0.95\,Q_P-P_B=0,
\\
& V_n-V_{n-1}=0,
\qquad \forall\, n\in N\setminus\{1\},
\\
& Q_n^{2}-t_n=0,
\qquad \forall\, n\in N,
\\[0.6em]
&
\left[
\begin{aligned}
& YP_n\\
& r_{A,n}t_n+kF_{A,n}F_{B,n}=0\\
& r_{B,n}+r_{A,n}=0\\
& c_n-V_n=0
\end{aligned}
\right]
\bigveebar
\left[
\begin{aligned}
& \neg YP_n\\
& FR_{i,n}=0 \quad (i\in I)\\
& r_{i,n}=0 \quad (i\in I)\\
& Q_{FR,n}=0\\
& c_n=0
\end{aligned}
\right],
\\
& \hspace{2em} \forall\, n\in N,
\\[0.6em]
&
\left[
\begin{aligned}
& YR_n\\
& FR_{i,n}-R_i=0 \quad (i\in I)\\
& Q_{FR,n}-Q_R=0
\end{aligned}
\right]
\bigveebar
\left[
\begin{aligned}
& \neg YR_n\\
& FR_{i,n}=0 \quad (i\in I)\\
& Q_{FR,n}=0
\end{aligned}
\right],
\\
& \hspace{2em} \forall\, n\in N,
\\[0.6em]
& \underline{\bigvee}_{n\in N} YF_n,
\\
& \underline{\bigvee}_{n\in N} YR_n,
\\[0.4em]
& YP_n \iff
\left(
\bigwedge_{j\in\{1,2,\ldots,n\}} \neg YF_j
\right)\vee YF_n,
\\
& \hspace{2em} \forall\, n\in N,
\\
& YR_n \Longrightarrow YP_n,
\qquad \forall\, n\in N,
\\[0.4em]
& F_{i,n},FR_{i,n},P_i,R_i,Q_n,Q_{FR,n},Q_{F0}\ge 0,
\\
& \hspace{2em} \forall\, n\in N,\ i\in I,
\\
& Q_P,Q_R,V_n,c_n,t_n\ge 0,
\qquad \forall\, n\in N,
\\
& YF_n,YR_n,YP_n\in\{\mathrm{False},\mathrm{True}\},
\\
& \hspace{2em} \forall\, n\in N.
\end{aligned}
\end{equation}
\endgroup

The first group of equalities enforces component and mass balances (assuming constant density) in the fresh feed stage, interior stages, and the upstream splitter.  
The purity specification \(0.95\,Q_P=P_B\) guarantees required \(0.95 mol/l\) concentration of the component \(B\) in the outlet.  
Reactor volumes are constrained to be equal. 
The auxiliary relation \(Q_n^{2}=t_n\) converts cubic reaction terms to quadratic form.  
The first disjunction activates the autocatalytic rate law, volume cost, and recycle‐related flows when a stage operates as a reactor. 
The second disjunction specifies whether or not the recycle stream enters a given stage.  
Logic conditions enforce that there can be only one unreacted feed and one recycle stream.
Also, logic conditions state that a stage and all subsequent stages become a reactor once the fresh feed appears, and that recycle can only return to a reactor stage.  

\end{appendices}

%% file: references.bib
@incollection{grossmannGeneralizedDisjunctiveProgramming2012,
	address = {New York, NY},
	title = {Generalized {Disjunctive} {Programming}: {A} {Framework} for {Formulation} and {Alternative} {Algorithms} for {MINLP} {Optimization}},
	volume = {154},
	isbn = {978-1-4614-1926-6 978-1-4614-1927-3},
	shorttitle = {Generalized {Disjunctive} {Programming}},
	url = {http://link.springer.com/10.1007/978-1-4614-1927-3_4},
	language = {en},
	urldate = {2025-03-19},
	booktitle = {Mixed {Integer} {Nonlinear} {Programming}},
	publisher = {Springer New York},
	author = {Grossmann, Ignacio E. and Ruiz, Juan P.},
	editor = {Lee, Jon and Leyffer, Sven},
	year = {2012},
	doi = {10.1007/978-1-4614-1927-3_4},
	note = {Series Title: The IMA Volumes in Mathematics and its Applications},
	pages = {93--115},
}

@article{mencarelliReviewSuperstructureOptimization2020,
	title = {A review on superstructure optimization approaches in process system engineering},
	volume = {136},
	issn = {0098-1354},
	url = {https://www.sciencedirect.com/science/article/pii/S0098135419313924},
	doi = {10.1016/j.compchemeng.2020.106808},
	urldate = {2025-03-19},
	journal = {Computers \& Chemical Engineering},
	author = {Mencarelli, Luca and Chen, Qi and Pagot, Alexandre and Grossmann, Ignacio E.},
	month = may,
	year = {2020},
	pages = {106808},
}

@article{grossmannSystematicModelingDiscretecontinuous2013,
	title = {Systematic modeling of discrete-continuous optimization models through generalized disjunctive programming},
	volume = {59},
	copyright = {Copyright © 2013 American Institute of Chemical Engineers},
	issn = {1547-5905},
	url = {https://onlinelibrary.wiley.com/doi/abs/10.1002/aic.14088},
	doi = {10.1002/aic.14088},
	language = {en},
	number = {9},
	urldate = {2025-03-19},
	journal = {AIChE Journal},
	author = {Grossmann, Ignacio E. and Trespalacios, Francisco},
	year = {2013},
	note = {\_eprint: https://onlinelibrary.wiley.com/doi/pdf/10.1002/aic.14088},
	pages = {3276--3295},
}

@article{trespalaciosReviewMixedIntegerNonlinear2014,
	title = {Review of {Mixed}-{Integer} {Nonlinear} and {Generalized} {Disjunctive} {Programming} {Methods}},
	volume = {86},
	copyright = {Copyright © 2014 WILEY-VCH Verlag GmbH \& Co. KGaA, Weinheim},
	issn = {1522-2640},
	url = {https://onlinelibrary.wiley.com/doi/abs/10.1002/cite.201400037},
	doi = {10.1002/cite.201400037},
	language = {en},
	number = {7},
	urldate = {2025-03-19},
	journal = {Chemie Ingenieur Technik},
	author = {Trespalacios, Francisco and Grossmann, Ignacio E.},
	year = {2014},
	note = {\_eprint: https://onlinelibrary.wiley.com/doi/pdf/10.1002/cite.201400037},
	pages = {991--1012},
}

@article{grossmannGeneralizedConvexDisjunctive2003,
	title = {Generalized {Convex} {Disjunctive} {Programming}: {Nonlinear} {Convex} {Hull} {Relaxation}},
	volume = {26},
	issn = {1573-2894},
	shorttitle = {Generalized {Convex} {Disjunctive} {Programming}},
	url = {https://doi.org/10.1023/A:1025154322278},
	doi = {10.1023/A:1025154322278},
	language = {en},
	number = {1},
	urldate = {2025-03-19},
	journal = {Computational Optimization and Applications},
	author = {Grossmann, Ignacio E. and Lee, Sangbum},
	month = oct,
	year = {2003},
	pages = {83--100},
}

@article{chenModernModelingParadigms2019,
	title = {Modern {Modeling} {Paradigms} {Using} {Generalized} {Disjunctive} {Programming}},
	volume = {7},
	copyright = {https://creativecommons.org/licenses/by/4.0/},
	issn = {2227-9717},
	url = {https://www.mdpi.com/2227-9717/7/11/839},
	doi = {10.3390/pr7110839},
	language = {en},
	number = {11},
	urldate = {2025-03-19},
	journal = {Processes},
	author = {Chen, Qi and Grossmann, Ignacio},
	month = nov,
	year = {2019},
	pages = {839},
}

@article{bernalneiraConvexMixedintegerNonlinear2024,
	title = {Convex mixed-integer nonlinear programs derived from generalized disjunctive programming using cones},
	volume = {88},
	issn = {1573-2894},
	url = {https://doi.org/10.1007/s10589-024-00557-9},
	doi = {10.1007/s10589-024-00557-9},
	language = {en},
	number = {1},
	urldate = {2025-10-01},
	journal = {Computational Optimization and Applications},
	author = {Bernal Neira, David E. and Grossmann, Ignacio E.},
	month = may,
	year = {2024},
	pages = {251--312},
}

@article{ramanModellingComputationalTechniques1994,
	series = {An {International} {Journal} of {Computer} {Applications} in {Chemical} {Engineering}},
	title = {Modelling and computational techniques for logic based integer programming},
	volume = {18},
	issn = {0098-1354},
	url = {https://www.sciencedirect.com/science/article/pii/0098135493E00107},
	doi = {10.1016/0098-1354(93)E0010-7},
	number = {7},
	urldate = {2025-03-19},
	journal = {Computers \& Chemical Engineering},
	author = {Raman, R. and Grossmann, I. E.},
	month = jul,
	year = {1994},
	pages = {563--578},
}

@article{furmanComputationallyUsefulAlgebraic2020,
	title = {A computationally useful algebraic representation of nonlinear disjunctive convex sets using the perspective function},
	volume = {76},
	issn = {1573-2894},
	url = {https://doi.org/10.1007/s10589-020-00176-0},
	doi = {10.1007/s10589-020-00176-0},
	language = {en},
	number = {2},
	urldate = {2025-03-22},
	journal = {Computational Optimization and Applications},
	author = {Furman, Kevin C. and Sawaya, Nicolas W. and Grossmann, Ignacio E.},
	month = jun,
	year = {2020},
	pages = {589--614},
}

@article{grossmannReviewNonlinearMixedInteger2002,
	title = {Review of {Nonlinear} {Mixed}-{Integer} and {Disjunctive} {Programming} {Techniques}},
	volume = {3},
	issn = {1573-2924},
	url = {https://doi.org/10.1023/A:1021039126272},
	doi = {10.1023/A:1021039126272},
	language = {en},
	number = {3},
	urldate = {2025-03-22},
	journal = {Optimization and Engineering},
	author = {Grossmann, Ignacio E.},
	month = sep,
	year = {2002},
	pages = {227--252},
}

@inproceedings{kronqvistStepsIntermediateRelaxations2021,
	address = {Cham},
	title = {Between {Steps}: {Intermediate} {Relaxations} {Between} {Big}-{M} and {Convex} {Hull} {Formulations}},
	isbn = {978-3-030-78230-6},
	shorttitle = {Between {Steps}},
	doi = {10.1007/978-3-030-78230-6_19},
	language = {en},
	booktitle = {Integration of {Constraint} {Programming}, {Artificial} {Intelligence}, and {Operations} {Research}},
	publisher = {Springer International Publishing},
	author = {Kronqvist, Jan and Misener, Ruth and Tsay, Calvin},
	editor = {Stuckey, Peter J.},
	year = {2021},
	pages = {299--314},
}

@article{ruizGlobalOptimizationNonconvex2017,
	title = {Global optimization of non-convex generalized disjunctive programs: a review on reformulations and relaxation techniques},
	volume = {67},
	issn = {1573-2916},
	shorttitle = {Global optimization of non-convex generalized disjunctive programs},
	url = {https://doi.org/10.1007/s10898-016-0401-0},
	doi = {10.1007/s10898-016-0401-0},
	language = {en},
	number = {1},
	urldate = {2025-03-25},
	journal = {Journal of Global Optimization},
	author = {Ruiz, Juan P. and Grossmann, Ignacio E.},
	month = jan,
	year = {2017},
	pages = {43--58},
}

@incollection{gunlukPerspectiveReformulationApplications2012,
	address = {New York, NY},
	title = {Perspective {Reformulation} and {Applications}},
	volume = {154},
	isbn = {978-1-4614-1926-6 978-1-4614-1927-3},
	url = {https://link.springer.com/10.1007/978-1-4614-1927-3_3},
	language = {en},
	urldate = {2025-03-26},
	booktitle = {Mixed {Integer} {Nonlinear} {Programming}},
	publisher = {Springer New York},
	author = {Günlük, Oktay and Linderoth, Jeff},
	editor = {Lee, Jon and Leyffer, Sven},
	year = {2012},
	doi = {10.1007/978-1-4614-1927-3_3},
	note = {Series Title: The IMA Volumes in Mathematics and its Applications},
	pages = {61--89},
}

@article{kariaAssessmentTwostepApproach2022,
	title = {Assessment of a two-step approach for global optimization of mixed-integer polynomial programs using quadratic reformulation},
	volume = {165},
	issn = {00981354},
	url = {https://linkinghub.elsevier.com/retrieve/pii/S0098135422002472},
	doi = {10.1016/j.compchemeng.2022.107909},
	language = {en},
	urldate = {2025-03-26},
	journal = {Computers \& Chemical Engineering},
	author = {Karia, Tanuj and Adjiman, Claire S. and Chachuat, Benoît},
	month = sep,
	year = {2022},
	pages = {107909},
}

@book{bynumPyomoOptimizationModeling2021,
	address = {Cham},
	series = {Springer {Optimization} and {Its} {Applications}},
	title = {Pyomo — {Optimization} {Modeling} in {Python}},
	volume = {67},
	copyright = {http://www.springer.com/tdm},
	isbn = {978-3-030-68927-8 978-3-030-68928-5},
	url = {http://link.springer.com/10.1007/978-3-030-68928-5},
	language = {en},
	urldate = {2025-03-27},
	publisher = {Springer International Publishing},
	author = {Bynum, Michael L. and Hackebeil, Gabriel A. and Hart, William E. and Laird, Carl D. and Nicholson, Bethany L. and Siirola, John D. and Watson, Jean-Paul and Woodruff, David L.},
	year = {2021},
	doi = {10.1007/978-3-030-68928-5},
}

@misc{GurobiOptimizationLLC,
	title = {Gurobi {Optimization}, {LLC}.},
	url = {https://www.gurobi.com/},
	urldate = {2025-03-27},
}

@incollection{ruizGeneralizedDisjunctiveProgramming2012,
	address = {Berlin, Heidelberg},
	title = {Generalized {Disjunctive} {Programming}: {Solution} {Strategies}},
	isbn = {978-3-642-23592-4},
	shorttitle = {Generalized {Disjunctive} {Programming}},
	url = {https://doi.org/10.1007/978-3-642-23592-4_4},
	language = {en},
	urldate = {2025-04-20},
	booktitle = {Algebraic {Modeling} {Systems}: {Modeling} and {Solving} {Real} {World} {Optimization} {Problems}},
	publisher = {Springer},
	author = {Ruiz, Juan P. and Jagla, Jan-H. and Grossmann, Ignacio E. and Meeraus, Alex and Vecchietti, Aldo},
	editor = {Kallrath, Josef},
	year = {2012},
	doi = {10.1007/978-3-642-23592-4_4},
	pages = {57--75},
}

@article{ovalleLogicBasedDiscreteSteepestDescent2025,
	title = {Logic-{Based} {Discrete}-{Steepest} {Descent}: {A} solution method for process synthesis {Generalized} {Disjunctive} {Programs}},
	volume = {195},
	issn = {0098-1354},
	shorttitle = {Logic-{Based} {Discrete}-{Steepest} {Descent}},
	url = {https://www.sciencedirect.com/science/article/pii/S0098135424004113},
	doi = {10.1016/j.compchemeng.2024.108993},
	urldate = {2025-04-25},
	journal = {Computers \& Chemical Engineering},
	author = {Ovalle, Daniel and Liñán, David A. and Lee, Albert and Gómez, Jorge M. and Ricardez-Sandoval, Luis and Grossmann, Ignacio E. and Bernal Neira, David E.},
	month = apr,
	year = {2025},
	pages = {108993},
}

@misc{OptimizationFirmBARON,
	title = {The {Optimization} {Firm}. {BARON}},
	url = {https://minlp.com/baron-solver},
	urldate = {2025-05-16},
}

@misc{GAMS,
	title = {{GAMS}},
	url = {https://www.gams.com/download/},
	urldate = {2025-05-16},
}

@article{ceriaConvexProgrammingDisjunctive1999,
	title = {Convex programming for disjunctive convex optimization},
	volume = {86},
	issn = {1436-4646},
	url = {https://doi.org/10.1007/s101070050106},
	doi = {10.1007/s101070050106},
	language = {en},
	number = {3},
	urldate = {2025-05-22},
	journal = {Mathematical Programming},
	author = {Ceria, Sebastián and Soares, João},
	month = dec,
	year = {1999},
	pages = {595--614},
}

@book{balasDisjunctiveProgramming2018,
	address = {Cham},
	title = {Disjunctive {Programming}},
	copyright = {http://www.springer.com/tdm},
	isbn = {978-3-030-00147-6 978-3-030-00148-3},
	url = {http://link.springer.com/10.1007/978-3-030-00148-3},
	language = {en},
	urldate = {2025-06-27},
	publisher = {Springer International Publishing},
	author = {Balas, Egon},
	year = {2018},
	doi = {10.1007/978-3-030-00148-3},
}

@article{papageorgiouPseudoBasicSteps2018,
	title = {Pseudo basic steps: bound improvement guarantees from {Lagrangian} decomposition in convex disjunctive programming},
	volume = {6},
	issn = {2192-4406},
	shorttitle = {Pseudo basic steps},
	url = {https://www.sciencedirect.com/science/article/pii/S2192440621000952},
	doi = {10.1007/s13675-017-0088-0},
	number = {1},
	urldate = {2025-07-12},
	journal = {EURO Journal on Computational Optimization},
	author = {Papageorgiou, Dimitri J. and Trespalacios, Francisco},
	month = mar,
	year = {2018},
	pages = {55--83},
}

@article{ruizHierarchyRelaxationsNonlinear2012,
	title = {A hierarchy of relaxations for nonlinear convex generalized disjunctive programming},
	volume = {218},
	issn = {0377-2217},
	url = {https://www.sciencedirect.com/science/article/pii/S037722171100899X},
	doi = {10.1016/j.ejor.2011.10.002},
	number = {1},
	urldate = {2025-07-13},
	journal = {European Journal of Operational Research},
	author = {Ruiz, Juan P. and Grossmann, Ignacio E.},
	month = apr,
	year = {2012},
	pages = {38--47},
}

@article{kronqvist50YearsMixedinteger2025,
	title = {50 years of mixed-integer nonlinear and disjunctive programming},
	copyright = {https://www.elsevier.com/tdm/userlicense/1.0/},
	issn = {0377-2217},
	url = {https://linkinghub.elsevier.com/retrieve/pii/S0377221725005417},
	doi = {10.1016/j.ejor.2025.07.016},
	language = {en},
	urldate = {2025-07-24},
	journal = {European Journal of Operational Research},
	author = {Kronqvist, Jan and Bernal Neira, David E. and Grossmann, Ignacio E.},
	month = jul,
	year = {2025},
	note = {Publisher: Elsevier BV},
}

@article{stubbsBranchandcutMethod011999,
	title = {A branch-and-cut method for 0-1 mixed convex programming},
	volume = {86},
	issn = {1436-4646},
	url = {https://doi.org/10.1007/s101070050103},
	doi = {10.1007/s101070050103},
	language = {en},
	number = {3},
	urldate = {2025-07-29},
	journal = {Mathematical Programming},
	author = {Stubbs, Robert A. and Mehrotra, Sanjay},
	month = dec,
	year = {1999},
	pages = {515--532},
}

@article{leeNewAlgorithmsNonlinear2000,
	title = {New algorithms for nonlinear generalized disjunctive programming},
	volume = {24},
	issn = {0098-1354},
	url = {https://www.sciencedirect.com/science/article/pii/S0098135400005810},
	doi = {10.1016/S0098-1354(00)00581-0},
	number = {9},
	urldate = {2025-08-06},
	journal = {Computers \& Chemical Engineering},
	author = {Lee, Sangbum and Grossmann, Ignacio E.},
	month = oct,
	year = {2000},
	pages = {2125--2141},
}

@techreport{BolusaniEtal2024OO,
  author = {Suresh Bolusani and Mathieu Besan{\c{c}}on and Ksenia Bestuzheva and Antonia Chmiela and Jo{\~{a}}o Dion{\'{i}}sio and Tim Donkiewicz and Jasper van Doornmalen and Leon Eifler and Mohammed Ghannam and Ambros Gleixner and Christoph Graczyk and Katrin Halbig and Ivo Hedtke and Alexander Hoen and Christopher Hojny and Rolf van der Hulst and Dominik Kamp and Thorsten Koch and Kevin Kofler and Jurgen Lentz and Julian Manns and Gioni Mexi and Erik~M\"{u}hmer and Marc E. Pfetsch and Franziska Schl{\"o}sser and Felipe Serrano and Yuji Shinano and Mark Turner and Stefan Vigerske and Dieter Weninger and Lixing Xu},
  title = {{The SCIP Optimization Suite 9.0}},
  type = {Technical Report},
  institution = {Optimization Online},
  month = {February},
  year = {2024},
  url = {https://optimization-online.org/2024/02/the-scip-optimization-suite-9-0/}
}

@article{chenPyomoGDPEcosystemLogic2022,
	title = {Pyomo.{GDP}: an ecosystem for logic based modeling and optimization development},
	volume = {23},
	issn = {1573-2924},
	shorttitle = {Pyomo.{GDP}},
	url = {https://doi.org/10.1007/s11081-021-09601-7},
	doi = {10.1007/s11081-021-09601-7},
	language = {en},
	number = {1},
	urldate = {2025-08-22},
	journal = {Optimization and Engineering},
	author = {Chen, Qi and Johnson, Emma S. and Bernal, David E. and Valentin, Romeo and Kale, Sunjeev and Bates, Johnny and Siirola, John D. and Grossmann, Ignacio E.},
	month = mar,
	year = {2022},
	pages = {607--642},
}

@article{trespalaciosCuttingPlanesImproved2016,
	title = {Cutting planes for improved global logic-based outer-approximation for the synthesis of process networks},
	volume = {90},
	issn = {0098-1354},
	url = {https://www.sciencedirect.com/science/article/pii/S0098135416301132},
	doi = {10.1016/j.compchemeng.2016.04.017},
	urldate = {2026-03-05},
	journal = {Computers \& Chemical Engineering},
	author = {Trespalacios, Francisco and Grossmann, Ignacio E.},
	month = jul,
	year = {2016},
	pages = {201--221},
}
